\documentclass[smallextended]{svjour3}       
\smartqed  
\usepackage{graphicx}
\usepackage{fullpage}								
\usepackage{amsmath,bm,mathrsfs,mathtools,amssymb}        
\usepackage{amsfonts}              
\usepackage{graphicx,subfigure,xcolor}       
\usepackage{caption}     
 \usepackage{url}

\usepackage{epstopdf}
\usepackage{tabularx}
\usepackage[section]{placeins}

\usepackage{enumerate} 			         
\usepackage{bbm}             		         
\usepackage[bookmarks=true,         bookmarksnumbered=true,
colorlinks=true,
pdfstartview=FitV,
linkcolor=blue, citecolor=blue, urlcolor=blue,
]{hyperref}

\usepackage{algorithm}
\usepackage[noend]{algpseudocode}
\usepackage[normalem]{ulem}
 \usepackage{booktabs}
\usepackage[section]{placeins}
\usepackage[english]{babel}
\usepackage{times,verbatim}

\numberwithin{theorem}{section}
\numberwithin{proposition}{section}
\numberwithin{lemma}{section}
\numberwithin{definition}{section}

\numberwithin{equation}{section}
\newtheorem{assumption}[theorem]{Assumption}

\newcommand{\rhoT}{\rho_T}

\newcommand{\rhsF}{\mbf{f}}

\newcommand{\vertiii}[1]{{\left\vert\kern-0.25ex\left\vert\kern-0.25ex\left\vert #1 
    \right\vert\kern-0.25ex\right\vert\kern-0.25ex\right\vert}}
\newcommand{\verti}[1]{{\left\vert #1\right\vert}}

\newcommand{\rmref}[1]{{\rm\ref{#1}}}

\newcommand{\mbf}[1]{\boldsymbol{#1}}

\newcommand{\innerp}[2]{\langle #1,#2 \rangle}

\newcommand{\dbinnerp}[1]{\langle\hspace{-1mm}\langle{#1}\rangle\hspace{-1mm}\rangle}

\newcommand{\norm}[1]{\left\| #1 \right\|}

\newcommand{\realR}[1]{\mathbb{R}^{#1}}
\newcommand{\real}{\mathbb{R}}

\newcommand{\br}{\mbf{r}}

\newcommand{\bx}{\mbf{x}}
\newcommand{\bB}{\mbf{B}}

\newcommand{\bX}{\mbf{X}}
\newcommand{\bY}{\mbf{Y}} 
\newcommand{\ba}{\mbf{a}}

\newcommand{\R}{\real}

\newcommand{\intkernel}{\phi}

\newcommand{\intkernelvar}{\varphi}

\newcommand{\rhsfo}{\mathbf{f}}
\newcommand{\hypspace}{\mathcal{H}}
\newcommand{\E}{\mathbb{E}}
\newcommand{\Prob}{\mathbb{P}}

\newcommand{\ProbIC}{\mu_0}

\newcommand{\argmin}[1]{\underset{#1}{\operatorname{arg}\operatorname{min}}\;}

\newcommand{\dimsp}{\mathrm{dim}}

\DeclareMathAlphabet{\mathpzc}{OT1}{pzc}{m}{it}
\DeclareMathOperator*{\esssup}{ess\,sup}

\usepackage{float}
\newfloat{algorithm}{t}{lop}

\begin{document}

\title{Learning interaction kernels in stochastic systems of interacting particles from multiple trajectories
\thanks{FL and MM are grateful for partial support from NSF-1913243, FL for NSF-1821211; MM for NSF-1837991, NSF-1546392, AFOSR-FA9550-17-1-0280 and the Simons Fellowship; ST for NSF-1821211, AFOSR-FA9550-17-1-0280 and AMS Simons travel grant.}
}

\titlerunning{Learning interaction kernels in stochastic interacting particle systems}        

\author{	Fei Lu         \and
        		Mauro Maggioni \and
        		Sui Tang 
}


\institute{Fei Lu \at
              Department of Mathematics, Johns Hopkins University,
              \email{feilu@math.jhu.edu}           
           \and
           Mauro Maggioni \at
              Department of Mathematics,
              Department of Applied Mathematics and Statistics,
              Johns Hopkins University,            
              \email{mauromaggionijhu@icloud.com}           
              \and
	  Sui Tang \at
              Department of Mathematics,
              University of California, Santa Barbara,
              \email{suitang@math.ucsb.edu}           
}

\date{Received: date / Accepted: date}

\maketitle

\begin{abstract}
We consider stochastic systems of interacting particles or agents, with dynamics determined by an interaction kernel which only depends on pairwise distances. We study the problem of inferring this interaction kernel from observations of the positions of the particles, in either continuous or discrete time, along multiple independent trajectories. We introduce a nonparametric inference approach to this inverse problem, based on a regularized maximum likelihood estimator constrained to suitable hypothesis spaces adaptive to data. We show that a coercivity condition enables us to control the condition number of this problem and prove the consistency of our estimator, and that in fact it converges at a near-optimal learning rate, equal to the min-max rate of $1$-dimensional non-parametric regression. In particular, this rate is independent of the dimension of the state space, which is typically very high. We also analyze the discretization errors in the case of discrete-time observations, showing that it is of order $1/2$ in terms of the time gaps between observations. This term, when large, dominates the sampling error and the approximation error, preventing convergence of the estimator. Finally, we exhibit an efficient parallel algorithm to construct the estimator from data, and we demonstrate the effectiveness of our algorithm with numerical tests on prototype systems including stochastic opinion dynamics and a Lennard-Jones model. 

 \end{abstract}

\date{}


\tableofcontents

\section{Introduction}
We consider a system of particles or agents interacting in a random environment, with their motion described by a first-order Stochastic Differential Equation in the form
\begin{equation}\label{eq:sod}
  d{\bx_{i,t}} =\frac{1}{N} \sum_{i'= 1}^N \intkernel(\| \bx_{j,t} - \bx_{i,t}\|)(\bx_{j,t} - \bx_{i,t} )dt+ \sigma d\bB_{i,t}\,, \quad \text{for $i = 1, \ldots, N$},
\end{equation}
where $\bx_{i,t} \in \realR{d}$ represents the position of particle $i$ at time $t$,  $\intkernel:\R^+\to \R$ is an {\em{interaction kernel}} dependent on the pairwise distance between particles, and $\bB_t$ is a standard Brownian motion in $\mathbb{R}^{Nd}$, with $\sigma>0$ representing the scale of the random noise.
This is a gradient system, with the energy potential $V_{\intkernel}:\R^{Nd} \to \R$
\begin{align}\label{potential}
V_{\intkernel}(\bX_t)=\frac{1}{2N}\sum_{i,i'}\Phi(\|\bx_{i,t}-\bx_{i',t}\|)\quad \text{ with }\ \ \  \Phi'(r)=\intkernel(r)r\,,
\end{align}
where $\bX_t=(\bx_{i,t})_{i=1,\dots,N}\in\mathbb{R}^{dN}$ is the state of the system.
Letting
\begin{equation}\label{eq:rhsF}
\rhsF_{\intkernel}:= -\nabla V_{\intkernel}\,,
\end{equation}
we can write Eq.\eqref{eq:sod} in vector format as 
\begin{align} \label{eq:sys_grad}
d\bX_t&=\rhsF_{\intkernel}(\bX_t)dt+\sigma d\bB_t\,.
\end{align}
The particles interact with each other based on their pairwise distance, with dissipation of the total energy, with the system tending to a stable point of the energy potential, while the random noise injects energy to the system. 

Such systems of interacting particles arise in a wide variety of disciplines, including interacting physical particles \cite{skorokhod1996_RegularityManyparticle,DOCBC2006} or granular media \cite{bell2005_ParticlebasedSimulation,baumgarten2019_GeneralConstitutive,benachour1998_NonlinearSelfstabilizing,bolley2013_UniformConvergence,carrillo2003_KineticEquilibration,cattiaux2007_ProbabilisticApproach} in Physics,  opinion aggregation on interacting networks in Social Science  \cite{hegselmann2002_OpinionDynamics,olfati2004consensus,MT2014}, and Monte Carlo sampling \cite{liu2019_SteinVariational,li2020_StochasticVersion}, to name just a few.

Motivated by these applications, the inference of such systems from data gains increasing attention. For deterministic multi-particles systems, various types of learning techniques have been developed (see e.g. \cite{BFHM17,LZTM19,LMT19,MMM19,almi2019datadriven,chen2019inferring} and the reference therein). When it comes to stochastic multi-particle systems, only a few efforts have been made, e.g. learning reduced Langevin equations on manifolds in \cite{CM:ATLAS} (without however assuming nor exploiting the structure of pairwise interactions), learning the parametric potential function in \cite{brillinger2012learning,chen2020maximum} from single trajectory data, and estimating the diffusion parameter in \cite{CM:ATLAS,huang2018_LearningInteracting}.

\vskip0.25cm
Our goal is to estimate the interaction kernel $\intkernel$ given discrete-time observation data from trajectories $\{\bX^{(m)}_{t_0:t_L}\}_{m=1}^M$, where the initial conditions $\{\bX^{(m)}_{t_0}\}_{m=1}^M$ are independent samples drawn from a distribution $\mu_0$ on $\mathbb{R}^{dN}$, and $t_0:t_L$ indicates times $0=t_0<t_1<\dots<t_l<\dots<t_L=T$, with with $t_l = l\Delta t$. 
\vskip0.25cm

Since in general, little information about the analytical form of the kernel is available, we infer it in a nonparametric fashion  (e.g. \cite{CS02,binev2005universal,Gyorfi06}). We note that the problem we consider is to learn a latent function in the drift term given observations from multiple trajectories, which is different from the ample literature on the inference of stochastic differential equations (see e.g. \cite{kutoyants_StatisticalInference2004,kaipio_StatisticalComputational2005}), focusing either on parameter estimation or inference for ergodic system. In particular, our learning approach is close in spirit to the nonparametric regression of the drift studied in \cite{nickl2019_NonparametricStatistical} for ergodic system and in \cite{comte_NonparametricDrift} from i.i.d paths. However, for systems of interacting particles one faces the curse of dimensionality when learning the high-dimensional drift directly as a general function on the high-dimensional state space $\mathbb{R}^{dN}$. Instead, we will exploit the structure of the system and learn the latent interaction kernel in the drift, which only depends on pairwise distances, and show that the curse of dimensionality may be avoided, when such inverse problem is well-conditioned.

We introduce a maximum likelihood estimator (MLE), along with an efficient algorithm that can be implemented in parallel over trajectories, with an hypothesis space adaptive to data to reach optimal accuracy. Under a coercivity condition, we prove that the MLE is consistent, and converges at the min-max rate for one-dimensional nonparametric regression. We also analyze the discretization errors due to discrete-time observations: we show it leads to an error in the estimator that is of order $\Delta t^{1/2}$ (with $\Delta t=T/L=t_{l+1}-t_l$), and as a result, it prevents us from obtaining the min-max learning rate in sample size. We demonstrate the effectiveness of our algorithm by numerical tests on prototype systems including opinion dynamics and a stochastic Lennard-Jones model (see Section \ref{main:numericalexamples}). Numerical results verify our learning theory in the sense that that the min-max rate of convergence is achieved, and the bias due to the numerical error is close to the order $\Delta t^{1/2}$.

\subsection{Overview of the main results}
We consider an approximate maximum likelihood estimator (MLE), 

which is the maximizer of the approximate likelihood of the observed trajectories, over a suitable hypothesis space $\hypspace$:
\begin{equation*}
\widehat \intkernel_{L,T,M,\hypspace}= \argmin{\intkernelvar \in \mathcal{H}}\mathcal{E}_{L,T,M}(\intkernelvar),
\end{equation*}
where $\mathcal{E}_{L,T,M}(\intkernelvar)$ is an approximation of the negative log-likelihood of the discrete data $\{\bX^{(m)}_{t_0:t_L}\}_{m=1}^M$. Using the fact that the drift term $\rhsF_{\intkernel}$ is linear in $\phi$ and hence $\mathcal{E}_{L,T,M}(\intkernelvar)$ is a quadratic functional, we propose an algorithm (see Algorithm \ref{alg:main}) that efficiently computes this MLE by least squares. With a data-adaptive choice of the basis functions $\{\psi_p\}_{p=1}^n$ for the hypothesis space $\hypspace$, we obtain the MLE 
\begin{equation}\label{eq:a_in_intro}
\widehat \intkernel_{L,T,M,\hypspace} = \sum_{p=1}^n \widehat a_{L,T,M, \hypspace}(p) \psi_p
\end{equation}
by computing the coefficients $\widehat a_{L,T,M, \hypspace}\in \R^n$ from normal equations. The  algorithm may be implemented by building in parallel the equations for each trajectory. 

We develop a systematic learning theory on the performance of this MLE. We propose first a coercivity condition that ensures the robust identifiability of the kernel $\intkernel$, in the sense that the derivative of the pairwise potential defined in \eqref{potential}, $\Phi'(r) = \intkernel(r)r$, can be uniquely identified in the function space $L^2(\R^+,\rho_T)$, where $\rho_T$ is the measure of all pairwise distances between particles. Then, we consider the convergence of the estimator, from both continuous-time and discrete-time observations, under the norm 
\begin{equation}
\vertiii{\varphi}:= \|\varphi(\cdot)\cdot\|_{L^2(\rhoT)}= \left( \int_{\R^+} |\varphi(r)r|^2\rhoT(dr) \right)^{1/2}\,. 
\label{e:triplenorm}
\end{equation} 

\noindent\textbf{The case of continuous-time observations} (Section \ref{sec:theory_cts_traj}). We consider the MLE
\begin{equation*}
\widehat \intkernel_{T,M,\hypspace}= \argmin{\intkernelvar \in \mathcal{H}}\mathcal{E}_{T,M}(\intkernelvar),
\end{equation*}
where $\mathcal{E}_{T,M}(\intkernelvar)$ is the exact negative log-likelihood of the continuous-time trajectories $\{\bX^{(m)}_{[0,T]}\}_{m=1}^M$. 
We show that the MLE is consistent, that is, converges in probability to the true kernel under the norm $\vertiii{\cdot}$. Furthermore, we show that the MLE converges at a rate which is independent of the dimension of the state space of the system, and corresponds to the minmax rate for one-dimensional non-parametric regression (\cite{CS02,cohen2013stability,binev2005universal,Gyorfi06}), when choosing the hypothesis space adaptively according to data, the H\"older continuity $s$ of the true kernel, and with dimension increasing with the amount of observed data.  With $\mathrm{dim}(\hypspace_{n}) \asymp (\frac{M}{\log M})^{\frac{1}{2s+1}}$, and assuming that the coercivity condition holds on $\hypspace_{n}$  with a constant $c_{\hypspace_{n}}>0$,  we have, with high probability and in expectation,
$$\vertiii {\widehat \intkernel_{T,M,\mathcal{H}_{n}}-\intkernel}^2 \lesssim \frac1{c^2_{\hypspace_{n}}} \left(\frac{\log M}{ M}\right)^{\frac{2s}{2s+1}} $$

\noindent\textbf{The case of discrete-time observations } (Section \ref{sec:theory_dis_traj}). In this case derivatives and statistics of the trajectories in-between observations need to be approximated, while keeping the estimator efficiently computable: this leads to further approximations of the likelihood, and consequently of the MLE. This discretization error of the approximations we use is of order $1/2$ in the observation time gap $\Delta t = T/L$, using an approximation of the likelihood based on the Euler-Maruyama integration scheme. We show that for some $C>0$, for any $\epsilon>0$, with high probability
\begin{align*}
\vertiii{\widehat \intkernel_{L,T,M,\hypspace} - \intkernel}  \leq  \vertiii{\widehat\intkernel_{T, \infty, \hypspace} - \intkernel}  + C\left(\sqrt\frac nM\epsilon+\Delta t ^{\frac{1}{2}}\right)\,,
\end{align*}
where $\widehat\intkernel_{T, \infty, \hypspace}  $ is the projection of the true kernel to $\hypspace$. The discretization error will flatten the learning curve when the sample size is large, overshadowing the sampling error and the approximation error cause by working within the hypothesis space.      
for some positive constants $c_2 $ and $c_3$, where $\widehat\intkernel_{T, \infty, \hypspace}  $ is the projection of the true kernel to $\hypspace$. The numerical error may overshadow the sampling error and the approximation error of the hypothesis space.

\begin{figure}[tbp]
\centering   
\includegraphics[width=0.8\textwidth]{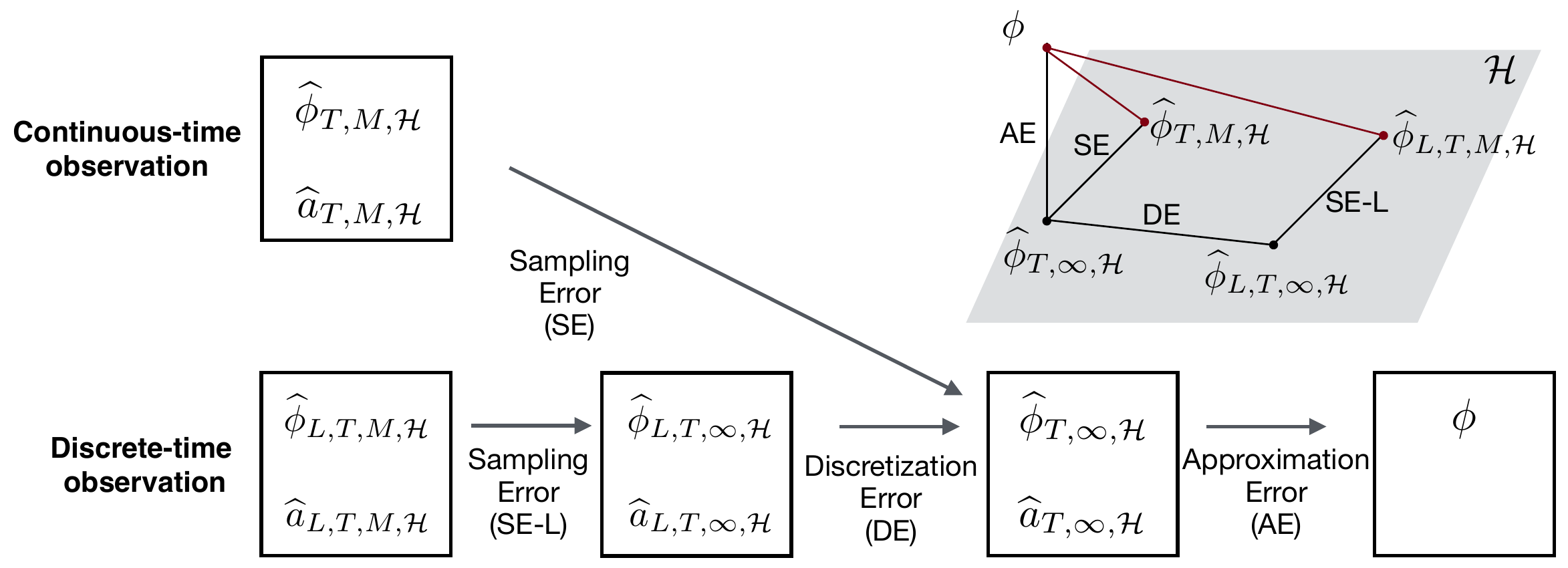}
\caption{Diagram of the (regularized) MLE error analysis and convergence. For continuous-time observations, we refer to Proposition \ref{covering} for analysis of SE: $\hat \phi_{T,M,\mathcal{H}}-\hat \phi_{T,\infty,\mathcal{H}}$ and Theorem \ref{maintheorem} for bounding the total estimation error $\hat \phi_{T,M,\mathcal{H}}-\phi$.  For discrete-time observations, we refer to Proposition \ref{prop:CI} for SE-L: $\hat \phi_{L,T,M,\mathcal{H}}-\hat \phi_{L,T,\infty,\mathcal{H}}$, Proposition \ref{prop:numError} for DE: $\hat \phi_{L,T,\infty,\mathcal{H}}-\hat \phi_{T,\infty,\mathcal{H}}$, and Theorem \ref{thm:error_discreteTime} for $ \hat \phi_{L,T,M,\mathcal{H}}-\phi$. 
}
\label{f:diagram}
\end{figure}
In both cases, we decompose the error in the MLE into sampling error from the trajectory data, and approximation error from the hypothesis space, as illustrated in the diagram in Figure \ref{f:diagram}. 
 In the case of continuous-time observations, the sampling error is the error between $\widehat \intkernel_{T,M,\hypspace} $  and  the MLE from infinitely many trajectories (denoted by $\widehat \intkernel_{T,\infty,\hypspace}$): this will be controlled with concentration equalities. The approximation error $\widehat \intkernel_{T,\infty,\hypspace} -\phi$ is adaptively controlled by a proper choice of hypothesis space. The analysis is carried out in the infinite-dimensional space $L^2(\rho_T)$. 
In the case of discrete-time observations, we provide a finite-dimensional analysis to study directly the MLE in our proposed algorithm, that is, analyzing the error of $\widehat a_{L,T,M, \hypspace}$ in \eqref{eq:a_in_intro}  with proper conditions on the basis functions. The sampling error $\widehat \intkernel_{L,T,M,\hypspace} - \widehat \intkernel_{L,T,\infty,\hypspace} $ is analyzed through $\widehat a_{L,T,M, \hypspace} -\widehat a_{L,T,\infty, \hypspace}$, and the discretization error between $\widehat \intkernel_{L,T,\infty,\hypspace} $ and $\widehat \intkernel_{T,\infty,\hypspace} $ is analyzed through $\widehat a_{L,T,M, \hypspace} -\widehat a_{T,\infty, \hypspace}$. The discretization error comes from the discrete-time approximation of the likelihood, and it vanishes when the observation time gap $\Delta t$ reduces to zero, recovering the convergence of the MLE as in the case of the continuous-time observations.

\subsection{Notation and outline} 

Throughout this paper, we use bold letters to denote vectors or vector-valued functions.  We use the notation in Table \ref{tab:1stOrder_def} for variables in the system of interacting particles. 
\begin{table}[t]
\centering
\small{
\small{\begin{tabular}{ c | l  }
\hline
Variable                    & Definition \\
\hline\hline
$\bx_{i,t}\in \R^d$ & position or opinion of particle $i$ at time $t$, see \eqref{eq:sod} \\
\hline
$\bX_t=(\bx_{1,t},\ldots, \bx_{N,t}) \in \R^{dN}$ & state vector: position of the $N$ particles, see \eqref{eq:sys_grad} \\
\hline
$\|\cdot\|$ & Euclidean norm in $\mathbb{R}^d$ or operator norm of a matrix\\
\hline
$\br_{ii'}(t), \br_{ii''}(t)\in \R^d$ & $\bx_{i',t}-\bx_{i,t}$ and $\bx_{i'',t}-\bx_{i,t}$, \\
\hline
$r_{ii'}(t), r_{ii''}(t)\in \R^+$ & $r_{ii'}(t)=\|\br_{ii'}(t) \|$ and $r_{ii''}(t)=\|\br_{ii'}(t)\|$,  \\
\hline
$\intkernel$  &  interaction kernel, see \eqref{eq:sod} \\
\hline
$\rhsfo_\intkernel$  & drift function of the system, see \eqref{eq:sys_grad}\\
\hline
$V_{\intkernel}=\frac{1}{2N}\sum_{i,i'}\Phi(\|\bx_{i,t}-\bx_{i',t}\|)$  & energy potential with $ \Phi'(r)=\intkernel(r)r$, see \eqref{potential} \\
\hline
\end{tabular}}  
}
\caption{Notation for the system of interacting particles driven by equation \eqref{eq:sod}.}
\label{tab:1stOrder_def} 
\end{table}

We restrict our attention to interaction kernels $\phi$ in the admissible set
\begin{align}\label{def_AddmissibleSet}
\mathcal{K}_{R,S}:=\{\varphi \in C^1(\mathbb{R}_+) :  \text{Supp}(\varphi) \subset [0,R], \|\varphi\|_{\infty}+\|\varphi'\|_{\infty} \leq S\}.
\end{align}
Let $\Omega$ be an arbitrary compact (or precompact) set of a Euclidean space (which may be $\R^{+}$, $\R^{d}$ or $\R^{dN}$), with the Lebesgue measure unless otherwise specified. We consider the following function spaces 
\begin{itemize}
\item $L^{\infty}(\Omega)$: the space of bounded functions on $\Omega$ with norm 
$\|g\|_{\infty}=\esssup_{x\in \Omega}|g(x)|$; 
\item $C(\Omega):$ the closed subspace of $L^{\infty}(\Omega)$ consisting of  continuous functions;
\item $C_c(\Omega):$ the set of functions in $C(\Omega)$ with compact support;

\item $C^{k,\alpha}(\Omega)$ with $k \in \mathbb{N}, 0<\alpha \leq 1$: the space of functions whose $k$-th derivative is  H\"older continuous of order $\alpha$.  In the special case of $k=0$ and $\alpha=1$, $g \in  C^{0,\alpha}(\Omega)$ is called  Lipchitz continuous  on $\Omega$; the Lipschitz constant of $g \in \text{Lip}(\Omega)$ is defined as
$\text{Lip}[g]:=\sup_{x\neq y} \frac{|g(x)-g(y)|}{\|x-y\|}.$

\end{itemize}

We summarize the notation for the inference of the interaction kernel in Table \ref{tab:1stOrderlearning_def}. 
\begin{table}[t]
\centering
\small{ 
\small{\begin{tabular}{ c | l }
\hline
Notation                    & Definition \\
\hline\hline
$M$   & number of observed trajectories \\
\hline
$t_0:t_L=\{t_l\}_{l=1}^{L}$ &  observation times in $[0,T]$, $0=t_0<\dots<t_L=T$, $t_l=l\Delta t=l T/L$ \\
\hline 
$\mu_0$ & probability distribution in $\mathbb{R}^{dN}$ for initial configurations $\bX_0$\\
\hline
$\hypspace$ and $\{\psi_{p}\}_{p=1}^{n}$  &  the hypothesis space of learning, and a basis for it \\
\hline
$\mathcal{E}_{T,M}(\cdot)$ and $\mathcal{E}_{L,T,M}(\cdot)$ &  empirical error functionals from continuous/discrete data, see \eqref{e:ETM} and \eqref{eq:errFunl} \\
\hline  $\widehat \intkernel_{T, M,\hypspace}$ and $\widehat \intkernel_{L, T, M,\hypspace}$ & minimizers, over $\hypspace$, of $\mathcal{E}_{T,M}(\intkernelvar)$ and $\mathcal{E}_{L,T,M}(\intkernelvar)$, see \eqref{MLE_cts} and \eqref{MLE} \\
\hline  
$\widehat a_{L, T, M,\hypspace}$  & coefficient vectors of  $\widehat \intkernel_{L, T, M,\hypspace}$ w.r.t. basis $\{\psi_{p}\}_{p=1}^{n}$, see \eqref{eq:phiEst} \\
\hline  $\vertiii{\cdot}$ & $L^2(\rho_T)$-based norm: $\vertiii{\phi} = \|\phi(\cdot)\cdot\|_{L^2(\rho_T)}$, see \eqref{e:triplenorm} \\
\hline\hline
\end{tabular}}  
}
\caption{Notation used for the estimator of the interaction kernel $\intkernel$.}
\label{tab:1stOrderlearning_def} 
\end{table}
The function space in which we perform the estimation is the space of functions $\varphi$ such that $\varphi(\cdot)\cdot \in L^2(\R^+, \rho_T)$, where $\rho_T$ is the measure of pairwise distances between all particles on the time interval $[0,T]$ (see \eqref{rhoT}). 
We will focus on learning on the compact (finite- or infinite-dimensional) subset of $L^{\infty}([0,R])$ (where $[0,R]$ is the support of the functions in the admissible set $\mathcal{K}_{R,S}$) in the theoretical analysis, however in the numerical implementation we will use finite-dimensional linear subspaces $L^2([0,R], \rho_T)$ spanned by piecewise polynomials functions. While these linear subspaces are not compact, it is shown that the minimizers over the whole linear space are bounded and thus the compactness requirements are not essential (e.g., see Theorem 11.3 in \cite{Gyorfi06}). We shall therefore assume the compactness of the hypothesis space in the theoretical analysis.

The remainder of the paper is organized as follows. We first provide an overview of our learning theory. In Section \ref{sec:algorithm}, we present a practical learning algorithm with theory-guided optimal settings on the choice of hypothesis spaces and with a practical assessment of the learning results. We then demonstrate the effectiveness of the algorithm on prototype systems including a stochastic model for opinion dynamics, and a stochastic Lennard-Jones model in Section \ref{main:numericalexamples}.  We establish a systematic learning theory analyzing the performance of the MLE, considering continuous-time observations in Section \ref{sec:theory_cts_traj} and discrete-time observations in Section \ref{sec:theory_dis_traj}. We present in the appendix detailed proofs.

\section{Nonparametric inference of the interaction kernel}\label{sec:algorithm}
We present in this section the nonparametric technique we study for the inference of the interaction kernel, and corresponding algorithms. We discuss assessment of the performance of the estimator and its performance in trajectory prediction.  The proposed estimator is based on maximum likelihood estimation on data-adaptive hypothesis spaces so as to achieve optimal rate of convergence, guided by our learning theory in Section \ref{sec:theory_cts_traj} -\ref{sec:theory_dis_traj}.  
 
\subsection{The maximum likelihood estimator}
As a variational approach, we set the error functional to be the negative log-likelihood of the data $\{\bX^{(m)}_{t_0:t_L}\}_{m=1}^M$, and compute the maximum likelihood estimator (MLE). 

\paragraph{The error functional.} Recall that by the Girsanov theorem, for a continuous trajectory $\bX_{[0,T]}$, its negative log-likelihood ratio between the measure induced by system \eqref{eq:sod}, with an admissible kernel $\intkernel$, and the Wiener measure is  
\begin{align} \label{lkhd_cts}
\mathcal{E}_{\bX_{[0,T]}}(\intkernel)&=\frac{1}{2\sigma^2TN}\int_{0}^T  \left(\|\rhsF_{\intkernel}(\bX_t)\|^2 - 2 \langle \rhsF_{\intkernel}(\bX_t), d\bX_t \rangle\, dt\right). 
\end{align} 
As we do not know the interaction kernel $\intkernel$ that generated the trajectory $\bX_{[0,T]}$, we can let $\intkernelvar$ be any possible admissible interaction kernel, and upon replacing $\intkernel$ by $\intkernelvar$ in \eqref{lkhd_cts}, observe that $\mathcal{E}_{\bX_{[0,T]}}(\intkernelvar)$ is the log-likelihood of seeing the trajectory $\bX_{[0,T]}$ if the system \eqref{eq:sod} were driven by the interaction kernel $\intkernelvar$. In this case $\mathcal{E}_{\bX_{[0,T]}}(\intkernelvar)$ may be interpreted as a error functional, which we wish to minimize over $\intkernelvar$, in order to obtain an estimator for $\intkernel$.

Given only discrete-time observations $\bX_{t_0:t_L}$, where $(t_l = l\Delta t, l=0,\dots,L)$ with $\Delta t = T/L$ (the case of non-equispaced-in-time observation is a straightforward generalization), the error functional $\mathcal{E}_{\bX_{[0,T]}}(\intkernelvar)$ may be approximated as
\begin{align} \label{lkhd_discrete}
\mathcal{E}_{ \bX_{t_1:t_L}}(\intkernelvar)  &:= \frac{ 1}{2\sigma^2 TN} \sum_{l=0}^{L-1}  \left( \|\rhsF_{\intkernelvar}(\bX_{t_l})\|^2  \Delta t-2 \langle \rhsF_{\intkernelvar}(\bX_{t_l}), \bX_{t_l}- \bX_{t_{l-1}} \rangle  \right). 
\end{align} 
 The corresponding approximate likelihood is equivalent to the likelihood based on the Euler-Maruyama (EM) scheme (whose transition probability density is Gaussian): 
\begin{equation}\label{EM}
\bX_{t_{l+1}} = \bX_{t_{l}} + \rhsF_\intkernelvar(\bX_{t_{l}}) \Delta t + \sigma \sqrt{\Delta t} \mathbf{W}_l\,,\quad \mathbf{W}_l\sim \mathcal{N}(0, I_{Nd\times Nd})
\end{equation}
Note that while higher-order approximations of the stochastic integral (or, equivalently, approximations based on higher order numerical schemes) may be more accurate than the EM scheme, they lead to nonlinear optimization problems in the computation of the MLE defined below, and we shall therefore avoid them. The EM-based approximation preserves the quadratic form of the error functional, and leads to an optimization problem that can may be solved  by least squares. As we show in Theorem \ref{thm:error_discreteTime}, this discrete-time approximation leads to an error term of order $\Delta t^{1/2}$ in the MLE, which will be small in the regime on which we focus in this work.

Since the observed discrete-time trajectories $\{\bX^{(m)}_{t_0:t_L}\}_{m=1}^M$ are independent, since $\bX_0$ is drawn i.i.d. from $\mu_0$, the joint likelihood of the trajectories is the product of the likelihood of each trajectory. Therefore, the corresponding \emph{empirical error functional} is defined to be
\begin{align}\label{eq:errFunl}
\mathcal{E}_{L,T,M}(\intkernelvar):=\frac{1}{M}\sum_{m=1}^{M}\mathcal{E}_{ \bX_{t_1:t_L}^{(m)}}(\intkernelvar). 
\end{align} 

\paragraph{A regularized Maximum Likelihood Estimator.} The regularized MLE we consider is a minimizer of the above empirical error functional over a suitable hypothesis space $\mathcal{H}$: 
\begin{equation}\label{MLE}
\widehat \intkernel_{L,T,M,\hypspace}= \argmin{\intkernelvar \in \mathcal{H}}\mathcal{E}_{L,T,M}(\intkernelvar),
\end{equation}
This regularized MLE is well-defined when the minimizer exists and is unique over $\hypspace$. We shall discuss the uniqueness of the minimizer in Section \ref{sec:coercivty}, where we show it is guaranteed by a coercivity condition.  
When the hypothesis space $\hypspace$ is a finite dimensional linear space, say, $\hypspace=\mathrm{span} \{\psi_i\}_{i=1}^n$ with basis functions $\psi_i: \R^+\to \R$, the regularized MLE is the solution of a least squares problem. To see this, letting $\intkernelvar = \sum_{i=1}^{n} a(i) \psi_i$ and $a :=(a(1),\dots,a(n))\in \mathbb{R}^n$, we have $\rhsF_{\intkernelvar}(\bX) =\sum_{i= 1}^n a(i) \rhsF_{\psi_i}(\bX) $, due to the linear dependence of $\rhsF_\varphi$ on $\varphi$. Then, we can write the error functional in Eq.\eqref{lkhd_discrete} for each trajectory as 
\begin{align*}
\mathcal{E}_{\bX^{(m)}_{t_1:t_L}}(\ba) := \mathcal{E}_{\bX^{(m)}_{t_1:t_L}}(\intkernelvar)  = a^T A^{(m)} a +a^T b^{(m)} ,
\end{align*} 
where the matrix $A^{(m)} \in \R^{n\times n}$ and the vector $b^{(m)}\in \R^{n}$ are given by 
\begin{equation}\label{eq:normal_Ab_1traj} 
\begin{aligned}
A^{(m)}(i,i')   &= \frac{ 1}{2\sigma^2LN} \sum_{l=0}^{L-1} \langle \rhsF_{\psi_i}(\bX_{t_{l}}^{(m)} ), \rhsF_{\psi_{i'}}(\bX_{t_{l}}^{(m)})\rangle, \quad  \\
b^{(m)} (i)    &= - \frac{ 1}{\sigma^2 L\Delta t N} \sum_{l=0}^{L-1} \langle \rhsF_{\psi_i}(\bX_{t_l}^{(m)}), \bX_{t_{l+1}}^{(m)}- \bX_{t_{l}}^{(m)} \rangle.
\end{aligned}
\end{equation} 
Hence,  corresponding to $\nabla \mathcal{E}_{L,T,M}=0$ for the error functional in \eqref{eq:errFunl}, we solve the normal equations for $a$ to obtain the solution $\widehat a_{L,T,M,\hypspace}$:
\begin{equation}\label{e:ALM}
 A_{M,L} \widehat a_{L,T,M,\hypspace} =b_{M,L}, \text{ where }\, A_{M,L}:=\frac{1}{M} \sum_{m=1}^{M}  A^{(m)},\, b_{M,L}:=\frac{1}{M} \sum_{m=1}^{M}  b^{(m)} 
\end{equation}
and corresponding desired MLE for the interaction kernel:
\begin{align}\label{eq:phiEst}
\widehat \intkernel_{L,T,M,\hypspace} = \sum_{i=1}^{n}  \widehat a_{L,T,M,\hypspace}(i) \psi_i.
\end{align}
The normal equations \eqref{e:ALM} are solved by least squares, so the solution always exists. We will show in Section \ref{sec:theory_dis_traj} that assuming a coercivity condition, the matrix $A_{M,L}\in\mathbb{R}^{n\times n}$ is invertible with high probability when $M$ and $L$ are large, so the least squares estimator is the unique solution to the normal equations, and the regularized MLE is the unique minimizer of the empirical error functional over $\hypspace$.

\subsection{Dynamics-adapted measures and function spaces}\label{sec:hypspace}

We will assess the estimation error in a suitable function space: $L^2(\R^+,\rho_T)$.  Here $\rho_T$ is the distribution of pairwise distances between all particles: 
\begin{align}\label{rhoT}
\rho_T (dr) &:= \frac{1}{\binom N2 T}\int_{t = 0}^T\bigg[\sum_{i,i'=1, i< i' }^N \E[\delta_{r_{ii'}(t)}(dr)] \, dt\bigg],
\end{align}%
where $\delta$ is the Dirac $\delta$ distribution, so that $\E[\delta_{r_{ii'}(t)}(dr)]$ is the distribution of the random variable $r_{ii'}(t)= || \bx_{i,t} -\bx_{i',t} ||$, with $\bx_{i,t}$ being the position of particle $i$ at time $t$. 
  
The probability measure $\rho_T$ depends on both the distribution of initial conditions $\mu_0$ and the measure determining the random noise on the path space, while it is independent of the observed data. The measure $\rho_T$ encodes the information about the dynamics marginalized to pairwise distances; regions with large $\rho_T$-measure correspond to pairwise distances between particles that are often encountered during the dynamics.

With observations of $M$ trajectories at $L$ discrete-times each, we introduce a corresponding measure
\begin{align} \label{e:rhoLM}
\rho^{L,M}_T(dr) &:= \frac{1}{\binom N2 L M}\sum_{l = 0, m=1}^{L - 1, M} \bigg[\sum_{i,i'=1, i< i' }^N\delta_{r^{(m)}_{ii'}(t_l)}(dr) \bigg]\,, 
\end{align}
where $r^{(m)}_{ii'}(t)= || \bx^{(m)}_{i,t} -\bx^{(m)}_{i',t} ||$ is from the $m$-th observed trajectory. 
We think of this as an approximation to $\rho_T$, in two significantly different aspects. In $L$, because as $L\rightarrow+\infty$ our observations tends to be continuous in time, and in $M$, as $\rho^{L,M}_T$ can be thought of, after letting $M\rightarrow+\infty$, as an empirical approximation to $\rho_T$ performed from data on the $M$ independent trajectories.

\paragraph{Accuracy of the estimator.}  
We measure the accuracy of our estimator $\widehat\intkernel_{L,T,M,\hypspace}$ 
by the quantity 
$$\|(\widehat \intkernel_{L,T,M,\hypspace}-\intkernel)(\cdot)\cdot\|_{L^2(\R^+, \rhoT)}.$$ 
The function $\intkernel(\cdot)\cdot$, instead of $\intkernel$, which at $r\in\mathbb{R}_+$ takes value $\intkernel(r)r$, appears naturally in our learning theory in Section \ref{sec:theory_cts_traj}, fundamentally because it is the derivative of the pairwise distance potential $\Phi$ in \eqref{potential}. For simplicity of notation, for a function $\varphi$ in the hypothesis space, we let
\begin{equation}\label{eq:newnorm}
\vertiii{\varphi}:= \|\varphi(\cdot)\cdot\|_{L^2(\rhoT)}= \left( \int_{\R^+} |\varphi(r)r|^2\rhoT(dr) \right)^{1/2}. 
\end{equation} Then the mean square error of the estimator is 
\begin{align}
\E\vertiii{\widehat\intkernel_{L,T,M,\hypspace} -\intkernel}^2.
\end{align}

\subsection{Hypothesis spaces and nonparametric estimators}
\label{s:hypspacesnonparamestimators}
As standard in nonparametric estimation, we let the hypothesis space $\hypspace$ grow in dimension with the number of observations, avoiding under- or over-parametrization, and leading to consistent estimators, that in fact reach an optimal minimax rate of convergence (see e.g. \cite{CS02,Gyorfi06,cucker2007learning}).

Similar to \cite{LZTM19,LMT19}, we set the basis functions  $\{\psi_i\}_{i=1}^n$ to be piecewise polynomials on a  partition of the support of the density function of the empirical probability measure $\rho^{L,M}_T$. 

Guided by the optimal rate convergence results in Section \ref{sec:theory_cts_traj}, we will set the dimension of the hypothesis space $\mathcal{H}$ to be 
\begin{equation}\label{eq:H_dim}
n=C ( M/\log{M})^{1/(2s+1)},
\end{equation}
 where the number $s$ is the H\"older index of continuity of the basis functions, and it is be chosen according the regularity of the true kernel. When $T$ is large and when the system is ergodic, we set 
 \begin{equation*}
n=C ( N_{ess}/\log{N_{ess}})^{1/(2s+1)},
\end{equation*}
where $N_{ess}:= M \frac{T}{\tau}$, with $\tau$ denoting the auto-correlation time of the system, is the effective sample size of the data. Here the auto-correlation time $\tau$ is the equivalent of the mixing time for a reversible ergodic Markov chain \cite{levin2017markov}. 

We estimate the auto-correlation time by the sum of the temporal auto-correlation function of a pairwise distance $r_{i,j}$ (we refer to \cite{thompson2010comparison} for detailed discussion on the estimation of auto-correlation time, which is a whole subject by itself). 

We will prove bounds, that hold with high probability, on the Mean Squared error (MSE) of the MLE $\intkernel_{L,T,M,\mathcal{H}_{n_M}}$ in \eqref{eq:phiEst} for fixed and large $M$, for a fixed time $T$ and for suitable hypothesis spaces $\mathcal{H}_{n_M}$ with dimension $n_M$ as in \eqref{eq:H_dim}.  When continuous-time trajectories are observed, the MSE is of the order $(\frac{\log M}{ M})^{\frac{2s}{2s+1}}$ with high probability, according to Theorem \ref{maintheorem}, and so is its expectation. In particular, this avoids the curse of dimensionality of the state space ($dN$). When the observations are discrete-time trajectories with observation gap $\Delta t$, the error is of the order $(\frac{\log M}{ M})^{\frac{2s}{2s+1}} + \Delta t^{1/2}$ with a high probability according to Theorem \ref{thm:error_discreteTime}.

\subsection{Algorithmic and computational considerations}
\begin{algorithm}
\caption{Learning interaction kernels from many trajectories}\label{alg:main}
\small{
\begin{algorithmic}[1]
\State {\bf Input:} Data consisting of $M$ independent  trajectories $\{\bX^m_{t_0:t_L}\}_{m=1}^M$; H\"older regularity $s$ of the true kernel. 
\State {\bf Output:} An estimator $\widehat \intkernel_{L,T,M,\hypspace_{n_M}} $ for the interaction kernel. 
  \State Compute the pairwise distances and the empirical measure $\rho^{L,M}_T$ in \eqref{e:rhoLM}.  
  \State Construct the basis $\{\psi_{p}\}_{p=1}^{n_M}$ with adaptive partition based on $\rho^{L,M}_T$, and with $n_M$ given by \eqref{eq:H_dim}.
  \State Assemble the normal equations (\ref{e:ALM}) (in parallel).
  \State Solve the normal equation and return $\widehat \intkernel_{L,T,M,\mathcal{H}_{n_M}}$ as in \eqref{eq:phiEst}. 
\end{algorithmic}}
\label{f:algo}
\end{algorithm}
 
The algorithm is summarized in Algorithm \ref{f:algo}. Note that the normal matrices $\{A^{(m)} \}$  and vectors $\{b^{(m)} \}$ are defined trajectory-wise and therefore may be computed in parallel. When the size of the system is large (i.e. $dN$ is large), this allows one to accelerate the computation of the estimator, by assembling these normal matrices and vectors for each trajectory in parallel, and updating the normal matrix $A_{M,L}$ and vector $b_{M,L}$. The total computational cost of constructing our estimator, given $P$ CPU's, is $O(L\frac{N^2d}PMn^2+n^3)$. This becomes $O(L\frac{N^2d}PM^{1+\frac{1}{2s+1}}+CM^{\frac{3}{2s+1}})$ when $n$ is chosen optimally according to Theorem \ref{maintheorem} and $\phi$ is at least in $C^{1,1}$ (corresponding to the index of regularity $s\ge 2$ in the theorem).

\subsection{Accurracy of trajectory prediction}
One application of estimating the interaction kernel from data is to preform predictions of the dynamics. Given an estimator, the following Proposition bounds its accuracy in predicting the trajectories of the system driven by the true interaction kernel:

\begin{proposition}\label{Trajdiff}
 Let $\widehat\intkernel \in \mathcal{K}_{R,S}$ be an estimator of the true kernel $\intkernel$ {, where $ \mathcal{K}_{R,S}$ is the admissible set defined in \eqref{def_AddmissibleSet}.} Denote by $\widehat{\bX}_t$ and $\bX_t$ the solutions of the systems with kernels $\widehat\intkernel$ and $\intkernel$ respectively, starting from the same initial condition and with the same random noise. Then we have
\[
\sup_{t\in[0,T]} \E\left[\frac{1}{N} \|\widehat{\bX}_t- \bX_t\|^2\right] \leq   2T^2 e^{8T^2 (R+1)^2S^2} \vertiii{\widehat{\intkernel}-\intkernel}^2
\]
where the measure $\rho_T$ is defined by \eqref{rhoT}. 
\end{proposition}
We postpone the proof to Section \ref{s:appendixproofs}.

\section{Learning theory: continuous-time observations}\label{sec:theory_cts_traj}
We analyze first the regularized MLE in the case of continuous-time observations $\{\bX^{(m)}_{[0,T]}\}_{m=1}^M$. We show that under a coercivity condition, the regularized  MLE is consistent, and that with proper choice of the hypothesis spaces, we can achieve an optimal learning rate $(\frac{\log M}{ M})^{\frac{2s}{2s+1}}$. 

Recall from Eq.\eqref{lkhd_cts} that
\begin{align*}
\mathcal{E}_{\bX_{[0,T]}}(\intkernelvar)&:=\frac{1}{2\sigma^2TN}\int_{0}^T  \left(\|\rhsF_{\intkernelvar}(\bX_t)\|^2 - 2 \langle \rhsF_{\intkernelvar}(\bX_t), d\bX_t \rangle  dt\right)
\end{align*} 
is the negative log-likelihood of a trajectory $\bX_{[0,T]}$, with respect to the measure induced by the system with interaction kernel $\varphi$. Then, the negative log-likelihood of independent trajectories $\{\bX^{(m)}_{[0,T]}\}_{m=1}^M$ is 
\begin{equation}
\mathcal{E}_{T,M}(\intkernelvar):=\frac{1}{M}\sum_{m=1}^{M}\mathcal{E}_{ \bX_{[0,T]}^{(m)}}(\intkernelvar)\,,
\label{e:ETM}
\end{equation}
and the regularized MLE over a hypothesis space $\mathcal{H}$ is 
\begin{equation}\label{MLE_cts}
\widehat \intkernel_{T,M,\mathcal{H}} \in \argmin{\intkernelvar \in \mathcal{H}}\mathcal{E}_{T,M}(\intkernelvar)\,.
\end{equation}
The existence of the minimizer follows from the fact that the error functional $\mathcal{E}_{T,M}(\intkernelvar)$ is quadratic in $\varphi$, which in turn is a consequence of the linearity of $\rhsF_\varphi$ in $\varphi$.  The uniqueness of the minimizer, however, requires a coercivity condition and is related to the learnability of the kernel, which we discuss in the next section.

\subsection{Identifiability and learnability: a coercivity condition}\label{sec:coercivty}

The uniqueness of the minimizer of the error functional $\mathcal{E}_{T,M}(\intkernelvar)$ over the hypothesis space ensures that the kernel is identifiable. This is not granted, even when the number of observed trajectories is infinite: denote
 \begin{align}\label{errfunctional}
 \mathcal{E}_{T,\infty}(\intkernelvar):=\mathbb{E} \mathcal{E}_{\mbf{X}_{[0,T]}}(\intkernelvar) = \lim_{M\to\infty} \mathcal{E}_{T,M}(\intkernelvar)\quad a.s.~,
 \end{align}
where $\mathbb{E}$ here, and in all that follows unless otherwise indicated, is the expectation over initial conditions, independently sampled from $\mu_0$, and over the Wiener measure underlying the random noise, and observe that 
  \begin{align}\label{difference}
 \mathcal{E}_{T,\infty}(\intkernelvar)-\mathcal{E}_{T,\infty}(\intkernel)&=\mathbb{E}\bigg[ \mathcal{E}_{\mbf{X}_{[0,T]}}(\intkernelvar)-\mathcal{E}_{T,\infty}(\intkernel) \bigg] =\frac{1}{2\sigma^2NT}\mathbb{E}\int_{0}^{T}\| \rhsfo_{\intkernelvar-\intkernel}(\bX_t)\|^2dt\,.
 \end{align}
Only when $\E \int_{0}^{T} \| \rhsfo_{\intkernelvar-\intkernel}(\bX_t)\|^2 dt > 0$ for any $\intkernelvar-\intkernel\neq 0$ can one ensure the uniqueness of minimizer. This motivates us to propose the following coercivity condition, introduced in \cite{BFHM17} in the case of non-stochastic systems: 

\begin{definition}[Coercivity condition] \label{def_coercivity}
We say that the stochastic system defined in \eqref{eq:sod} satisfies a coercivity condition on a set $\hypspace$ of functions on $\mathbb{R}_+$, with a constant $0<c_{\hypspace}$, if 
\begin{align}\label{gencoer}
  c_{\hypspace} \vertiii{\intkernelvar}^2  \!\!\leq   \!\!\frac{1}{2\sigma^2NT}\int_0^T\sum_{i=1}^{N}\E \bigg[ \big\| \frac{1}{N}\sum_{i'= 1}^{N}  \intkernelvar(r_{ii'}(t))\br_{ii'}(t) \big\|^2\bigg]dt
 \end{align} 
for all $\intkernelvar \in \hypspace$ such that $\intkernelvar(\cdot)\cdot\in L^2(\rhoT)$. Here $\vertiii{\cdot}$ denotes the norm defined in \eqref{eq:newnorm}. We will denote by $c_{\hypspace}$ the largest constant for which the inequality holds, and call it the coercivity constant.
 \end{definition}

The coercivity condition ensures identifiability of the kernel. We emphasize that the kernel is latent, in the sense that its values at $\{r_{ii'}=\|\bx_{i'}-\bx_i\|\}$ are undeterminable from data. In fact, to recover $(\intkernel(r_{ii'}) ) \in \R^{\frac{N(N-1)}{2}}$ from the observed trajectories, even if we ignore the stochastic noise in the system and assume to have access to $\rhsfo_\phi(\bx)\in\R^{dN}$, which consists of a linear combination of $(\intkernel(r_{ii'}) )$ with coefficients $\br_{ii'}=\bx_{i'}-\bx_i$, we face a linear system that is underdetermined as soon as $dN$ (=number of known quantities) $\leq \frac{N(N-1)}{2}$ (=number of unknowns), i.e. for $d<(N-1)/2$. Thus, in general the exact values of $\intkernel$ at locations $\{r_{ii'}\}_{i,i'}$ can not be determined. Furthermore, we have  stochastic noise in the system. This suggests that the inverse problem of estimating the interaction kernel in a space of continuous functions is ill-posed. We will see that the coercivity condition ensures well-posedness in $L^2(\rhoT)$, both in the sense of uniqueness and in the sense of stability.
  
 The coercivity condition plays a key role in the learning of the kernel. Beyond ensuring learnability of kernels by ensuring the uniqueness of minimizer over any compact convex sets, it also enables us to control the error of the estimator by the discrepancy  between the expectation of error functionals, as is shown in Proposition \ref{convexity}.  We will use this property to establish the convergence of the estimators in later sections.  
 
 To simplify notation, we define a bilinear functional product over $\mathcal{H}$ by 
\begin{align} \label{def:double_Innerp}
\dbinnerp{\intkernelvar_1, \intkernelvar_2}:=\frac{1}{2\sigma^2NT}\E\int_{0}^{T}\langle \rhsfo_{\intkernelvar_1}(\bX_t),   \rhsfo_{\intkernelvar_2}(\bX_t) \rangle dt, \quad\forall \intkernelvar_1, \intkernelvar_2\in \mathcal{H}. 
\end{align}

\begin{proposition}\label{convexity} 
 Let $\hypspace$ be a compact convex subset of $L^{2}(\R^+,\rhoT)$ and assume the coercivity condition \eqref{gencoer} holds true on $\hypspace$ with constant $c_{\hypspace}$.  Then, the error functional $\mathcal{E}_{T,\infty}$ defined in \eqref{errfunctional} has a unique minimizer over $\mathcal{H}$ in $L^2(\rhoT)$:
 \begin{equation}\label{eq:phi_infty_def}
 \widehat{\phi}_{T,\infty, \mathcal{H}}=\argmin{\intkernelvar \in \mathcal{H}} \mathcal{E}_{T,\infty}(\intkernelvar)\,.
 \end{equation}
Moreover,  for all $\intkernelvar \in \mathcal{H}$, 
 \begin{equation}\label{eq_minH}
 \mathcal{E}_{T,\infty}(\intkernelvar)- \mathcal{E}_{T,\infty}(\widehat{\phi}_{T,\infty, \hypspace}) \geq c_{\hypspace} \vertiii{\intkernelvar-\widehat{\phi}_{T,\infty, \hypspace}  }^2.
 \end{equation}
\end{proposition}

\begin{proof} From Equation \eqref{difference}, we have $\mathcal{E}_{T,\infty}(\intkernelvar)-\mathcal{E}_{T,\infty}(\intkernel)=\dbinnerp{\intkernelvar-\intkernel, \intkernelvar-\intkernel}$. Then,  
\begin{align*}
\mathcal{E}_{T,\infty}(\intkernelvar)&- \mathcal{E}_{T,\infty}(\widehat{\phi}_{T,\infty,\hypspace}) \\
=& \dbinnerp{\intkernelvar-\intkernel, \intkernelvar-\intkernel}-\dbinnerp{\widehat{\phi}_{T,\infty,\hypspace}-\intkernel, \widehat{\phi}_{T,\infty,\hypspace}-\intkernel}\\ 
=& \dbinnerp{ \intkernelvar-\widehat{\phi}_{T,\infty,\hypspace}, \intkernelvar+\widehat{\phi}_{T,\infty,\hypspace}-2\intkernel}\\ 
=&\dbinnerp{\intkernelvar-\widehat{\phi}_{T,\infty,\hypspace}, \intkernelvar-\widehat{\phi}_{T,\infty,\hypspace} }+ 2\dbinnerp{\intkernelvar-\widehat{\phi}_{T,\infty,\hypspace}, \widehat{\phi}_{T,\infty,\hypspace}-\intkernel} \\
\geq &  c_{\hypspace} \vertiii{\intkernelvar-\widehat{\phi}_{T,\infty, \hypspace} }^2 + 2\dbinnerp{\intkernelvar-\widehat{\phi}_{T,\infty,\hypspace}, \widehat{\phi}_{T,\infty,\hypspace}-\intkernel} ,
\end{align*} 
where the inequality follows from the coercivity condition. Then, Eq.\eqref{eq_minH} follows once we notice that  $$\dbinnerp{\intkernelvar-\widehat{\phi}_{T, \infty,\hypspace}, \widehat{\phi}_{T,\infty,\hypspace}-\intkernel} \geq 0$$ by the convexity of $\hypspace$. In fact,  since $t\intkernelvar+(1-t) \widehat{\phi}_{L,\infty,\mathcal{H}} \in \mathcal{H}$, $\forall t \in [0, 1]$, we have $\mathcal{E}_{T,\infty}(t\intkernelvar+(1-t) \widehat{\phi}_{T,\infty,\mathcal{H}} )- \mathcal{E}_{T,\infty}(\widehat{\phi}_{T,\infty,\mathcal{H}}) \geq 0$ since $\widehat{\phi}_{T,\infty,\mathcal{H}}$ is a minimizer, and so, equivalently,
\begin{align*}
& t\dbinnerp{ \intkernelvar-\widehat{\phi}_{T,\infty, \mathcal{H}}, t\intkernelvar+(2-t)\widehat{\phi}_{T,\infty,\mathcal{H}}-2\intkernel} \geq 0 \\
 \Leftrightarrow & \dbinnerp{\intkernel-\widehat{\phi}_{T,\infty,\mathcal{H}}, t\intkernel+(2-t)\widehat{\phi}_{T,\infty, \mathcal{H}}-2\intkernel}\geq 0.
\end{align*} 

Sending $t \rightarrow 0^+$, we obtain $\dbinnerp{\intkernelvar-\widehat{\phi}_{T,\infty, \mathcal{H}}, 2\widehat{\phi}_{T,\infty, \mathcal{H}}-2\intkernel} \geq 0.$ 
\end{proof}

 \paragraph{Well-conditioning from coercivity.} When the hypothesis space $\hypspace$ is a finite-dimensional linear space, the coercivity constant provides a lower bound for the smallest eigenvalue of the limit of the normal equations matrix $A_{M,L}$ in Eq.\eqref{e:ALM} as $M,L\rightarrow+\infty$. Therefore, when the sample size $M$ is large and when the observation frequency $L$ is high, the matrix $A_{M,L}$ is invertible with a high probability (see Corollary \ref{corollary:eignvalue} for details), and thus the coercivity condition ensures the uniqueness of the regularized MLE in Eq.\eqref{eq:phiEst}:
\begin{proposition} \label{prop:A_infty}
Suppose that the coercivity condition holds on $\hypspace = \mathrm{span}\{\psi_1,\cdots, \psi_n\}$, where the basis functions satisfy  
$\langle \psi_p(\cdot)\cdot,\psi_{p'}(\cdot)\cdot \rangle_{L^2{(\rhoT)}}=\delta_{p,p'}$. 
Let $A_{\infty}=\big(\dbinnerp{\psi_{p},\psi_{p'}}\big)_{p, p'} \in \mathbb{R}^{n \times n}$ with the bilinear functional $\dbinnerp{\cdot,\cdot}$ defined in \eqref{def:double_Innerp}. Then the smallest singular value of $A_{\infty}$ is 
$\lambda_{\min}(A_{\infty}) = c_{\hypspace}\,.$

\end{proposition} 
\begin{proof} For an arbitrary $a\in \R^n$, denoting $\psi= \sum_{p=1}^n a_p \psi_p$, we have
\begin{equation}\label{eq:A-eigen}
a^T A_\infty a = \dbinnerp{\psi,\psi} \geq c_{\hypspace}\vertiii{\psi}^2 = c_{\hypspace}\|a\|^2
\end{equation}
where the first equality follows from that the functional $\dbinnerp{\cdot,\cdot}$ is bilinear, and the inequality follows from the coercivity condition. Note that by the definition of the coercivity constant in \eqref{gencoer}, we have 
\[ 
c_\hypspace = \sup_{\psi\in \hypspace}\frac{\dbinnerp{\psi,\psi} }{\vertiii{\psi}^2} = \sup_{\psi\in \hypspace, \vertiii{\psi} =1} \dbinnerp{\psi,\psi} ,
\]
 which is attained at some $\psi_* \in \hypspace$ since $\hypspace$ is finite dimensional.  Hence, the inequality in \eqref{eq:A-eigen} becomes inequality for  $\psi_*$ and the smallest eigenvalue of of $A_{\infty}$ is $c_\hypspace $.
 \end{proof}

Proposition \ref{prop:A_infty} suggests that for the hypothesis space $\hypspace$, it is important to choose a basis that is orthonormal in $L^2(\rhoT)$, so as to make the matrix in the normal equations as well-conditioned as possible given the dynamics. In practice, the unknown $\smash{\rhoT}$ is approximated by the empirical density $\smash{\rhoT^{L,M}}$. Therefore, when using local basis functions, it is natural to use a partition of the support of $\smash{\rho_T^{ M}}$. 

\paragraph{The coercivity condition and positive integral operators.} The coercivity condition introduces constraints on the hypothesis spaces and on the distribution of the solutions of the system, and it is therefore natural that it depends on the distribution $\mu_0$ of the initial condition $\bX_0$, the true interaction kernel $\intkernel$, and the random noise.  We review below briefly the recent developments in \cite{li2019identifiability}, where the coercivity condition is proved to hold on any compact sets of $L^2(\rhoT)$ for special classes of systems, such as linear systems and nonlinear systems with a stationary distribution.    As discussed in \cite{BFHM17,LZTM19,LMT19} for the deterministic cases, we believe that the coercivity condition is ``generally'' satisfied for ``relevant'' hypothesis spaces, with a constant independent of the number of particles $N$, thanks to the exchangeability of both the distribution of the initial conditions and that of the particles at any time $t$. 

The coercivity condition is equivalent to the positiveness of integral operators that arise in the expectation in Eq.\eqref{gencoer}. More precisely, by the exchangeability of the distribution of $\bX_t$, one can rewrite Eq.\eqref{gencoer} as
\begin{align*}
c_{\hypspace} \|\varphi(\cdot)\cdot\|_{L^2(\rhoT)}^2  \! & \leq 
\frac{1}{T}\int_0^T \E[\varphi(|\br_{12}^t|)\varphi(|\br_{13}^t|)\innerp{\br_{12}^t}{\br_{13}^t}]dt  \\
& =\int_0^\infty \int_0^\infty \varphi(r)\varphi(s)  \overline{K}_T(r,s)drds,
\end{align*}
where the integral kernel  $\overline{K}_T:\R^+\times \R^+ \to \R$ is defined as
 \begin{equation} \label{kernelK}
\overline{K}_T(r,s) := (rs)^{d}  \int_{S^{d-1}}\int_{S^{d-1}}\innerp{\xi}{\eta} \frac{1}{T}\int_0^T p_t(r\xi,s\eta)dt d\xi d\eta,
 \end{equation}
with $p_t(u,v)$ denoting the joint density function of the random vector $(\br_{12}^t, \br_{13}^t)$ and $S^{d-1}$ denoting the unit sphere in $\R^d$. The integral kernel $\overline{K}$ is symmetric semi-positive definite and leads to a semi-positive self-adjoint integral operator $L_{\overline{K}}$. Then, the coercivity condition holds on $\hypspace$ if and only if $\hypspace$ is contained in the eigen-space of $L_{\overline{K}}$ with eigenvalues no less than $c_{\hypspace}$. In particular, if $L_{\overline{K}}$ is strictly positive, then the coercivity condition holds true for any compact $\hypspace\subset L^2(\rho_T)$. Using M\"untz-type theorems, it is shown in \cite{li2019identifiability} that $L_{\overline{K}}$ is strictly positive definite, and therefore the coercivity condition holds, for a large class of systems with interaction kernels in form of $\phi(r)= (a+r^\theta)^{\gamma-1}r^{\theta-2}$ with $a\geq 0$ and $\{ (\theta, \gamma) \in (0,1] \times (1,2]: \theta\gamma >1 \}$.


\subsection{Consistency and rate of convergence of the estimator} 
In this section, we consider using a family of finite dimensional linear spaces $\{\mathcal{L}_n: n\in \mathbb{N}^+\} \subset C^{1,1}[0,R]$ as hypothesis spaces and establish the consistency and rate of convergence of our estimators.  We assume  the spaces $\{\mathcal{L}_n: n\in \mathbb{N}^+\} \subset C^{1,1}[0,R]$ satisfying Markov-Bernstein type inequality: there exist $c_1, \gamma>0$ s.t. for all $\intkernelvar \in \mathcal{L}_n$
\begin{align}\label{markovbernstein}
{\|\intkernelvar'\|_{\infty}}\leq c_1\dimsp(\mathcal{L}_n)^\gamma {\|\intkernelvar\|_{\infty} }\,.
\end{align}
This condition has a long history and rich literature in classical approximation theory, where it is studied when function spaces satisfy \eqref{markovbernstein} (e.g. see the survey paper \cite{ward2012lp}), which is an important step in establishing inverse approximation theorems.  This  kind of inequality holds true on many function spaces that are commonly used as approximation spaces in practice, including:
\begin{itemize}
\item  $\mathcal{L}_n:$ trigonometric polynomials of degree $n$ on $[0, 2\pi]$ (similarly on $[0,R]$), for which
$\|\phi'\|_{\infty}\leq \frac12({\dimsp(\mathcal{L}_n)-1})\|\phi\|_{\infty}$. This result dates back to Bernstein \cite{bernstein1912ordre}.  
\item $\mathcal{L}_n$: the polynomial space consisting of all polynomials with  degree less than $n-1$ on $[0, R]$ (see Theorem 3.3 in \cite{schumaker2007spline}), for which
$\|\intkernelvar'\|_{\infty}\leq \frac2R{(\dimsp(\mathcal{L}_n)+1)^2}\|\intkernelvar\|_{\infty}$.
As a result, \eqref{markovbernstein} also holds true for polynomial splines; other extensions include rational functions. We refer to the reader to \cite{kalmykov2017bernstein} for details. 
\end{itemize}

If we choose a compact convex hypothesis set $\hypspace_M$ contained in some $\mathcal{L}_n$, with a suitable correspondence between $n$ and $M$, such that the distance between $\hypspace_M$ and the true kernel $\intkernel$ vanishes as $M$ increases, the following consistency result holds:

\begin{theorem}[Strong consistency of estimators]\label{main:consistency} Suppose $\phi \in \mathcal{K}_{R,S}$, the admissible set defined in \eqref{def_AddmissibleSet}. Let $\{\mathcal{L}_n: n\in \mathbb{N}^+\} \subset C^{1,1}[0,R]$ satisfying \eqref{markovbernstein} and
$$\inf_{\intkernelvar \in \mathcal{L}_{n}}\|\intkernelvar-\intkernel\|_{\infty} \xrightarrow{n\rightarrow \infty} 0.$$
 Let $S_0 \geq S$ and $\hypspace_M=\mathcal{B}_{2S_0}^{\infty}(\mathcal{L}_{n_M})=\{\varphi\in\mathcal{L}_{n_M}\,:\,||\varphi||_\infty<2S_0\}$ with $\mathrm{dim}(\mathcal{L}_{n_M})=n_M$ and $\lim_{M\rightarrow \infty}\frac{n_M\log n_M}{M}  =0$.  Finally, suppose the coercivity condition holds true on $ \cup_{n}\mathcal{L}_{n}$. Then we have 
$$\lim_{M\rightarrow \infty}\vertiii{\widehat \intkernel_{T,M,\hypspace_M}-\intkernel}  =0 \quad\text{ with probability one.}$$
\end{theorem}

If we know the explicit approximation rate of the family $\{\mathcal{L}_n: n\in \mathbb{N}^+\}$, then by carefully choosing the dimension of hypothesis spaces as a function of $M$, we can obtain a near-optimal rate of convergence of our estimators. 
 
\begin{theorem}[Convergence rate of estimators]\label{maintheorem}
Suppose $\phi \in \mathcal{K}_{R,S}$, the admissible set defined in \eqref{def_AddmissibleSet}. Assume that there exits a sequence of linear  spaces $\{\mathcal{L}_n: n\in \mathbb{N}^+\} \subset C^{1,1}[0,R]$ satisfying \eqref{markovbernstein} with the properties
\begin{itemize}
\item[(i)] $\mathrm{dim}(\mathcal{L}_n)\leq c_0 n$ for $n \in \mathbb{N}^+$,
\item[(ii)] $\inf_{\intkernelvar \in \mathcal{L}_n}\|\intkernelvar-\intkernel\|_{\infty}\leq c_2n^{-s}$.
\end{itemize}
For example, when $\phi \in C^{k,\alpha}$ with $s=k+\alpha \geq 2$, we may choose $\mathcal{L}_n$ to consist of polynomial splines of degree $\lfloor s-1\rfloor$ with uniform knots on $[0, R]$. 
Let $\hypspace_n=\mathcal{B}_{S_0}^{\infty}(\mathcal{L}_{n})$ with $S_0=c_2+S$ and $n \asymp \left({M}/{\log M}\right)^{{1}/{(2s+1)}}$, and assume that the coercivity condition holds on $\mathcal{L}:=\cup_{n}\mathcal{L}_n$  with a constant $c_{\mathcal{L}}>0$. Then we have  
$$
\E\left[\vertiii {\widehat \intkernel_{T,M,\mathcal{H}_{n}}-\intkernel}^2\right] \leq \frac{C}{c^2_{{\mathcal{L}}}} \left(\frac{\log M}{ M}\right)^{\frac{2s}{2s+1}}\,,
$$ 
where $C$ is a constant depending only on $\sigma,N, T, R, S_0$. 
\end{theorem}

It is fruitful to compare (up to log terms) the rate ${2s}/({2s+1})$ to that for nonparametric $1$-dimensional regression, where one can observe directly noisy values of the target function $\intkernel$ at sample points drawn i.i.d from $\rhoT$. For the function space $C^{k,\alpha}$, this rate is min-max optimal. Our numerical examples in Section \ref{main:numericalexamples} empirically validate the desired convergence rate for $s=1,2$ where we use piecewise constant and linear polynomials.  Note that in our setting, learning $\phi$ is an inverse problem, as we do not directly observe the values $\{\intkernel(\|\bx_{i',t}^{(m)} - \bx_{i,t}^{(m)}\|)\}_{ i, i' =1,\ m=1}^{N,N,M}$.  We also do not require the underlying stochastic process satisfies certain mixing properties and starts from a stationary distribution. Obtaining this optimal convergence rate in $M$ for short time trajectory observations is therefore satisfactory.  For long trajectories and under ergodicity assumptions, rates in terms of $MT$ are likely to be obtainable: in Section \ref{main:numericalexamples} we present numerical evidence that suggests that the error does decrease with $MT$ at a near optimal rate.


\subsection{Proof of the main theorems}
In the following part, we present the proof for Theorem \ref{maintheorem}, which also yields the proof for Theorem \ref{main:consistency}. The main techniques includes the It\^o formula, concentration inequalities of unbounded random variables, and a generalization of a novel covering argument in \cite{wang2011optimal} that enables us to deal efficiently with the fluctuations in the data due to the stochastic noise in the dynamics of the system.   \\

One major obstacle in the non-asymptotic analysis of our regularized MLE estimators is the unboundness of stochastic integral of the form $\frac{1}{T}\int_{0}^{T}\langle \rhsfo_{\intkernelvar}(\bX_t), d\bX_t\rangle dt$ appearing in the empirical error functional. Unlike the deterministic case $\sigma=0$,  our empirical error functional $\mathcal{E}_{T,M}(\cdot)$ is  in general  not continuous over  $\mathcal{H}$ with respect to the $\|\cdot\|_{\infty}$ norm. In the following, we first leverage {the general It\^o formula} described in Theorem \ref{Itoformula}, to obtain a form of the empirical error functional that does not involve a stochastic integral and is amenable to analysis; we then show that it is continuous on $C^{1,1}([0,R])$ with respect to the $\|\cdot\|_{1,1}$ norm. {Therefore, in the following preliminary results for the proofs of the main theorems, we consider the following generic hypothesis space:  
\begin{assumption}  The hypothesis space $\mathcal{H}$ is a compact convex subset of $C^{1,1}([0,R])$ with respect to the uniform norm $\|\cdot\|_\infty$ and bounded above by $S_0\geq S$. 
 \end{assumption}
}

\begin{lemma}\label{keylemma}Suppose $\intkernelvar \in \mathcal{K}_{R,S}$, the admissible set defined in \eqref{def_AddmissibleSet}. Let 
\begin{align}\label{potentialfunction}
V_{\intkernelvar}(\bX_t)=\frac{1}{2N}\sum_{i,i'}\Psi(\|\bx_{i,t}-\bx_{i',t}\|)\, \text{ with }\, \Psi'(r)=\intkernelvar(r)r\,;
\end{align} then{, we have, almost surely}
$$-(dV_{\intkernelvar})(\bX_t)=\langle \rhsfo_{\intkernelvar}(\bX_t), d\bX_t\rangle+\frac{\sigma^2}{2N}\sum_{i=1}^N\sum_{i'\neq i}\left(\intkernelvar'(\|\bx_{ii'}\|)\|\bx_{ii'}\|+\intkernelvar(\|\bx_{ii'}\|)d \right)dt$$ 
\end{lemma}
\begin{proof}  Let $g(\bX)=V_{\intkernelvar}(\bX)$.
Note that $g$ is $C^2$, with derivatives

\begin{align*}
\frac{\partial{g(\bX)}}{\partial \bx_i}&=\frac{\partial{V_{\intkernelvar}(\bX)}}{\partial \bx_i}=\frac{1}{N}\sum_{i'\neq i}\varphi(\|\bx_i-\bx_{i'}\|)(\bx_{i}-\bx_{i'})\\ \frac{\partial^2{g(\bX)}}{\partial \bx_i\partial \bx_{k}} &=-\delta_{ki}\frac{1}{N}\left(\sum_{i'\neq i}\varphi(\|\bx_i-\bx_{i'}\|)\mathbf{I}_d+\frac{\varphi'(\|\bx_i-\bx_{i'}\|)}{\|\bx_i-\bx_{i'}\|}(\bx_i-\bx_{i'})\otimes(\bx_i-\bx_{i'})\right) \\
& \quad  -\delta_{k\neq i} \frac{1}{N}\left( \varphi(\|\bx_i-\bx_{k}\|)\mathbf{I}_d+\frac{\varphi'(\|\bx_i-\bx_{k}\|)}{\|\bx_i-\bx_{k}\|}(\bx_i-\bx_{k})\otimes(\bx_i-\bx_{k})\right)\,.
\end{align*}
Using It\^o's formula (Theorem \ref{Itoformula}) for the It\^o process $g(\bX_t)$,  the conclusion follows. 
\end{proof}

\begin{proposition}\label{randomdifference}Suppose $\intkernelvar_1, \intkernelvar_2 \in \mathcal{H}$, then {it holds almost surely that }
\begin{align}
|\mathcal{E}_{\bX_{[0,T]}}(\intkernelvar_1)-\mathcal{E}_{\bX_{[0,T]}}(\intkernelvar_2)|\nonumber \leq C_1 \|\intkernelvar_1-\intkernelvar_2\|_{\infty}+C_2\|\intkernelvar_1'-\intkernelvar_2'\|_{\infty}, \end{align} 
where $C_1=\frac{ R^2S_0}{\sigma^2}+\frac{ R^2 }{2\sigma^2T}+\frac{d}{2}$ and $C_2=\frac{R}{2}$.
\end{proposition}

\begin{proof} Note that 
\begin{align*}
&\mathcal{E}_{\bX_{[0,T]}}(\intkernelvar_1)-\mathcal{E}_{\bX_{[0,T]}}(\intkernelvar_2)\\&=\frac{1}{2\sigma^2NT}\int_{0}^{T}\underbrace{\langle \rhsfo_{\intkernelvar_1-\intkernelvar_2}(\bX_t), \rhsfo_{\intkernelvar_1+\intkernelvar_2}(\bX_t)\rangle}_{I_1}dt-2\underbrace{\langle \rhsfo_{\intkernelvar_1-\intkernelvar_2}(\bX_t), d\bX_t\rangle}_{I_2}.
\end{align*} 

$I_1$ satisfies  
\begin{align}
 |I_1|&\leq  \|\rhsfo_{\intkernelvar_1+\intkernelvar_2}(\bX_t)\|\|\rhsfo_{\intkernelvar_1-\intkernelvar_2}(\bX_t)\|\nonumber \leq N\|\intkernelvar_1-\intkernelvar_2\|_{\infty}\|\intkernelvar_1+\intkernelvar_2\|_{\infty}R^2,
 \end{align}
since $\|\rhsfo_{\intkernelvar}(\bX_t)\| \leq \sqrt{N}R\|\intkernelvar\|_{\infty}$.  For $I_2$,  Lemma \ref{keylemma} yields 
\begin{align*}
|I_2|&\leq|V_{\intkernelvar_1-\intkernelvar_2}(\bX(0))-V_{\intkernelvar_1-\intkernelvar_2}(\bX_t)|\\&\quad\quad\quad+\frac{\sigma^2}{2N}\left|\int_{0}^{T}\sum_{i=1}^N\sum_{i'\neq i}((\intkernelvar_1-\intkernelvar_2)'(r_{ii'})r_{ii'}+(\intkernelvar_1-\intkernelvar_2)(r_{ii'})d )dt\right| \nonumber\\&\leq \frac{(N-1)}{2}\|\intkernelvar_1-\intkernelvar_2\|_{\infty}R^2+\frac{(N-1)\sigma^2T}{2}(\|\intkernelvar_1'-\intkernelvar_2'\|_{\infty}R+\|\intkernelvar_1-\intkernelvar_2\|_{\infty}d),
\end{align*} where we used 
\begin{align*}
|V_{\intkernelvar}(\bX_t)-V_{\intkernelvar}(\bX(0))|&\leq \frac{1}{2N}\sum_{ii'}|\int_{r_{ii',0}}^{r_{ii',T}} \intkernelvar(r)r dr |  \leq \frac{(N-1)\|\intkernelvar\|_{\infty}R^2}{2}, 
\end{align*} 
which follows from its definition in \eqref{potentialfunction}.  Combining the estimates for $I_1$ and $I_2$, and using $\|\intkernelvar_1 +\intkernelvar_2\|_{\infty}\leq 2S_0$, we obtain
\begin{align}
&|\mathcal{E}_{\bX_{[0,T]}}(\intkernelvar_1)-\mathcal{E}_{\bX_{[0,T]}}(\intkernelvar_2)|\nonumber\\
&\leq \frac{ R^2}{2\sigma^2}\|\intkernelvar_1-\intkernelvar_2\|_{\infty} \left( \|\intkernelvar_1+\intkernelvar_2\|_{\infty} +\frac{1}{T}\right)+\frac{1}{2}(d\|\intkernelvar_1-\intkernelvar_2\|_{\infty}+\|\intkernelvar_1'-\intkernelvar_2'\|_{\infty}R)\nonumber\\
&\leq C_1\|\intkernelvar_1-\intkernelvar_2\|_{\infty}+C_2\|\intkernelvar_1'-\intkernelvar_2'\|_{\infty}, 
\end{align}where  $C_1=\frac{ R^2S_0}{\sigma^2}+\frac{ R^2 }{2\sigma^2T}+\frac{d}{2}$ and $C_2=\frac{R}{2}$.
\end{proof}

When $M=\infty$, i.e. we observe infinitely many trajectories, the expectation of our error functional $\mathcal{E}_{T,\infty}$, as in \eqref{errfunctional}, does not involve the stochastic integral term. From the proof of Proposition \ref{randomdifference}, we see that it is  continuous over $\mathcal{H}$ with respect to the $\|\cdot\|_{\infty}$ norm:

\begin{corollary}\label{continuousproperty}Suppose $\intkernelvar_1, \intkernelvar_2 \in \mathcal{H}$, then, with  $C_1=\frac{2R^2S_0}{\sigma^2}$, we have 
\begin{align}
|\mathcal{E}_{T,\infty}(\intkernelvar_1)-\mathcal{E}_{T,\infty}(\intkernelvar_2)|\nonumber \leq C_1 \|\intkernelvar_1-\intkernelvar_2\|_{\infty}. \end{align}
\end{corollary}

\begin{proof} Using \eqref{difference}, we obtain that

  \begin{align*}
 |\mathcal{E}_{T,\infty}(\intkernelvar_1)-\mathcal{E}_{T,\infty}(\intkernelvar_2)|&=\frac{1}{2\sigma^2NT}\mathbb{E}\int_{0}^{T}\big|  \langle \rhsfo_{\intkernelvar_1-\intkernelvar_2}(\bX_t), \rhsfo_{\intkernelvar_1+\intkernelvar_2-2\intkernel}(\bX_t) \rangle \big|dt \\ &\leq \frac{R^2}{2\sigma^2}\| \intkernelvar_1-\intkernelvar_2\|_{\infty} \| \intkernelvar_1+\intkernelvar_2-2\intkernel \|_{\infty} \leq \frac{2R^2S_0}{\sigma^2}\| \intkernelvar_1-\intkernelvar_2\|_{\infty} 
 \end{align*}
 
\end{proof}

Recall the definition
 $$\widehat{\phi}_{T,\infty, \mathcal{H}}:=\argmin{\intkernelvar \in \mathcal{H}} \mathcal{E}_{T,\infty}(\intkernelvar).$$ 
 We now analyze the discrepancy between the empirical minimizer $\widehat{\phi}_{T,M, \mathcal{H}}$  and $\widehat{\phi}_{T,\infty, \mathcal{H}}$, which we called Sampling Error (SE) in the diagram in Figure \ref{f:diagram}. We introduce a measurable function on the path space by 
 \begin{equation}
 D_{\intkernelvar}: =\mathcal{E}_{\mbf{X}_{[0,T]}}(\intkernelvar)-\mathcal{E}_{\mbf{X}_{[0,T]}}(\widehat{\phi}_{T,\infty, \hypspace})
 \end{equation} for any $\intkernelvar \in \hypspace$. From  Proposition \ref{convexity} we have
\begin{align}\label{coercivityinequality}
\E D_{\intkernelvar} &=\mathcal{E}_{T,\infty}(\intkernelvar)-\mathcal{E}_{T,\infty}(\widehat{\phi}_{T,\infty, \mathcal{H}})\geq c_{\hypspace}\vertiii{\intkernelvar-\widehat\intkernel_{T,\infty,\mathcal{H}}}, 
\end{align} so $D_{\intkernelvar}$ in fact bounds (in expectation) the distance between $\intkernelvar$ and $\widehat{\phi}_{ T,\infty, \hypspace}$ w.r.t. the $|||\cdot|||$-norm. 
We now perform a non-asymptotic analysis of $D_{\intkernelvar}$. We shall show that the random variable $D_{\intkernelvar}$ satisfies moment conditions, sufficient to guarantee strong concentration about its expectation (Proposition \ref{covering}).  To do this, we decompose $D_{\intkernelvar}$ as the sum of a  bounded component only involving  time integrals and  an unbounded component involving  stochastic integrals: 
\begin{align*}
D_{\intkernelvar}:&=D_{\intkernelvar}^{\mathrm{bd}}-D_{\intkernelvar}^{\mathrm{ubd}}\\D_{\intkernelvar}^{\mathrm{bd}} :&=\frac{1}{2\sigma^2NT}\int_{0}^{T}\langle \rhsfo_{\intkernelvar-\widehat{\phi}_{T,\infty, \hypspace}}(\bX_t),\rhsfo_{\intkernelvar+\widehat{\phi}_{T,\infty, \hypspace}-2\intkernel}(\bX_t)\rangle dt\\D_{\intkernelvar}^{\mathrm{ubd}}:&=\frac{1}{\sigma^2 NT}\int_{0}^{T}\langle \rhsfo_{\intkernelvar-\widehat{\phi}_{T,\infty, \hypspace}}(\bX_t),d\bB(t)\rangle
\end{align*}
We  prove moment conditions independently for each of these components in the next two Lemmata.

\begin{lemma}[Bounds on $D_{\intkernelvar}^{\mathrm{ubd}}$]\label{unboundedpart} For $\intkernelvar\in \mathcal{H}$ and $p=2,3,4,\dots$,
\begin{align*}\E \big| D_{\intkernelvar}^{\mathrm{ubd}}|^p  \leq C \big(\big\|\intkernelvar-\widehat{\phi}_{T,\infty, \hypspace}\big\|_{\infty}\big)^{p-2}\vertiii{\intkernelvar-\widehat{\phi}_{T,\infty, \hypspace}}^2
\end{align*} where $C=(\frac{p(p-1)}{2})^{\frac{p}{2}}\frac{R^{p-2}}{\sigma^{2p}(NT)^{\frac{p+2}{2}}}$.
\end{lemma}
\begin{proof}First of all, note that 
\begin{align}\label{boundfphi}
\big\|\rhsfo_{\intkernelvar-\widehat{\phi}_{T,\infty, \hypspace}}(\bX_t)\big\|\leq \sqrt{N}R \big\|\intkernelvar-\widehat{\phi}_{T,\infty, \hypspace}\big\|_{\infty}\,.
\end{align}
Therefore  $\rhsfo_{\intkernelvar-\widehat{\phi}_{T,\infty, \hypspace}}(\bX_t)$ is a $L^2$-integrable process.  Applying Theorem \ref{momentinequality}, we obtain  

\begin{align*}
&\mathbb{E}  \bigg| \int_{0}^{T}\langle \rhsfo_{\intkernelvar-\widehat{\phi}_{T,\infty, \hypspace}}(\bX_t), d\bB(t)\rangle\bigg|^p \leq C_{p,T} \mathbb{E}\int_{0}^{T}\big\|\rhsfo_{\intkernelvar-\widehat{\phi}_{T,\infty, \hypspace}}(\bX_t) \big\|^pdt\\
&\leq C_{p,T}\big(\sqrt{N}R\|{\intkernelvar-\widehat{\phi}_{T,\infty, \hypspace}}\|_{\infty} \big)^{p-2}\mathbb{E}\int_{0}^T\big\|\rhsfo_{\intkernelvar-\widehat{\phi}_{T,\infty, \hypspace}}(\bX_t)\big\|^2dt\\
&\leq C_{p,T}\big(\sqrt{N}R\|{\intkernelvar-\widehat{\phi}_{T,\infty, \hypspace}}\|_{\infty} \big)^{p-2} NT\vertiii{\intkernelvar-\widehat{\phi}_{T,\infty, \hypspace}}^2\,,
\end{align*}  with $C_{p,T}=\big(\frac{p(p-1)}{2}\big)^{\frac{p}{2}}T^{\frac{p-2}{2}}$.   
The conclusion then follows by adding in the scaling factor $\frac{1}{\sigma^2 NT}$.
\end{proof}

\begin{lemma}[Bounds on $D_{\intkernelvar}^{\mathrm{bd}}$]\label{boundedpart}For $\intkernelvar\in \mathcal{H}$ and $p = 2,3,4,\dots$, 
\begin{align*}
\E \big| D_{\intkernelvar}^{\mathrm{bd}} \big|^p \leq \left(\frac{2S_0}{\sigma^2}\right)^pR^{2p-2}\big\|\intkernelvar-\widehat{\phi}_{T,\infty, \hypspace}\big\|_{\infty}^{p-2} \vertiii{\intkernelvar-\widehat{\phi}_{T,\infty, \hypspace}}^2. 
\end{align*}
\label{boundedpart}
\end{lemma}
\begin{proof} From the inequality \eqref{boundfphi} and the linear dependence of $\rhsfo_{\intkernelvar}$ on $\intkernelvar$, we have
\begin{align*}
\big| D_{\intkernelvar}^{\mathrm{bd}}\big|&\leq \frac{2S_0R}{\sigma^2\sqrt{N}T}\int_0^T\big\|\rhsfo_{\intkernelvar-\widehat{\phi}_{T,\infty, \hypspace}}(\bX_t)\big\|dt \nonumber \leq \frac{2S_0R}{\sigma^2}\sqrt{\frac{1}{NT}\int_{0}^{T}\big\|\rhsfo_{\intkernelvar-\widehat{\phi}_{T,\infty, \hypspace}}(\bX_t)\big\|^2dt}.
\end{align*}Therefore, 
$$\E|D_{\intkernelvar}^{\mathrm{bd}}|^p \leq ( \frac{2S_0}{\sigma^2})^pR^{2p-2}\big\|\intkernelvar-\widehat{\phi}_{T,\infty, \hypspace}\big\|_{\infty}^{p-2}\vertiii{\intkernelvar-\widehat{\phi}_{T,\infty, \hypspace}}^2.$$
\end{proof}

Now we combine  Lemma \ref{unboundedpart} and \ref{boundedpart} to prove the moment condition for $D_{\intkernelvar}$.

\begin{lemma}[{Moment conditions}]\label{momentcondition} For $\intkernelvar \in \mathcal{H}$, and every $p=2,3,\cdots$, we have
\begin{align*}
&\E\bigg[ \big| D_{\intkernelvar}-\E D_{\intkernelvar} \big|^p \bigg] \leq \frac{1}{2} p! K_{\intkernelvar,\mathcal{H}}^{p-2} C_{\intkernelvar,\mathcal{H}}, \end{align*}
where \begin{align}\label{constant2}
K_{\intkernelvar,\mathcal{H}}:= C_0\big\|\intkernelvar-\widehat{\phi}_{T,\infty, \hypspace}\big\|_{\infty}\,,\quad  C_{\intkernelvar,\mathcal{H}}:=C_0^2 C_1\vertiii{\intkernelvar-\widehat{\phi}_{T,\infty, \hypspace}}^2
\end{align} with $C_0=\sqrt{\frac{2 e^2R^2}{\sigma^4NT}}, $ 
$C_1 =\max_{p\geq 2} \frac{1}{\sqrt{2\pi p}NTR^2}\bigg(1+ \frac{c_{\sigma,S_0,R}}{C_0^p} \bigg),$ and 
$ 
c_{\sigma,S_0,R}=\max_{p\geq 2} (\frac{8S_0}{\sigma^2e})^p \frac{R^{2p-2}}{\sqrt{2\pi}p^{p+\frac{1}{2}}}.
$
\end{lemma}
\begin{proof} The proof is based on the Jensen's inequality, Lemma \ref{unboundedpart} and Lemma \ref{boundedpart}.
\begin{align*}
\E \big| D_{\intkernelvar}-\E D_{\intkernelvar}\big|^p  &\leq 2^{p-1}\E  \big| D_{\intkernelvar}^{\mathrm{bd}}-\E D_{\intkernelvar}^{\mathrm{bd}} \big|^p+2^{p-1}\E\big|D_{\intkernelvar}^{\mathrm{ubd}} \big|^p \\&\leq 2^{2p-1}\E \big|D_{\intkernelvar}^{\mathrm{bd}} \big|^p +2^{p-1}\E[\big|D_{\intkernelvar}^{\mathrm{ubd}} \big|^p\nonumber\\ &\leq \frac{1}{2}C_0 \big\|\intkernelvar-\widehat{\phi}_{T,\infty, \hypspace}\big\|_{\infty}^{p-2}\vertiii{\intkernelvar-\widehat{\phi}_{T,\infty, \hypspace}}^2 \\& \leq  \frac{1}{2} p! (C_1\|\intkernelvar-\widehat{\phi}_{T,\infty, \hypspace}\|_{\infty})^{p-2} C_1^2C_2\vertiii{\intkernelvar-\widehat{\phi}_{T,\infty, \hypspace}}^2\nonumber, \end{align*}
where  the constants are \begin{align*}
C_0&:=(\frac{8S_0}{\sigma^2})^pR^{2p-2}+(2p(p-1))^{\frac{p}{2}}\frac{R^{p-2}}{\sigma^{2p}(NT)^{\frac{p+2}{2}}},\\
C_1&:=\sqrt{\frac{2 e^2R^2}{\sigma^4NT}},
\quad\quad C_2:=\max_{p\geq 2} \bigg(\frac{1}{\sqrt{2\pi p}NTR^2}+ \frac{c_{\sigma,S_0,R}}{C_1^p} \bigg),\\
c_{\sigma,S_0,R}&:=\max_{p\geq 2} (\frac{8S_0}{\sigma^2e})^p \frac{R^{2p-2}}{\sqrt{2\pi}p^{p+\frac{1}{2}}}, 
\end{align*}  and the last inequality is derived from the Stirling's lemma. 

\end{proof}

We now tie the discrepancy functionals for finite and infinite $M$:
\begin{proposition}[Sampling Error bound]\label{covering}
 Let $0 <\delta<1$ and $\{\intkernel_j\}_{j=1}^{\mathcal{N}}$ be an $\eta$ net of functions in a compact convex hypothesis space $\mathcal{H} \subset Ball_{S_0}(L^{\infty}[0,R])$. Denote 
\begin{align*}
\mathcal{D}_{\intkernelvar_j,\infty}:&= \E D_{\intkernelvar_j}=\mathcal{E}_{T,\infty}(\intkernelvar_j)-\mathcal{E}_{T,\infty}(\widehat\intkernel_{T,\infty,\mathcal{H}})\\
\mathcal{D}_{\intkernelvar_j,M}:&=\mathcal{E}_{T,M}(\intkernelvar_j)-\mathcal{E}_{T,M}(\widehat\intkernel_{T,\infty,\mathcal{H}})\,. 
\end{align*}Then with probability at least $1-\frac{\delta}{2}$, we have 
\begin{align}\label{coveringargument}
\mathcal{D}_{\intkernelvar_j,\infty}- \mathcal{D}_{\intkernelvar_j,M}\leq \epsilon_{M,\delta, \mathcal{N}}+ \frac{1}{2}\mathcal{D}_{\intkernelvar_j,\infty}
\end{align}
for all $j$, where $\epsilon_{M,\delta,\mathcal{N}}=\frac{C}{M}\log(\frac{2\mathcal{N}}{\delta})$, $C=2\frac{C_0^2C_1}{c_\mathcal{H}}+4C_0S_0$, and $C_0,C_1$ as in \eqref{constant2}, with $c_{\hypspace}$ the coercivity constant defined in \eqref{gencoer}.
\end{proposition}

\begin{proof}For each $\intkernelvar_j \in \mathcal{H}$, recall that in Eq.\eqref{coercivityinequality}, the coercivity condition on $\mathcal{H}$ implies that 
$$\mathcal{D}_{\intkernelvar_j,\infty} \geq c_{\hypspace}  \vertiii{\intkernelvar_j- \widehat{\phi}_{T,\infty, \hypspace}}^2. 
$$ 
Then, Eq.\eqref{constant2} in Lemma \ref{momentcondition} yields that 
\begin{align}
&\E|{D}_{\intkernelvar_j}-\mathbb{E}{D}_{\intkernelvar_j}|^p \leq \frac{1}{2} p! K_{\intkernelvar_j,\mathcal{H}}^{p-2}\frac{C_0^2C_1}{c_{\mathcal{H}}} \mathcal{D}_{\intkernelvar_j,\infty}.\end{align}
Therefore the random variable   ${D}_{\intkernelvar_j}$ satisfies the moment condition in Corollary \ref{berstein2}, and so $\forall \epsilon>0$
$$\Prob{\mathcal{D}_{\intkernelvar_j,\infty} -\mathcal{D}_{\intkernelvar_j,M} \geq \sqrt{\epsilon(\epsilon+ \mathcal{D}_{\intkernelvar_j,\infty})}} \leq \exp \bigg(\frac{-M\epsilon}{\frac{2C_0^2C_1}{c_{\mathcal{H}}}+2K_{\intkernelvar_j,\mathcal{H}}}\bigg). $$
 We have $K_{\intkernelvar_j,\mathcal{H}}\leq 2C_0S_0$ where $C_0$ is defined in \eqref{constant2}.  Taking a union bound on all these events, over $j \in \{1,2,\cdots,\mathcal{N}\}$,  we obtain that
\begin{align}\label{maxprob}
\Prob{ \max_{1\leq j \leq \mathcal{N}}\frac{\mathcal{D}_{\intkernelvar_j,\infty} -\mathcal{D}_{\intkernelvar_j,M}}{\sqrt{\mathcal{D}_{\intkernelvar_j,\infty}+\epsilon}} \geq \sqrt{\epsilon}} \leq \mathcal{N}\exp\bigg( \frac{-M\epsilon}{\frac{2C_0^2C_1}{c_\mathcal{H}}+4C_0S_0}\bigg)\,.
\end{align}
Setting the right-hand side to be $\frac{\delta}{2}$, we get $\epsilon_{M,\delta,\mathcal{N}}=\frac{C}{M}\log(\frac{2\mathcal{N}}{\delta})$, where 
$
C:=2\frac{C_0^2C_1}{c_\mathcal{H}}+4C_0S_0
$.
Using the inequality 
$
\sqrt{\epsilon_{M,\delta,\mathcal{N}}( \epsilon_{M,\delta,\mathcal{N}}+D_{\varphi_j,\infty})}\leq \epsilon_{M,\delta,\mathcal{N}} +\frac{1}{2}D_{\varphi_j,\infty},
$
we conclude that with probability at least $1-\frac{\delta}{2}$
$$
\mathcal{D}_{\intkernelvar_j,\infty}  -\mathcal{D}_{\intkernelvar_j,M} \leq \epsilon_{M,\delta,\mathcal{N}}+ \frac{1}{2}\mathcal{D}_{\intkernelvar_j,\infty}.
$$
\end{proof}

\begin{proof}[of Theorem \ref{maintheorem}]
For $\hypspace_n=\mathcal{B}_{S_0}^{\infty}(\mathcal{L}_n)$, let $\{\intkernelvar_j: j=1,\cdots, \mathcal{N}\}$ be an $\eta$-net of $\hypspace_n$. 
Let $$\widehat \intkernel_{T,M,\hypspace_n}=\mathrm{argmin}_{\intkernelvar \in \hypspace_n}\mathcal{E}_{T,M}(\intkernelvar).$$ Then there exists $\intkernelvar_{j_M}$ in the net such that 
$\|\intkernelvar_{j_M}-\widehat \intkernel_{T,M,\hypspace_n}\|_{\infty} \leq \eta$;
by Corollary \ref{continuousproperty}
\begin{align}\label{equation1}
\big|\mathcal{D}_{\intkernelvar_{j_M},\infty}- \mathcal{D}_{\widehat \intkernel_{T,M,\hypspace_n},\infty}\big| &= \big|\mathcal{E}_{T,\infty}(\intkernelvar_{j_M})- \mathcal{E}_{T,\infty}(\widehat \intkernel_{T,M,\hypspace_n})\big| \leq \eta \frac{2S_0R^2}{\sigma^2}.
\end{align}
{On the other hand, since $\hypspace_n \subset  \mathcal{L}_n \subset C^{1,1}([0,R])$, thanks to the almost sure control in Proposition \ref{randomdifference} and the uniformly bound $\sup_{\intkernelvar \in \mathcal{L}_n}\frac{\|\intkernelvar'\|_{\infty}}{\|\intkernelvar\|_{\infty} }\leq c_1(\dimsp(\mathcal{L}_n))^\gamma$ from the assumption \eqref{markovbernstein}, we have, almost surely, \\  }

\begin{equation}
\begin{aligned}
\big|\mathcal{D}_{\intkernelvar_{j_M},M}- \mathcal{D}_{\widehat \intkernel_{T,M,\hypspace_n},M}\big|
&=\big| \mathcal{E}_{T, M}(\intkernelvar_{j_M})- \mathcal{E}_{T,M}(\widehat \intkernel_{T,M,\hypspace_n})\big|\\
& \leq  \eta(C_1+c_1C_2 \mathrm{dim}(\mathcal{L}_n)^\gamma), 
\end{aligned}
\label{equation2}
\end{equation}  where $C_1 = \frac{ R^2S_0}{\sigma^2}+\frac{ R^2 }{2\sigma^2T}+\frac{d}{2}$, $C_2 =\frac{R}{2}$.

By Lemma \ref{covering}, for each $\eta>0$, with probability at least $1-\frac{\delta}{2}$, \eqref{coveringargument} holds for this $\eta$-net $\{\intkernelvar_j: j=1,\cdots,\mathcal{N}\}$. Combining \eqref{coveringargument} with \eqref{equation1} and \eqref{equation2}, we conclude that, with probability at least $1-\frac{\delta}{2}$, 
\begin{align*}
& \mathcal{D}_{\widehat \phi_{T,M,\mathcal{H}_n},\infty}- \mathcal{D}_{\widehat \phi_{T,M,\mathcal{H}_n},M} \\
\leq  & | \mathcal{D}_{\widehat \phi_{T,M,\mathcal{H}_n},\infty}- \mathcal{D}_{\intkernelvar_{j_M},\infty} | + |\mathcal{D}_{\intkernelvar_{j_M},\infty}-\mathcal{D}_{\phi_{j_M},M}|+ |\mathcal{D}_{\phi_{j_M},M}- \mathcal{D}_{\widehat \phi_{T,M,\mathcal{H}_n},M} |   \\
\leq &  \eta (C_0+ C_1+c_1C_2\mathrm{dim}(\mathcal{L}_n)^{\gamma} )+\mathcal{D}_{\phi_{j_M},\infty}- \mathcal{D}_{\phi_{j_M},M}\\
\leq &  \eta (C_0+ C_1+c_1C_2\mathrm{dim}(\mathcal{L}_n)^{\gamma} )+ \epsilon_{M,\delta,\mathcal{N}}+ \frac{1}{2} \mathcal{D}_{ \phi_{j_M}, \infty} \\ 
\leq & \eta (\frac{3C_0}{2}+ C_1+c_1C_2\mathrm{dim}(\mathcal{L}_n)^{\gamma} )+ \epsilon_{M,\delta,\mathcal{N}}+ \frac{1}{2} \mathcal{D}_{\widehat \phi_{T,M,\mathcal{H}_n},\infty}.
\end{align*} 
Notice that $\mathcal{D}_{\smash{\widehat \phi_{T,M,\mathcal{H}_n},M}}\leq 0$, so the above inequality  implies that 
\begin{align}\label{estimator1}
 \mathcal{D}_{\widehat \phi_{T,M,\mathcal{H}_n},\infty} \leq  \eta (3C_0+ 2C_1+2c_1C_2\mathrm{dim}(\mathcal{L}_n)^{\gamma} )+ 2\epsilon_{M,\delta,\mathcal{N}}. 
 \end{align} 
 The covering number of $\hypspace_n$ satisfies $\mathcal{N}(\hypspace_n, \eta) \leq \left(\frac{4S_0}{\eta}\right)^{c_0n}$(e.g. Proposition 5 in \cite{CS02}). 
By the triangle inequality, we split the error we want to control into Sampling Error (SE) and Approximation Error (AE) (see Figure \ref{f:diagram})
\begin{align} \label{errordecomposition}
\vertiii{\widehat \intkernel_{T,M,\hypspace_n}-\intkernel}^2\leq 2\vertiii{ \widehat \intkernel_{T,M,\hypspace_n}- \widehat\intkernel_{T,\infty,\hypspace_n}}^2+2\vertiii{\widehat \intkernel_{T,\infty,\hypspace_n}- \intkernel}^2. 
\end{align}
From  \eqref{estimator1} and the coercivity condition \eqref{coercivityinequality}, we obtain that, with probability at least $1-\frac{\delta}{2}$, 
\begin{align}\label{sampleerror}
&\vertiii{ \widehat \intkernel_{T,M,\hypspace_n}- \widehat\intkernel_{T,\infty,\hypspace_n}}^2 \leq \frac{1}{c_{\mathcal{H}_n}}\mathcal{D}_{\widehat \intkernel_{T,M,\hypspace_n}, \infty}\nonumber\\ &\leq \frac{\eta}{c_{\mathcal{H}_n}} (3C_0+ 2C_1+2c_1C_2\mathrm{dim}(\mathcal{L}_n)^{\gamma} )+ \frac{2}{c_{\mathcal{H}_n}}\epsilon_{M,\delta,\mathcal{N}}. 
\end{align}
Let $\phi_{\mathcal{H}_n}:=\mathrm{argmin}_{\psi\in \mathcal{H}_n}\|\psi-\intkernel\|_{\infty}$. By coercivity condition \eqref{coercivityinequality}, we have 
\begin{align*}
\vertiii{\widehat \intkernel_{T,\infty,\hypspace_n}- \intkernel_{\mathcal{H}_n}}^2&\leq \frac{1}{c_{\mathcal{H}_n}}(\mathcal{E}_{T,\infty}(\phi_{\mathcal{H}_n})-\mathcal{E}_{T,\infty}(\widehat \intkernel_{T,\infty,\hypspace_n}))\\
&\leq  \frac{1}{c_{\mathcal{H}_n}}(\mathcal{E}_{T,\infty}(\phi_{\mathcal{H}_n})-\mathcal{E}_{T,\infty}(\intkernel)) \leq \frac{1}{c_{\mathcal{H}_n}} \vertiii{\phi_{\mathcal{H}_n}-\intkernel}^2,
\end{align*} where we used
\begin{align*}
\mathcal{E}_{T,\infty}(\intkernel)&\leq \mathcal{E}_{T,\infty}(\widehat \intkernel_{T,\infty,\hypspace_n})\,,\quad
\big| \mathcal{E}_{T,\infty}(\intkernel)- \mathcal{E}_{T,\infty}(\intkernelvar)\big| &\leq \vertiii{\intkernel-\intkernelvar}^2\,, \quad\forall \intkernelvar \in \mathcal{H}_n.
\end{align*}
Therefore, we have
  \begin{align}\label{approxerror}
  \vertiii{\widehat \intkernel_{T,\infty,\hypspace_n}- \intkernel}^2
 \leq (2+\frac{2}{c_{\mathcal{H}_n}})\vertiii{\intkernel_{\hypspace_n}- \intkernel}^2 \leq \frac{4R^2}{c_{\mathcal{H}_n}}n^{-2s}.
\end{align}

Now we combine the estimates \eqref{errordecomposition}, \eqref{sampleerror} and \eqref{approxerror} together, and  let  $\eta=n^{-2s-\gamma}$ and $n=(\frac{M}{\log M})^{\frac{1}{2s+1}}$,  and  note that $c_{\mathcal{L}}=c_{\cup_n\mathcal{L}_n}=c_{\cup_n\mathcal{H}_n}\leq  c_{\mathcal{H}_n}=c_{\mathcal{L}_n}\leq 1$ for all $n$. We obtain that, with probability at least $1-\frac{\delta}{2}$, the following estimate holds true: 
\begin{align}\label{probestimates}
 \vertiii{\widehat \intkernel_{T,M,\hypspace_n}-\intkernel}^2& \leq \frac{2\eta}{c_{\mathcal{L}}}(3C_0+2C_1+2c_1C_2\mathrm{dim}(\mathcal{L}_n)^{\gamma})+\frac{8R^2}{c_{\mathcal{L}}}n^{-2s}+\frac{4}{c_{\mathcal{L}}}\epsilon_{M,\delta,N}\nonumber\\
  &\leq  \frac{C_3}{c_{\mathcal{L}}}n^{-2s}+\frac{C_4}{c_{\mathcal{L}}}\frac{n\log n}{M}+  \frac{4C}{c_{\mathcal{L}}M}\log(\frac{2}{\delta})\nonumber\\ &\leq \frac{C_5}{c_{\mathcal{L}}} \left(\frac{\log M}{M}\right)^{\frac{2s}{2s+1}}+\frac{4C}{c_{\mathcal{L}}M}\log({2}/{\delta}),
\end{align}
 where we used \eqref{coveringargument} to get $\epsilon_{M,\delta,N}$,  $C$, and $\{C_i\}_{i=0}^2$ is defined in  \eqref{coveringargument}, \eqref{equation1}, and \eqref{equation2} respectively, and 
 \begin{align*}
 C_3 &=6C_0+4C_1+4c_0^{\gamma}c_1C_2+8R^2\\
 C_4&=4c_0C|\log(4S_0)|+4c_0(2s+\gamma)C\\
 C_5&=C_3+\frac{C_4}{2s+1}. 
 \end{align*}
 
The bound in expectation is obtained by standard techniques, writing
\begin{align*}
\E \vertiii{\widehat \intkernel_{T,M,\hypspace_n}-\intkernel}^2=\int_{0}^{\infty}\Prob{ \vertiii{\widehat \intkernel_{T,M,\hypspace_n}-\intkernel}^2 >\epsilon}d\epsilon\,,
\end{align*}
and splitting the integration interval into $[0, \frac{C_5}{c_{\mathcal{L}}} (\frac{\log M}{M})^{\frac{2s}{2s+1}}]$ and  $[\frac{C_5}{c_{\mathcal{L}}} (\frac{\log M}{M})^{\frac{2s}{2s+1}}, \infty]$. On the first interval, we use $\Prob{ \vertiii{\widehat \intkernel_{T,M,\hypspace_n}-\intkernel}^2 >\epsilon}\le1$.  On the second one, we use a change of variables and the probability estimate \eqref{probestimates}. We obtain 
\begin{align}
\E \vertiii{\widehat \intkernel_{T,M,\hypspace_n}-\intkernel}^2 &\leq \frac{C_5}{c_{\cup_n\hypspace_n}}\left(\frac{\log M}{ M}\right)^{\frac{2s}{2s+1}}+\frac{4C}{c_{\cup_n\mathcal{H}_n}}\frac{1}{M}\nonumber\\ &\leq \frac{C_6}{c_{\cup_n\hypspace_n}^2}\left(\frac{\log M}{ M}\right)^{\frac{2s}{2s+1}}
\end{align} where $C_6$ is an absolute constant only depending on $\sigma, N, T, S_0,R$. 
\end{proof}
 
\begin{proof}[of Theorem \ref{main:consistency}] In this proof, $a \lesssim b$  means there there exist a constant $c$ such that $a\leq cb$. 
 For any $\epsilon>0$, we claim 
 $$\sum_{M=1}^{\infty}\Prob{\vertiii{\widehat{\intkernel}_{T,M,\mathcal{H}_M}-\intkernel}^2 \geq \epsilon} < \infty.$$
Strong consistency will then follow from the Borel-Cantelli Lemma. Notice that 
\begin{align*}
\Prob{ \vertiii{\widehat{\intkernel}_{T,M,\mathcal{H}_M}-\intkernel}^2 \geq \epsilon } 
&\leq \Prob { \vertiii{\widehat{\intkernel}_{T,M,\mathcal{H}_M}-\widehat{\intkernel}_{T,\infty,\mathcal{H}_M} }^2 \geq \frac{\epsilon}{2} }\\ 
&+\Prob{\vertiii{\widehat{\intkernel}_{T,\infty,\mathcal{H}_M}-\intkernel}^2 \geq \frac{\epsilon}{2} }\,, 
\end{align*} and 
$\Prob{\vertiii{\widehat{\intkernel}_{T,\infty,\mathcal{H}_M}-\intkernel}^2 \geq \frac{\epsilon}{2}}=0$ when $M$ is large enough (see \eqref{approxerror}). It suffices to prove 
$$
\sum_{M=1}^{\infty}\Prob {\vertiii{\widehat{\intkernel}_{T,M,\mathcal{H}_M}-\widehat{\intkernel}_{T,\infty,\mathcal{H}_M}}^2 \geq \epsilon } < \infty\,.
$$ 
Let $\{\intkernelvar_j\}_{j=1}^{\mathcal{N}}$ be an $\eta$ net for $\mathcal{H}_M$. Consider the event 
$$\Lambda_{\eta,M,\epsilon}=\{\max_{1\leq j \leq \mathcal{N}} \frac{1}{2}\mathcal{D}_{\intkernelvar_j,\infty} -\mathcal{D}_{\intkernelvar_j,M}\geq \epsilon  \}\,.$$  
The bound \eqref{maxprob} in Proposition \ref{covering} yields 
\begin{align}
\Prob{\Lambda_{\eta,M,\epsilon}} \lesssim \mathcal{N}\exp\big( -c_{\mathcal{H}_M} M\epsilon \big).
\end{align} Using the fact that there exists $j_M \in \{1,2\cdots,\mathcal{N}\}$ such that $\|\intkernel-\intkernelvar_{j_M}\|_{\infty}\leq \eta$, and following the same argument as in \eqref{equation1} and \eqref{equation2},  we obtain,
$$ 
\Prob{\frac{1}{2} \mathcal{D}_{\widehat{\intkernel}_{T,M,\mathcal{H}_M},\infty} -\mathcal{D}_{\widehat{\intkernel}_{T,M,\mathcal{H}_M},M}\gtrsim \eta n_{M}^{\gamma}+\epsilon} 
\leq  \Prob {\Lambda_{\eta,M,\epsilon}}  
\lesssim \mathcal{N}\exp\big( -c_{\mathcal{H}_M} M\epsilon \big)\,. 
$$ 
Notice that $\mathcal{D}_{\widehat{\intkernel}_{T,M,\mathcal{H}_M},M}\leq 0$ and 
$ \mathcal{D}_{\widehat{\intkernel}_{T,M,\mathcal{H}_M},\infty} \geq c_{\mathcal{H}_M} 
\vertiii{\widehat{\intkernel}_{T,M,\mathcal{H}_M}-\widehat{\intkernel}_{T,\infty,\mathcal{H}_M}}^2$, so that we have 
 $$ 
\Prob{c_{\mathcal{H}_M} \vertiii{\widehat{\intkernel}_{T,M,\mathcal{H}_M}-\widehat{\intkernel}_{T,\infty,\mathcal{H}_M}}^2 \gtrsim \eta n_{M}^{\gamma}+\epsilon } \lesssim \mathcal{N}\exp\big( -c_{\mathcal{H}_M} M\epsilon \big)\,. 
$$ 
Let $\eta n_{M}^{\gamma}=\epsilon$, i.e.,  $\eta= n_{M}^{-\gamma}\epsilon$,  by assumption, we have $c_{\mathcal{H}_M}\geq c_{\cup_M \mathcal{H}_M}>0$ and $\lim_{M\rightarrow \infty}\frac{n_M\log n_M}{M}= 0$, we have
\begin{align*}
\Prob{\vertiii{\widehat{\intkernel}_{T,M,\mathcal{H}_M}-\widehat{\intkernel}_{T,\infty,\mathcal{H}_M}}^2 \gtrsim \epsilon} 
& \lesssim \exp\big( n_{M} \log\frac{n_M}{\epsilon}-c_{\cup_M\mathcal{H}_M} M\epsilon \big)\\ 
&=\exp \bigg(-c_{\cup_M\mathcal{H}_M}M\epsilon\big(1- \frac{n_{M}}{M\epsilon} \log\frac{n_M}{\epsilon}\big)\bigg)\\ 
&\lesssim \exp \bigg(-\frac{1}{2}c_{\cup_M\mathcal{H}_M}M\epsilon \bigg)
\end{align*} 
when $M$ is large enough. By the comparison Theorem, the series 
$$
\sum_{M=1}^{\infty}\Prob {\vertiii{\widehat{\intkernel}_{T,M,\mathcal{H}_M}-\widehat{\intkernel}_{T,\infty,\mathcal{H}_M}}^2 \gtrsim \epsilon}$$ 
converges. The claim is proved. 
\end{proof}

\section{Learning theory: discrete-time observations}\label{sec:theory_dis_traj}
In this section, we analyze the estimation error of the (regularized) MLE $\widehat \intkernel_{L,T,M,\hypspace}$, defined in \eqref{MLE} for finite dimensional linear space $\hypspace$ and for discrete-time observations. We show that it is of order $\sqrt\frac nM + \Delta t^{1/2}$ with high probability, where $n$ is the dimension of $\hypspace$ and $\Delta t $ is the observation gap. As a consequence, the MLE is consistent when $M\to \infty$ and $\Delta t\to 0$; and the MLE converges at an optimal rate as when $n$ is optimally chosen as in \eqref{eq:H_dim}. 

We shall first prove the main theorems on the error of the MLE in Section \ref{sec:errBdMLE}, postponing  technical details, including concentration inequalities and discretization error bounds, to later subsections. 

Recall that we denote $\bX_{[0,T]}$ the solution to system \eqref{eq:sod} with the true interaction kernel $\phi$, and denote $\{\bX^{(m)}_{t_0:t_L}\}_{m=1}^M$ independent trajectories observed at discrete times $t_l= l\Delta t$ with $\Delta t = T/L$.  Recall that when $\hypspace = \mathrm{span}\{\psi_p\}_{p=1}^n$, the MLE $\widehat \intkernel_{L,T,M,\hypspace}$ is given in \eqref{eq:phiEst}. 

Throughout this section, we assume
\begin{assumption}[Basis functions]\label{assump_basis} Assume that the basis functions $\{ \psi_p \}_{p=1}^n \subset C^1_b(\R^+,\R) $ satisfy the following conditions: 
\begin{itemize}
\item[(a)] $\{\psi_p(\cdot)(\cdot)\}_{p=1}^n$ are orthonormal in $L^2(\rhoT)$;
\item[(b)] $\max_k \|\psi_p\|_\infty \leq b_0$, $\max_k \|\psi'_k(\cdot)(\cdot)\|_\infty \leq b_1$; 
\item[(c)] there exists a constant $c_{\rhoT}$ such that   $n\leq c_{\rhoT} \min(b_0^2R,b_1R^{3/2})$.
\end{itemize}
\end{assumption} 
Item $(a)$ aims to make the normal equations matrix nonsingular, as discussed in  Proposition \ref{prop:A_infty}. In item $(b)$, the uniform bound for the derivatives aims to control the discretization error due to discrete-time observations; the uniform boundless of the functions will be used for concentration inequalities. Item $(c)$ states that the number of such orthonormal basis functions are bounded by the measure $\rho_T$ and the uniform bounds of the functions and their derivatives.  Item $(c)$ often follows from $(a)-(b)$, with an intuition from examples including polynomials and trigonometric polynomials, and smoothed piecewise polynomials, if $r^2\rhoT(dr)$ is equivalent to the Lebesgue measure on an interval $[R_0,R]\subset \R^+$. Such an interval is where pairwise distances explores with a noticeable probability (see for example, in Figure \ref{f:SODH1_kernel} and Figure \ref{t:LJH1_kernel}). It exists in general when the initial distribution spreads out the pairwise distances or when the system is ergodic, since the density of $\rhoT$ is smooth and nonnegative on $\R^+$.

\subsection{Error bounds for the MLE}\label{sec:errBdMLE}
We show first that the $L^2(\rho_T)$ error of the estimator $\widehat \intkernel_{L,T,M,\hypspace}$ in \eqref{eq:phiEst} converges as $M\to \infty$ and $\Delta t : = T/L\to 0$, with high probability. 
\begin{theorem}[Error bounds for the MLE] \label{thm:error_discreteTime}
Let the hypothesis space be $\hypspace=\mathrm{span} \{\psi_i\}_{i=1}^n  $, where the set of functions $\{\psi_i\}_{i=1}^n$, satisfying Assumption \rmref{assump_basis}, are orthonormal in $L^2(\rhoT)$ with respect to the norm $\vertiii{\cdot}$ defined in Eq.\eqref{eq:newnorm}. 

Suppose that the coercivity condition holds on $\hypspace$ with a constant $c_\hypspace>0$.  Then, with a probability at least $1-{(4n+2n^2) }\exp{\left(-\frac{\epsilon^2 }{8c_1^2} \right)}$, the error of the estimator $\widehat \intkernel_{L,T,M,\hypspace}$ in \eqref{eq:phiEst} satisfies
\begin{align}
\vertiii{\widehat \intkernel_{L,T,M,\hypspace} - \intkernel}  \leq  \vertiii{\widehat\intkernel_{T, \infty, \hypspace} - \intkernel}  + c_2\left(\sqrt\frac nM\epsilon+ c_3\Delta t ^{\frac{1}{2}}\right),
\end{align}
where $\widehat\intkernel_{T, \infty, \hypspace}  $ is the projection of the true kernel to $\hypspace$, $\Delta t = T/L \leq c_\hypspace/(2c_3)$, and the constants are 
\begin{equation} \label{eq:constants_c123}
\begin{aligned}
c_1 &= R b_0 (R b_0+2\sigma/\sqrt{T}),\quad \\
 c_2 & = 4c_\hypspace^{-1} (1+  \vertiii{\intkernel_{T, \infty, \hypspace}} ),\\
c_3 &=  dN(b_1+ b_0) Rb_0 (Rb_0  \Delta t ^{\frac{1}{2}} + \sqrt{d}\sigma)\sqrt{c_{\rhoT} \min(b_0^2R,b_1R^{3/2})}.\\
\end{aligned}
\end{equation}  
\end{theorem}
 \begin{remark}[The discretization error may dominate the statistical error] \label{rmk_flatLearningCurve} When the observation gap $\Delta t=0$, we recover the min-max learning rate $M^{-\frac{s}{2s+1}}$ proved in the previous section, if we choose the optimal dimension $n=C ( M/\log{M})^{1/(2s+1)}$ for the hypothesis space. However, when $\Delta t>0$, once $M^{-\frac{s}{2s+1}}(\log{M})^{-\frac{1}{2(2s+1)}} \lesssim \Delta t ^{\frac{1}{2}}$, the discretization error will dominate the error of the estimator, preventing us from observing the min-max learning rate. This phenomenon is well-illustrated by the left plots in Figure \ref{t:ODH1_Convergence_Plot} and \ref{t:LJH1_Convergence_Plot} in our numerical experiments. 
 \end{remark}
 \begin{remark} We assumed $C^1_b$ regularity for the basis functions $\{\psi_p\}$ for  the above numerical error analysis, stronger than that of piecewise polynomials (which may be discontinuous) used in the numerical tests. Such a difference between the regularity requirements stems from the numerical representation, and we can view the piecewise polynomials as numerical approximations of regular functions. This view is supported by the numerical experiments: the estimator has only small jumps at the discontinuities in the high probability region.   
 \end{remark}
\begin{remark}
A smaller coercivity constant $c_\hypspace$ corresponds to a worse conditioner problem (Proposition \ref{prop:A_infty}), and so the condition $L\gtrsim 1/c_\hypspace$ that requires $L$ to be larger for small $c_\hypspace$ makes sense.
\end{remark} 
 
The error of the MLE $\widehat \intkernel_{L,T,M,\hypspace} $ consists of three parts:  approximation error, discretization error and sampling error: 
\begin{align}\label{eq:error_decomposition}
& \, \vertiii{\widehat \intkernel_{L,T,M,\hypspace} - \intkernel} \notag \\
 \leq &  \underbrace{ \vertiii{ \widehat\intkernel_{T, \infty, \hypspace}  - \intkernel} }_{\text{ Approximation Error} }  + \underbrace{ \vertiii {  \widehat\intkernel_{L,T,\infty,\mathcal{H}} - \widehat\intkernel_{T,\infty,\hypspace} } }_{\text{Discretization Error}} 
  + \underbrace{ \vertiii {\widehat \intkernel_{L,T,M,\hypspace} - \widehat\intkernel_{L,T,\infty,\hypspace} } }_{\text{Sampling Error}},
 \end{align}
where $  \widehat\intkernel_{L,T,\infty,\mathcal{H}} $ is the infinite-data estimator. We shall study the discretization error and the sampling error by analyzing the differences between their coefficient vectors. 

All these coefficients are solutions to the corresponding normal equations (e.g. Eq.\eqref{e:ALM}). To facilitate the study of these normal matrices and vectors, we introduce the following notions.
For any $f, g\in C^1_b(\R^{Nd},\R^{Nd})$, we define the following functionals of the observation paths:
\begin{equation}\label{eq:xi_eta}
\begin{aligned}
\xi(f, \bX_{t_0:t_L}) &=  \frac{1}{TN} \sum_{l=1}^{L-1} \innerp{f(\bX_{t_l})}{ \bX_{t_{l+1} } -\bX_{t_{l} }}, \\
 \eta(f, g,X_{t_0:t_L}) &=  \frac{1}{LN} \sum_{l=1}^{L-1} \innerp{f(\bX_{t_l})}{g(\bX_{t_l}) }, \\
 \xi(f, \bX_{[0,T]}) &=  \frac{1}{TN} \int_0^T \innerp{f(\bX_t)}{d\bX_t},  \\ 
  \eta(f, g, \bX_{[0,T]}) &=  \frac{1}{TN} \int_0^T \innerp{f(\bX_t)}{g(\bX_t)}dt. 
\end{aligned}
\end{equation}
Correspondingly, we define the empirical functionals
\begin{equation}\label{eq:xi_etaML}
\begin{aligned}
\xi_{M,L}(f)  &=  \frac{1}{M}\sum_{m=0}^M\xi(f, \bX_{t_0:t_L}^{(m)}),   & \eta_{M,L}(f, g)  &=  \frac{1}{M}\sum_{m=0}^M\eta(f,g, \bX_{t_0:t_L}^{(m)}), \\
\xi_{M,\infty}(f)  &=  \frac{1}{M}\sum_{m=0}^M\xi(f, \bX_{[0,T]}^{(m)}),   & \eta_{M,\infty}(f, g) &=  \frac{1}{M}\sum_{m=0}^M\eta(f,g, \bX_{[0,T]}^{(m)}),  
\end{aligned}
\end{equation}
Using the notation of empirical functional introduced in \eqref{eq:xi_eta}-\eqref{eq:xi_etaML}, we consider the following normal matrixes and vectors: 
\begin{equation}\label{eq:Ab_ML}
\begin{aligned}
b_{M,L}(k) & = \xi_{M,L}(f_{\psi_p})  ,   & A_{M,L}(k,k') &=  \eta_{M,L}(f_{\psi_p}, f_{\psi_{k'} }) , \\
b_{\infty,L}(k) & = \E[ \xi(f_{\psi_p}, \bX_{t_0:t_L})] ,  & A_{\infty,L}(k,k') & = \E[\eta(f_{\psi_p}, f_{\psi_{k'} }, \bX_{t_0:t_L})],  \\
b_{\infty}(k) & = \E[ \xi(f_{\psi_p}, \bX_{[0,T]})] ,  & A_{\infty}(k,k') & = \E[\eta(f_{\psi_p}, f_{\psi_{k'} }, \bX_{[0,T]})].
\end{aligned}
\end{equation}
It is clear that  (here, to ease the notation, we denote the coefficient $ \widehat a_{L,T,M,\hypspace}$ in Figure \ref{f:diagram} as $a_{M,L}$, and similarly for others)
\begin{equation}\label{eq:Ainvb's}
\begin{aligned}
\widehat\intkernel_{L,T,M,\hypspace} &=\sum_{i=1}^n a_{M,L}(i)\psi_i, &\text{ with } a_{M,L} & = A_{M,L}^{-1} b_{M,L} ,  \\ 
 \widehat\intkernel_{L,T,\infty,\mathcal{H}} &=\sum_{i=1}^n a_{\infty,L}(i)\psi_i, &\text{ with } a_{\infty,L} & = A_{\infty,L}^{-1} b_{\infty,L} ,\\  
 \widehat\intkernel_{T,\infty,\hypspace} &=\sum_{i=1}^n a_\infty(i)\psi_i,                      &\text{ with } a_{\infty} &= A_{\infty}^{-1}b_{\infty}. 
\end{aligned}
\end{equation}
Here the matrix $A_\infty$ is invertible due the coercivity condition: its smallest eigenvalue is the coercivity constant $c_\hypspace$ (see  Proposition \ref{prop:A_infty}). The matrix $A_{\infty,L}$ is invertible when $L= T/(\Delta t)$ is large, with its smallest eigenvalue bounded below by $c_\hypspace- c_3 \Delta t^{1/2}$, see Corollary \ref{cor_AmatL}.  Furthermore,  Corollary \ref{corollary:eignvalue} shows that, with probability at $1-\delta$, the matrix $A_{M,L}$ is invertible with its smallest eigenvalue bounded blow by $c_\hypspace - \left(\sqrt\frac nM\epsilon+ c_3\Delta t ^{\frac{1}{2}}\right)$.

\begin{proof}[of Theorem \rmref{thm:error_discreteTime} ] 
By Eq.\eqref{eq:error_decomposition}, it suffices to prove upper bounds for the discretization error and the sampling error separately:
\begin{align*}\label{def:err_Kernel_est}
&\text{discretization error:} &  \vertiii {  \widehat\intkernel_{L,T,\infty,\mathcal{H}} - \widehat\intkernel_{T,\infty,\hypspace} } & \leq  c_2c_3\Delta t^{1/2} , \\
& \text{sampling error:}  & \vertiii {\widehat \intkernel_{L,T,M,\hypspace} - \widehat\intkernel_{L,T,\infty,\hypspace}   } &\leq  c_2\sqrt\frac nM\epsilon, 
 \end{align*}
 where the second inequality holds  with probability at least $1-\delta$. 
 
For the discretization error, since $\{\psi_i(\cdot)(\cdot)\}$ are orthonormal in $L^2(\rhoT)$, we have, by Eq.\eqref{eq:Ainvb's}: 
 \begin{align*}  
\vertiii{  \widehat\intkernel_{L,T,\infty,\mathcal{H}} - \widehat\intkernel_{T,\infty,\hypspace} }^2 & =   \vertiii{ \sum_{i=1}^n [ a_{\infty,L}(i)  - a_\infty(i ) ] \psi_i }^2   \\
 &  =  \norm{a_{\infty,L} - a_\infty}^2 = \|A_{\infty,L}^{-1} b_{\infty,L} - A_\infty^{-1}b_{\infty}\|^2. 
  \end{align*}
Using the formula $A_{\infty,L}^{-1} -A_{\infty}^{-1}  = A_{\infty,L}^{-1}(A_\infty - A_{\infty,L})A_\infty^{-1} $ (see e.g. \cite[Appendix B9]{higham2008functions}), we have 
\begin{align*}
& \norm{A_{\infty,L}^{-1}b_{\infty,L} - A_\infty^{-1} b_\infty} = \norm{A_{\infty,L}^{-1}(b_{\infty,L}-b_\infty) + (A_{\infty,L}^{-1} -A_\infty^{-1}) b_\infty}  \nonumber \\
 \leq & \norm{A_{\infty,L}^{-1}} \left( \norm{b_{\infty,L}-b_\infty} + \norm{A_\infty - A_{\infty,L}} \norm{A_\infty^{-1} b_\infty} \right). 
\end{align*}
Note that (i) $\norm{A_{\infty,L}^{-1}}\leq 2 c_\hypspace^{-1}$, since $c_3\Delta t ^{1/2} < \frac{1}{2}c_\hypspace$; (ii) by Proposition \ref{prop:numError} in combination of Assumption \ref{assump_basis},  
\begin{align*}
\norm{b_{\infty,L}-b_\infty}  \leq c_3  \Delta t ^{1/2}, \quad  \norm{A_{\infty,L} - A_\infty } \leq c_3  \Delta t ^{1/2};
\end{align*}
and (iii) $\norm{A_\infty^{-1} b_\infty} = \vertiii{ \widehat\intkernel_{T,\infty,\hypspace}}$.   Then, 
\begin{equation}\label{eq:numErrorAb}
 \norm{A_{\infty,L}^{-1}b_{\infty,L} - A_\infty^{-1} b_\infty}  \leq 2c_\hypspace^{-1}(1+\vertiii{ \widehat\intkernel_{T,\infty,\hypspace}}) c_3 \Delta t ^{1/2},
\end{equation}
and the inequality for the discretization error follows.

Similarly, for the sampling error, we have 
\begin{align*}
  & \vertiii{ \widehat \intkernel_{L,T,M,\hypspace} - \widehat\intkernel_{L,T,\infty,\hypspace}  }            =  \norm{a_{M,L} - a_{\infty,L} }  
    =  \|A_{M,L}^{-1}b_{M,L} - A_{\infty,L}^{-1} b_{\infty,L} \|  \\
    = &\norm{A_{M,L}^{-1}(b_{M,L}-b_{\infty,L}) + (A_{M,L}^{-1} -A_{\infty,L}^{-1}) b_{\infty,L} }  \\
 \leq & \norm{A_{M,L}^{-1}} \left( \norm{b_{M,L}-b_{\infty,L}} + \norm{A_{M,L}^{-1} -A_{\infty,L}^{-1}} \norm{A_{\infty,L}^{-1} b_{\infty,L}} \right). 
\end{align*}
Note that (i) $\norm{A_{M,L}^{-1}}\leq 2 c_\hypspace^{-1}$ when $M$ and $L$ are large enough such that  $\sqrt\frac nM\epsilon+ c_3\Delta t ^{\frac{1}{2}} < \frac{1}{2}c_\hypspace $; (ii) we have, by Proposition \ref{prop:CI},  that
\begin{align*}
\norm{b_{M,L}-b_{\infty,L}}     \leq \sqrt\frac nM\epsilon,  \quad  \norm{ A_{M,L} - A_{\infty, L} }  \leq \sqrt\frac nM\epsilon 
\end{align*}
hold with probability at least $1-\delta$; (iii)  since $ c_3\Delta t ^{\frac{1}{2}} < \frac{1}{2}c_\hypspace $ and $\norm{A_{\infty}^{-1} b_{\infty}} = \vertiii{ \widehat\intkernel_{T,\infty,\hypspace}}$, we have
\[ 
\norm{A_{\infty,L}^{-1} b_{\infty,L}} \leq  2 (1+\vertiii{ \widehat\intkernel_{T,\infty,\hypspace}}). 
\]
Then,  the inequality for the sampling error follows.
\end{proof}

\subsection{Concentration and discretization error of empirical functionals}
We introduce concentration inequalities for the above empirical functionals on the path space of diffusion processes and  a bound on the discretization error of the estimator due to discrete-time approximations. 
 Our first lemma studies concentration of the discrete-time empirical functionals $\xi_{M,L}$ and $\eta_{M,L}$. 

\begin{lemma}[Concentration of empirical functionals] \label{lemma:CI}
Let $\{\bX^{(m)}_{t_0:t_L}\}_{m=1}^M$ be discrete-time observations, with $t_l= l\Delta t$ and $L= T/\Delta t$, of the system \eqref{eq:sod} with $\phi$. Then for any $f,g\in C_b(\R^{dN}, \R^{dN} )$, the error functionals defined in \eqref{eq:xi_etaML} satisfy the concentration inequalities:  
\begin{equation} \label{ConcIneq_xi_eta} 
\begin{aligned}
\Prob{|\xi_{M,L}(f) - \E[\xi_{M,L}(f)]|> \epsilon} &\leq 4 e^{-\frac{M\epsilon^2 }{8C_1^2}} \\
\Prob{|\eta_{M,L}(f,g) - \E[\eta_{M,L}(f,g)]|> \epsilon} &\leq 2 e^{-\frac{M\epsilon^2 }{2C_2^2}}, 
\end{aligned}
\end{equation}
for any $\epsilon >0$, where $C_1 = \frac{1}{N}\|f\|_\infty \max(\frac{2\sigma \sqrt{N}}{\sqrt{T}}, \, \|f_\phi\|_\infty)$, and $C_2 =\frac{1}{N} \|f\|_\infty \|g\|_\infty$. 
 Furthermore, 
\begin{align}\label{ineq_xi7eta}
 \Prob{|\xi_{M,L}(f) - \E[\xi_{M,L}(f)]| < \epsilon,\, |\eta_{M,L}(f,g) - \E[\eta_{M,L}(f,g)]|< \epsilon} &\geq 1- 8 e^{-\frac{M\epsilon^2 }{8C^2}}, 
 \end{align}
 where $C=  \frac{1}{N}\|f\|_\infty \max(\frac{2\sigma }{\sqrt{T/N}},  \|f_\phi\|_\infty,\, \|g\|_\infty)$. 
\end{lemma}
\begin{proof}
Note that $\verti{\eta(f, g,\bX_{[t_0:t_L]})} \leq \|f\|_\infty\|g\|_\infty$. Then, the exponential inequality for $\eta_{ML}$ follows from the Hoeffding inequality, which states that, for i.i.d.~random variables $\{Z_i\}$ bounded above by $K$, one has $$\Prob{\verti{\frac{1}{M}\sum_{m=1}^M (Z_i-\E Z_i) } > \epsilon} \leq 2 \exp{(-\frac{M\epsilon^2}{2K^2}) }.$$

To study $\xi_{ML}$,  we decompose $\xi(f,\bX_{[t_0:t_L]})$ into two parts, a bounded part and a martingale part:
\begin{align*}
\xi(f,\bX_{[t_0:t_L]})  = & \frac{1}{TN} \sum_{l=0}^{L-1} \innerp{f(\bX_{t_l}) \mathbf{1}_{[t_l,t_{l+1}]}(s)}{\bX_{t_{l+1}} - \bX_{t_l}} \\ 
= & \frac{1}{TN} \int_0^T\innerp{f^L(s)}{f_\phi(\bX_s)}ds  +  \frac{1}{TN} \int_0^T\innerp{f^L(s)}{\sigma d\bB_s} =: Z_T + Y_T,
\end{align*}
where we denote $f^L(s) :=\sum_{l=0}^{L-1}  f(\bX_{t_l}) \mathbf{1}_{[t_l,t_{l+1}]}(s) $. We call $Z_T$ a bounded part because 
\[ \verti{Z_T}=\verti{ \frac{1}{TN} \int_0^T\innerp{f^L(s)}{f_\phi(\bX_s)}ds}\leq \frac{1}{N} \|f\|_\infty  \|f\|_\infty.\] 
We call $Y_T$ a martingale part since $TNY_T= \int_0^T\innerp{f^L(s)}{\sigma d\bB_s}$ is a martingale. Correspondingly, we can write 
\[
\xi_{ML} = \frac{1}{M}\sum_{m=1} Z_T^{(m)} + Y_T^{(m)}.
\]
Then, denoting $K_1:=\frac{1}{N} \|f\|_\infty  \|f_\phi\|_\infty$ and $K_2= 2\sigma \|f\|_{\infty}$, and noticing that $C_1= K_1 + K_2 /\sqrt{TN}$, 
we can conclude the first concentration inequality in \eqref{ConcIneq_xi_eta} from 
\begin{align*}
  & \Prob{|\xi_{M,L}(f) - \E[\xi_{M,L}(f)]|> \epsilon} \\
\leq \, &  \Prob{\verti{\frac{1}{M}\sum_{m=1} Z_T^{(m) }- \E Z_T }\geq \frac{\epsilon}{2} } + \Prob{ \verti{\frac{1}{M}\sum_{m=1} Y_T^{(m)} }\geq \frac{\epsilon}{2}}\\
 \leq \,& 2 e^{-\frac{M\epsilon^2}{2K_1^2}} + e^{-\frac{TNM\epsilon^2}{8K_2^2}} ,
\end{align*}
where the first exponential bound follows directly from Hoeffding inequality applied to $\{Z_T^{(m)}\}$, and the second exponential bound $\Prob{ \verti{\frac{1}{M}\sum_{m=1} Y_T^{(m)} }\geq \frac{\epsilon}{2}} \leq e^{-\frac{M\epsilon^2}{8K_2^2}}$ is proved as follows. 

Note that $\E Y_T=0$ and $TNY_T= \int_0^T\innerp{f^L(s)}{\sigma d\bB_s}$ is a martingale 
satisfying  $\E [e^{\sigma^2\int_0^T\verti{ f^L(s)}^2d_s}] < \infty$ because $\verti{ f^L(s)}^2\leq \|f\|_\infty$. 
 By the Novikov theorem, the process $( e^{\alpha TNY_T - \frac{\alpha^2}{2} \sigma^2 \int_0^T \verti{ f^L(s)}^2d_s }, T\geq 0)$ is a martingale for any $\alpha\in \R$ (see e.g. \cite[Corollary 5.13]{karatzas1998brownian}). Therefore, with $\alpha = \frac{\lambda}{MTN}$, we have
\[
\E\left[e^{\frac{\lambda}{M}Y_T} \right] = \E\left[e^{ \sigma^2\frac{\lambda^2}{(MTN)^2} \int_0^T \verti{ f^L(s)}^2d_s} 
 \right] \leq e^{\sigma^2\frac{\lambda^2}{M^2TN} \|f\|_\infty^2}
\]
for any $\lambda>0$. As a consequence, we have 
\[
\Prob{ \verti{\frac{1}{M}\sum_{m=1} Y_T^{(m)} }\geq \frac{\epsilon}{2} } 
\leq  \inf_{\lambda>0} e^{-\frac{\lambda \epsilon}{2}} \E\left[ e^{\frac{\lambda}{M}Y_T} \right] \leq  \inf_{\lambda>0}  e^{- \frac{\lambda \epsilon}{2} + \lambda^2\frac{\sigma^2\|f\|_\infty}{TNM}} \leq e^{-\frac{TNM\epsilon^2}{8K_2^2}}. 
\]

Lastly, Eq.\eqref{ineq_xi7eta} follows directly from Eq.\eqref{ConcIneq_xi_eta}. 
\end{proof}
We remark that here we focus on the case $M\to \infty$ with finite time $T$. If the system is ergodic, one may extends the concentration inequalities to the case when $T\to \infty$.

The next lemma shows that the discretization error of the empirical functionals, as discrete-time approximation of the integrals, is at order $\Delta t^ \frac{1}{2}$. 
\begin{lemma}[Discretization error of empirical functionals] \label{lemma:xi_numError}
Let $f,g\in C^1_b(\R^{dN}, \R^{dN}) $.  Let $\bX_{t_0:t_L}$ be a discrete-time trajectory, with $t_l=l\Delta t$ and $L= T/\Delta t$, of the system \eqref{eq:sod} with $\phi$. Then, the error functionals defined in \eqref{eq:xi_eta} satisfy
\begin{equation} \label{eq:xi_etaBD}
\begin{aligned}
\verti{  \E[ \xi(f, \bX_{t_0:t_L})]  - \E[\xi(f, \bX_{[0,T]})]          }  & \leq  C_1  \Delta t^\frac{1}{2}; \\
\verti{ \E[\eta(f,g, \bX_{t_0:t_L})] - \E[\eta(f,g, \bX_{[0,T]})]   }& \leq  C_2  \Delta t^\frac{1}{2},
\end{aligned}
\end{equation}
where the constants are 
\begin{align*}
C_1&= \|\nabla f\|_\infty  \|f_\phi\|_\infty  \left( \|f_\phi\|_\infty \Delta t^{\frac{1}{2}}/N  + \sqrt{d/N}\sigma\right), \\
C_2&=( \|\nabla f\|_\infty \|g\|_\infty + \|\nabla g\|_\infty \|f\|_\infty )  \left( \|f_\phi\|_\infty \Delta t^{\frac{1}{2}}/N  + \sqrt{d/N}\sigma\right). 
\end{align*}
\end{lemma}
\begin{proof}
Note that since $\bX_{[0,T]}$ is a solution to the system \eqref{eq:sod}, we have for each $l$,  
\begin{align*}
 \E \int_{t_l}^{t_{l+1}}\innerp{f(\bX_{t_l}) - f(\bX_s)} { d\bX_s}
=&\, \E \int_{t_l}^{t_{l+1}}\innerp{f(\bX_{t_l}) - f(\bX_s)} { f_\phi(\bX_s)}ds \\
\leq &\,\|\nabla f\|_\infty \|f_\phi\|_\infty \E \int_{t_l}^{t_{l+1}} |\bX_{t_l} - \bX_s|ds \\
\leq &\, \|\nabla f\|_\infty \|f_\phi\|_\infty  \left( \|f_\phi\|_\infty \Delta t ^2 + \sqrt{dN}\sigma \Delta t^{3/2} \right), 
\end{align*}
where in the first inequality we have applied the mean value theorem to bound $f(\bX_{t_l}) - f(\bX_s)$:
\[
|f(\bX_{t_l}) - f(\bX_s)|\leq  \| \nabla f\|_\infty  |\bX_{t_l} - \bX_s|, 
\]
and in the second inequality, we used the fact that
\begin{align*}
  \E |\bX_{t_l} - \bX_s| & = \E \left |\int_{t_l}^s f_\phi(\bX_r))dr +\sigma (\bB_s - \bB_{t_l}) \right | \\
  & \leq  \|f_\phi\|_\infty (s-t_{l}) +  \sqrt{dN}\sigma   (s-t_{l})^{1/2}. 
\end{align*}
Thus, we obtain the bound in \eqref{eq:xi_etaBD} by a summation over $l$: 
\begin{align*}
 \E[\xi(f, \bX_{t_0:t_L}) - \xi(f, \bX_{[0,T]}) ] &= \frac{1}{TN} \sum_{l=1}^{L-1} \E\int_{t_l}^{t_{l+1}}\innerp{f(\bX_{t_l}) - f(\bX_s)} { d\bX_s}, \\
 &\leq \|\nabla f\|_\infty  \|f_\phi\|_\infty  \left( \|f_\phi\|_\infty \Delta t /N + \sqrt{d/N}\sigma \Delta t^{1/2} \right). 
\end{align*}

Similarly, we have 
\[
\verti {\innerp{f(\bX_{t_l})}{ g(\bX_{t_l})} - \innerp{f(\bX_s)}{g(\bX_s) }  } \leq (\|\nabla f\|_\infty \|g\|_\infty + \|\nabla g\|_\infty \|f\|_\infty) \verti{\bX_s-\bX_{t_l}},
\]
and the bound for $\eta$ follows from the fact that 
\begin{align*}
 \E[\eta(f,g, \bX_{t_0:t_L}) &- \eta(f, g, \bX_{[0,T]}) ] \\ 
 =\, &\frac{1}{TN} \sum_{l=1}^{L-1} \E\int_{t_l}^{t_{l+1}}\verti {\innerp{f(\bX_{t_l})}{ g(\bX_{t_l})} - \innerp{f(\bX_s)}{g(\bX_s) }  } ds. 
\end{align*} 
\end{proof}

\subsection{Error bounds for the normal matrix and vector}

\begin{proposition}[Discretization error]\label{prop:numError}
For the normal matrix $A_{\infty,L}$ and vector $b_{\infty,L}$ defined in \eqref{eq:Ab_ML} with $\{ \psi_p \}_{p=1}^n$ satisfying Assumption {\rm \ref{assump_basis}},
we have 
\begin{align*}
\|b_{\infty,L} - b_{\infty}\|  \leq \sqrt{n} C  \Delta t ^{\frac{1}{2}}\,,\quad
\|A_{\infty,L} - A_{\infty}\| \leq \sqrt{n} C \Delta t ^{\frac{1}{2}},
\end{align*}
where the constant $C$ is 
$ C = dN(b_1+ b_0) Rb_0 (Rb_0  \Delta t ^{\frac{1}{2}} + \sqrt{d}\sigma) $. 
\end{proposition} 
\begin{proof} Applying Lemma \ref{lemma:xi_numError}, in combination of the basic fact that $\|b\| \leq \sqrt{n} \max_{k=1,\dots,n} |b(k)|$ for any $b\in \R^n$, and  $\|A\| \leq \sqrt{n} \max_{k,k'=1,\dots,n} |A(k,k')|$ for any $A\in \R^{n\times n}$, we obtain
\begin{align*}
\|b_{\infty,L} - b_{\infty}\|  \leq \sqrt{n} C_1  \Delta t ^{\frac{1}{2}}\,,\quad
\|A_{\infty,L} - A_{\infty}\| \leq \sqrt{n} C_2 \Delta t ^{\frac{1}{2}},
\end{align*}
  with constants $C_1$ and  $C_2$ in  form of
\begin{align*}
C_1&=\|f_\phi\|_\infty  \left( \|f_\phi\|_\infty \Delta t^{\frac{1}{2}}/N  + \sqrt{d/N}\sigma\right)  \max_{k=1,\cdots, n} \|\nabla f_{\psi_p}\|_\infty , \\
C_2&=   \left( \|f_\phi\|_\infty \Delta t^{\frac{1}{2}}/N  + \sqrt{d/N}\sigma\right)  \max_{k,k'=1,\cdots, n} ( \|\nabla f_{\psi_p}\|_\infty \|f_{\psi_p'}\|_\infty + \|\nabla f_{\psi_{k'} }\|_\infty \|f_{\psi_p}\|_\infty ) . 
\end{align*}
To complete the proof, we are left to estimate $\norm{  f_{\psi_p} }_\infty $ and $ \norm{  \nabla f_{\psi_p} }_\infty $. From the definition of $f_\cdot$, we have
\begin{equation}\label{eq:f_phiBD}
\norm{  f_{\psi_p} }_\infty^2 = \sup_x\sum_{i=1}^N \verti{\frac{1}{N}\sum_{j=1}^N \psi_p(|\bX_j-\bX_i|)(\bX_j-\bX_i)}^2 \leq R^2 b_0^2 N,
\end{equation}
and $\norm{  f_{\phi} }_\infty \leq R b_0\sqrt{N}$ as well. Note that for each $i,i'\in \{1,\cdots, N\}$,  with notation $\br_{ji}= \bx_j-\bx_i$ and $r_{ji}= \verti{\br_{ji} }$,  we have,
\begin{align*} 
\nabla_{\bx_i'} \left(f_{\psi_p}(\bx) \right)_i &  = \delta_{ii'} \frac{1}{N} \sum_{j=1, j\neq i}^N   \left(\psi_p(r_{ji})\mathbf{I}_d+\psi_p'(r_{ji}) \frac{\br_{ji}\otimes\br_{ji} }{r_{ji}}\right) \\
& \quad  + \delta_{i'\neq i} \frac{1}{N}\left(\psi_p(r_{ii'})\mathbf{I}_d+\psi_p'(r_{ii'}) \frac{\br_{ii'}\otimes\br_{ii'} }{r_{ii'}}\right) . 
\end{align*}
Thus,  the norm this $d\times d$ matrix is uniformly bounded,
\[
\sup_x \norm{ \nabla_{\bx_i'} \left(f_{\psi_p}(\bx) \right)_i  } \leq d (b_1 + b_0), 
\]
and as a result, the norm of the $dN \times dN$ matrix is uniformly bounded, 
\[
\norm{ \nabla f_{\psi_p}   }_\infty \leq dN (b_1  + b_0). 
\]
Combining the above estimates with $\| f_\phi\|_\infty\le Rb_0N$ (the same as $\| f_{\psi_p}\|_\infty $), we obtain that $C_1$ and $C_2$ are both bounded by $C$. 
\end{proof}

It follows directly that the matrix $A_{\infty,L}$ is invertible:
\begin{corollary}\label{cor_AmatL}

The smallest eigenvalue of the matrix $A_{\infty,L}$ defined in \eqref{eq:Ab_ML} is bounded below by $c_\hypspace- c_3\Delta t^{1/2}$ when $c_3 \Delta t^{1/2}< c_\hypspace$, with $c_3$ defined in \eqref{eq:constants_c123}.
\end{corollary} 
\begin{proof} Recall that from Proposition \ref{prop:A_infty}, we have $a^T A_{\infty} a \geq c_{\hypspace} |a|^2$ for an arbitrary $a\in \R^n$. Then,
\[
a^T A_{\infty,L} a = a^T (A_{\infty,L}-A_\infty) a + a^T A_{\infty} a \geq (c_{\hypspace} - c_2\Delta t^{1/2}) \|a\|^2
\]
by Proposition \ref{prop:numError} with the bound of $\sqrt{n}$ in Assumption \ref{assump_basis}. 
\end{proof}

We prove next that the matrix $A_{M,L}$ is invertible, and concentrates around  $A_{\infty, L}$. 
\begin{proposition}[Concentration of the normal matrix and vector] \label{prop:CI}
Suppose that the coercivity condition holds on $\mathcal{H}=\mathrm{span} \{\psi_i\}_{i=1}^n$ with a constant $c_{\hypspace}>0$, where $\{ \psi_p \}_{p=1}^n$ satisfying Assumption {\rm \ref{assump_basis}}. 
Then, the normal matrix $A_{M,L}$ and vector $b_{M,L}$ defined in \eqref{eq:Ab_ML} satisfy concentration inequalities in the sense that for any $\epsilon >0$, 
\begin{equation}  
\begin{aligned}
\Prob{\norm{ A_{M,L} - A_{\infty,L} }> \epsilon} &\leq 2n^2 e^{-\frac{M\epsilon^2 }{2nC^2}} \\
\Prob{\norm{ b_{M,L} - b_{\infty,L} }> \epsilon} &\leq 4ne^{-\frac{M\epsilon^2 }{8nC^2}},  \\
\Prob{\norm{ A_{M,L} - A_{\infty,L} }< \epsilon, \, \norm{ b_{M,L} < b_{\infty,L} } < \epsilon} &\geq 1- (4n+2n^2) e^{-\frac{M\epsilon^2 }{8nC^2}}\,, 
\end{aligned}
\end{equation}
where the constant $C$ is $C= R b_0 (R S_0+2\sigma/\sqrt{T})$.  
\end{proposition}
\begin{proof} Recall that by definition in \eqref{eq:Ab_ML}, $b_{M,L}(k)  = \xi_{M,L}(f_{\psi_p}) $ with $  \E[b_{M,L}] =b_{\infty,L}$ and $A_{M,L}(k,k') =  \eta_{M,L}(f_{\psi_p}, f_{\psi_{k'} })$ with $ \E[A_{M,L}]=A_{\infty,L} $.  
 Lemma \ref{lemma:CI} implies that that each of these entries concentrates around there mean:
 \begin{align*}
\Prob{|\xi_{M,L}(f_{\psi_p}) - b_{\infty,L}(k)|> \frac{\epsilon}{\sqrt{n}}} &\leq 4  e^{-\frac{M\epsilon^2 }{8nC^2}} , \\
 \Prob{|\eta_{M,L}(f_{\psi_p},f_{\psi_p'}) - A_{M,L}(k,k')|>\frac{\epsilon}{\sqrt{n}}} &\leq 2  e^{-\frac{M\epsilon^2 }{2nC^2}}. 
\end{align*}
where the constant $C$ is obtained from \eqref {eq:f_phiBD}. In combination of the basic fact that $\|b\| \leq \sqrt{n} \max_{k=1,\dots,n} |b(k)|$ for any $b\in \R^n$, and  $\|A\| \leq \sqrt{n} \max_{k,k'=1,\dots,n} |A(k,k')|$ for any $A\in \R^{n\times n}$, we obtain
\begin{align*}
& \Prob{\norm{ b_{M,L} - b_{\infty,L} }> \epsilon} \leq \sum_k \Prob(|\xi_{M,L}(f_{\psi_p}) - b_{\infty,L}(k)|> \frac{\epsilon}{\sqrt{n}}) \leq 4 n e^{-\frac{M\epsilon^2 }{8nC^2}} , \\
& \Prob{\norm{ A_{M,L} - A_{\infty,L} }> \epsilon}
\leq  \sum_{k,k'} \Prob{|\eta_{M,L}(f_{\psi_p},f_{\psi_p'}) - A_{\infty,L}(k,k')|>\frac{\epsilon}{\sqrt{n}}} \leq 2 n^2 e^{-\frac{M\epsilon^2 }{2nC^2}}. 
\end{align*}
The third exponential inequality follows directly by combining the first two.  
\end{proof}

\begin{corollary}\label{corollary:eignvalue} 
 
Denote $\lambda_{\mathrm{min}}(A_{ML})$ the smallest eigenvalue of the normal matrix $A_{ML}$ defined in \eqref{eq:Ab_ML}. We have 
\[
\Prob{\lambda_{\mathrm{min}}(A_{ML}) > c_{\hypspace} -\epsilon} >1-\delta
\]
 with $\delta =2n^2 \exp{\left(-\frac{M\epsilon_1^2 }{2nc_1^2} \right)} $, for any $\epsilon_1>0$ and any $\Delta t = T/L$ such  that $ \epsilon_1 +  c_3 \Delta t ^{\frac{1}{2}}= \epsilon < c_{\hypspace} $, where $c_1 $ and $c_3$ are defined in \eqref{eq:constants_c123}. 
\end{corollary} 
\begin{proof} Note that for any $a\in \R^n$ such that $\|a\|=1$, we have,  by Corollary \ref{cor_AmatL}, $a^T A_{\infty, L}a \geq  c_{\hypspace} - c_3  \Delta t ^{\frac{1}{2}} $.  

Meanwhile, Proposition \ref{prop:CI} implies that
\begin{align*} 
\|a^TA_{M,L} a - a^T A_{\infty, L}a  \| &\leq \|A_{M,L} - A_{\infty, L}\|  \leq \epsilon_1
\end{align*}
with probability at least $1-\delta$. Thus, 
\[  \|a^TA_{M,L} a \| \geq \|a^T A_{\infty, L}a  \| -\epsilon_1 \ge c_{ \hypspace}- c_2 \Delta t ^{\frac{1}{2}} -\epsilon_1, \]
and  the corollary follows.  
\end{proof}
\begin{remark} The above corollary requires $\epsilon >   c_3 \Delta t ^{\frac{1}{2}}$. This condition can be removed if the coercivity holds for the discrete-time observations on $\hypspace$ with a constant $c_{\hypspace,T,L}$, which can be tested numerically from a date set with a large $M$. In fact, we obtain directly from the above proof that  $\Prob{\lambda_{\mathrm{min}}(A_{ML}) > c_{\hypspace,T,L} -\epsilon} >1-\delta$ with $\delta =2n^2 \exp{\left(-\frac{M\epsilon_1^2 }{2nc_1^2} \right)} $, for any $\epsilon>0$. 
\end{remark}
\begin{remark} In practice, the minimum eigenvalue of $A_\infty$ may be small due to the redundancy of the local basis functions or due to the coercivity constant on $\hypspace$ being small. Thus, the smallest eigenvalue of $A_{M,L}$ may be zero.  On the other hand, these matrices are always symmetric and nonnegative, so it is advisable to regularize the matrix by pseudo-inverse. 
\end{remark}

\section{Examples and numerical simulation results}
\label{main:numericalexamples}
In this section, we performed numerical experiment to validate that our estimator defined in \eqref{MLE}, and implemented by Algorithm \ref{alg:main}, behaves in practice as predicted by the theory. We consider two examples: a stochastic opinion dynamical system and a stochastic Lennard-Jones system, using observations from simulated data. 

The setup for the numerical simulations is as follows. We simulate sample paths on the time interval $[0, T]$ with the standard Euler-Maruyama scheme (see \eqref{EM}), with a sufficiently small time step length $dt$. When observations are made at every time-step, i.e. $\Delta t = t_{l+1}-t_l = dt$ for each $l$,  we view  $\bX_{\mathrm{train},M}:=\{\bX^{(m)}_{t_0:t_L}\}_{m=1}^M$ as continuous-time trajectories. When observations occur spaced in time with observation gap $\Delta t$ equal to an integer multiple of $dt$, we refer to them here as discrete-time observations. 



From the observations we construct the empirical probability measure $\rho^{L,M}_T$ (defined in \eqref {e:rhoLM}), and let $[R_{\min},R_{\max}]$ be its support.
We choose the hypothesis spaces $\hypspace$ consisting of piecewise constant or piecewise linear polynomials on interval-based partitions of $[R_{\min},R_{\max}]$. This choice is dictated by the ease of obtaining an orthonormal basis for $\hypspace$, ease and efficiency of computation, and ability to capture local features of the interaction kernel. To avoid discontinuities at the extremes of the intervals in the partition, and to reduce stiffness of the equations of the system with the estimated interaction kernels, we interpolate the estimator linearly on a fine grid and extrapolate it with a constant to the left of $R_{\min}$ and the right of $R_{\max}$. This post-processing procedure ensures  the Lipschitz continuity of the estimators. We use the post-processed estimators to predict and generate the dynamics with the estimated interaction kernels.


We mainly focus on the case where $T$ is small and report on the results as follows: 
\begin{itemize}
\item {\bf Interaction kernel estimation.} We compare $\intkernel$ and $\widehat\intkernel_{T,M,\hypspace}$,  the true and estimated interaction kernels (after smoothing), by plotting them side-by-side, superimposed with an estimate of $\rhoT$, obtained as in \eqref{e:rhoLM} by using $\smash{M_{\rhoT}}$ ($\smash{M_{\rhoT}}\gg M$) independent trajectories. The estimated kernel is plotted in terms of its mean and standard deviation, computed over $10$ independent learning trials. To demonstrate the dependence of the estimator on the sample size and the scale of the random noise, we report the above for different values of $M$ and $\sigma$.

\item {\bf Trajectory prediction.} In the spirit of Proposition \ref{Trajdiff}, we compare the discrepancy between the true trajectories (evolved using $\intkernel$) and predicted trajectories (evolved using $\widehat\intkernel_{T,M,\hypspace}$)  on both the training time interval $[0,T]$ and on a future time interval $[T, T_{f}]$,  over two different sets of initial conditions -- one taken from the training data, and one consisting of new samples from $\mu_0$. When simulating the trajectories for the systems driven by $\widehat\intkernel_{T,M,\hypspace}$ using the EM scheme, we use the same initial conditions and the same realization of the random noise as in the trajectory of the system driven by $\intkernel$. The mean trajectory error is estimated using $M$ test trajectories (the same number as in the training data).  

\item {\bf Rate of convergence.} We report the convergence rate of $\widehat\intkernel_{T,M,\hypspace}$ to $\intkernel$ in the $\vertiii{\cdot}$ norm on $\smash{L^2(\rhoT)}$ as the sample size $M$ increases, with the dimension of $\hypspace$ growing with $M$ according to Theorem \ref{maintheorem}, for  different scales $\sigma$ of the random noise. We also investigate numerically the convergence rate when both $T$ and $M$ increase, with the dimension of the hypothesis space $\hypspace$ set according to the effective sample size as discussed in Section \ref{sec:hypspace}. 

\item {\bf Discretization errors from discrete-time observations.} To study the discretization error due to discrete-time observations, we report the convergence rate (in $M$) of estimators $\widehat\intkernel_{L,T,M,\hypspace}$ obtained from data with different observation gaps $\Delta t=T/L$. We also verify numerically that the $\vertiii{\cdot}$ error of the estimators increases with $\Delta t$ as predicted by Theorem \ref{thm:error_discreteTime}. These experiments are carried out for different values of the square root of the diffusion constant  $\sigma$. 
\end{itemize}

We will report the conclusions of our experiments in Section \ref{s:expconclusions}

\subsection{Example 1: Stochastic opinion dynamics}
We first consider a 1D system of stochastic opinion dynamics with interaction kernel
\[
\intkernel(r) = \left\{
        \begin{array}{ll}
           0.4,    & \quad 0                          \le r < \frac{1}{\sqrt{2}}-0.05, \\
           -0.3\cos(10\pi(r-\frac{1}{\sqrt{2}}+0.05)) + 0.7, & \quad \frac{1}{\sqrt{2}}-0.05 \le r <  \frac{1}{\sqrt{2}}+0.05, \\
           1,    & \quad  \frac{1}{\sqrt{2}}+0.05    \le r <0.95,\\
             0.5  \cos(10\pi (r -0.95)) + 0.5, & \quad 0.95 \le r <  1.05\\
             0,&\quad 1.05 \le r
  \end{array}
    \right.
\] 
It is straightforward to see that $\intkernel$ is in $C_c^{1,1}([0,2])$ and non-negative. Systems of this form are motivated in various applications, from Biology to in social science, where $\intkernel$ models how the opinions of people influence each other (see \cite{Krause2000,BHT2009,MT2014,BT2015,CKFL2005SI} and references therein), with one or a multiplicity of consensuses may be eventually reached. In the system we consider,  each agent tries to align its opinions more with its farther neighbours than with its closer neighbours: such interactions are called heterophilious.  For deterministic systems of this type, \cite{MT2014} shows that the opinions of agents merge into clusters, with the number of clusters significantly smaller than the number of agents. This is natural, as increased alignment with farther neighbors increases mixing and consensus. In our stochastic setting, the random noise prevents the opinions from converging to single opinions. Instead, soft clusters form at large time, that are metastable states for the dynamics, i.e. states where agents dwell for long times, rarely switching between them.

\begin{table}[H]
\centering
\begin{tabular}{| c | c | c | c | c | c | c |c|c|c|}
\hline 
 $d$ &$N$&$M_{\rhoT}$ &  $dt$   & $[0;T;T_f]$   & $\ProbIC$          & deg($\psi$) &$n$ \\ 
\hline 
 $1$ &10 &$5\cdot10^4$ &0.01 & $[0;5;50]$ & $\mathcal{U}([0, 8])^{\otimes N}$ & 0&$ 40(\frac{M}{\log M})^{\frac{1}{3}}$ \\
\hline
\end{tabular}
\caption{\textmd{\footnotesize{(OD) Parameters for the system. }}}
\label{t:OD_params}
\end{table}
We study the performance of our estimators of the interaction kernel, from trajectory data. Table \ref{t:OD_params} summarizes parameters of the setup. 
In this example, we  choose $\hypspace_{n_M}$ to be the function space consisting of piecewise constant functions on $n_M$ uniform partitions of the interval $[0,10]$.

\begin{figure}[tbp]
\centering     
\subfigure[$\sigma=0.1,M=32$]{\label{figOD:1}\includegraphics[width=0.48\textwidth]{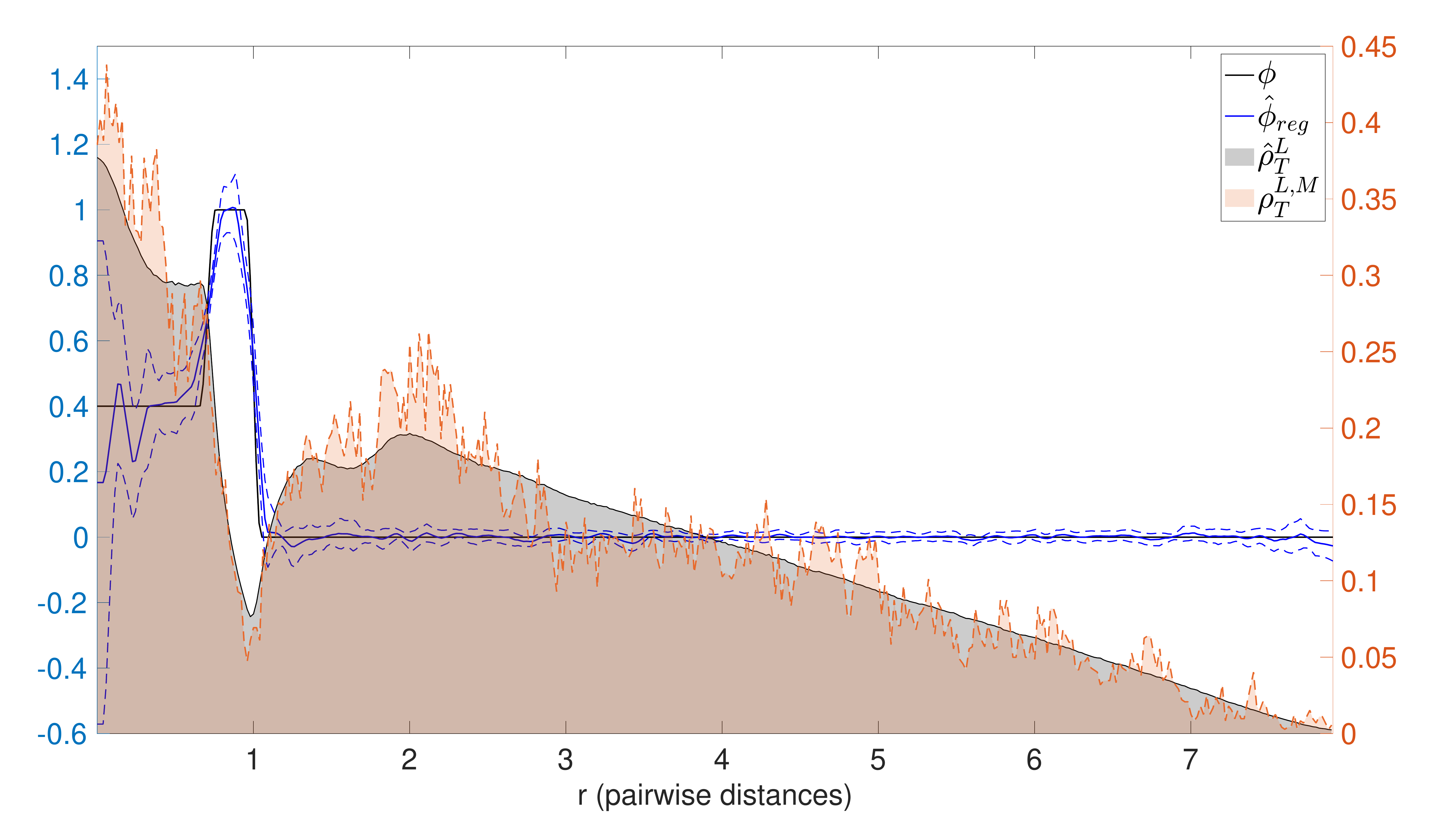}}
\subfigure[$\sigma=0.5,M=32$]{\label{figOD:3}\includegraphics[width=0.48\textwidth]{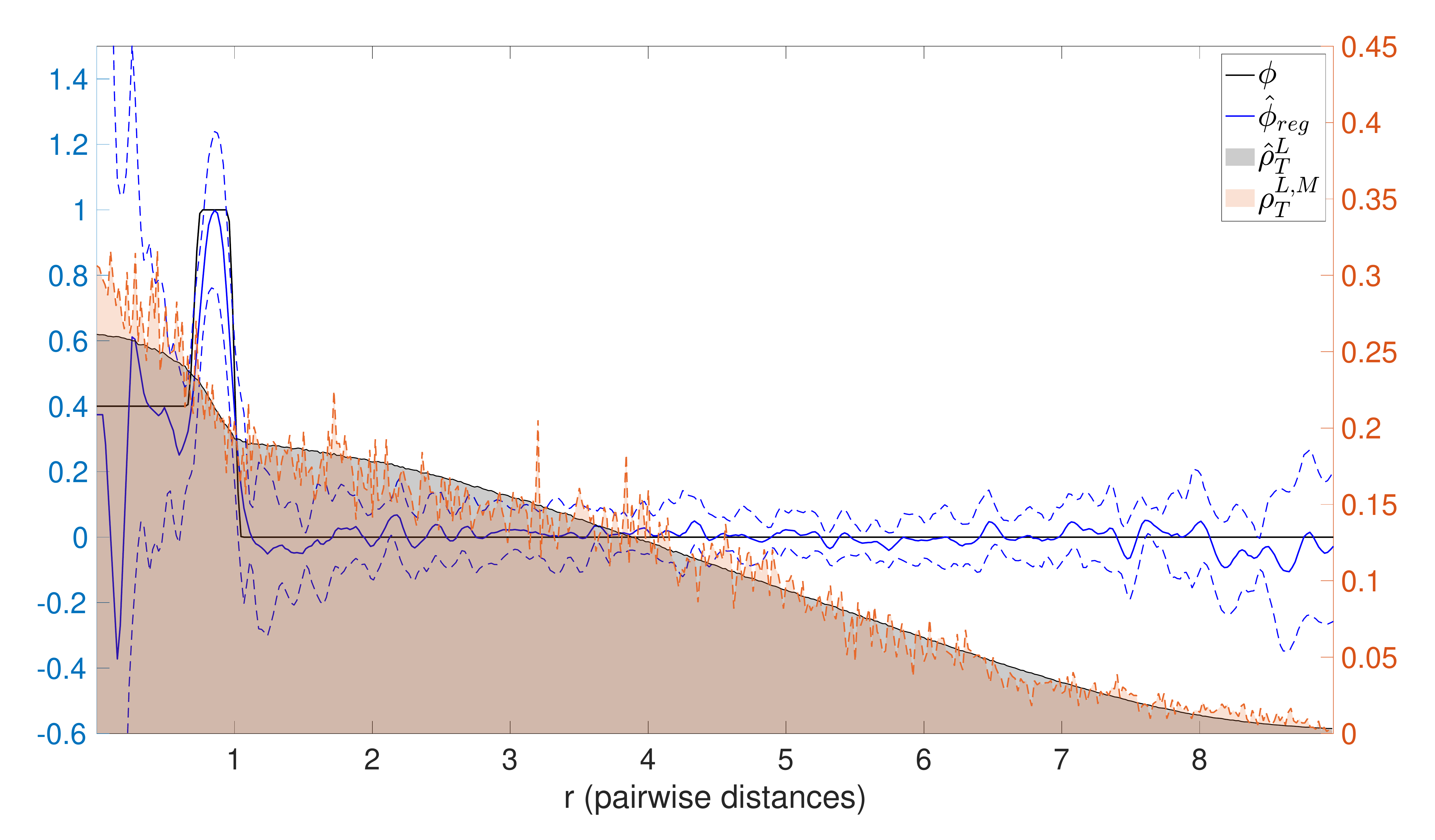}}
\subfigure[$\sigma=0.1,M=4096$]{\label{figOD:2}\includegraphics[width=0.48\textwidth]{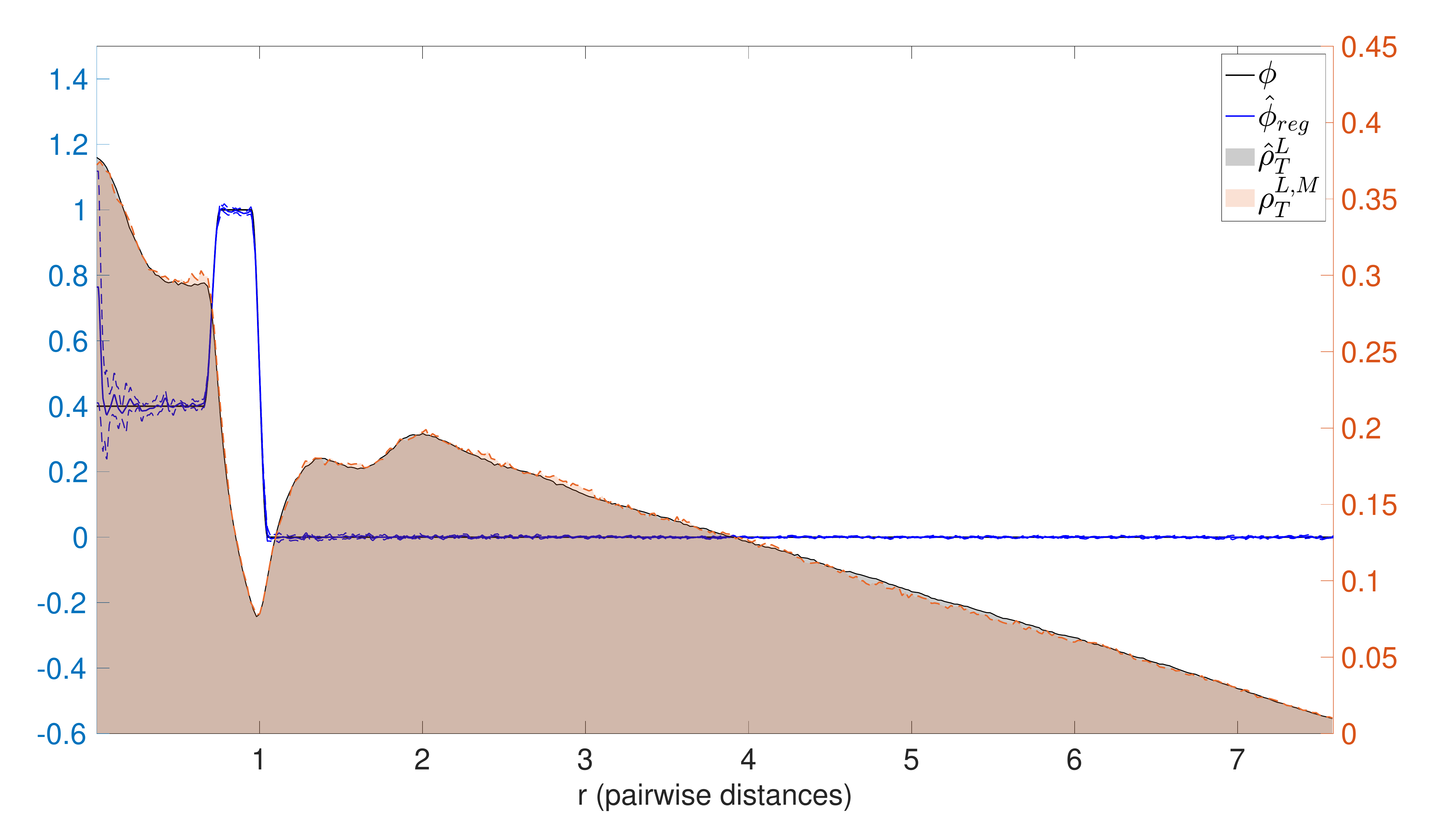}}
\subfigure[$\sigma=0.5,M=4096$]{\label{figOD:4}\includegraphics[width=0.48\textwidth]{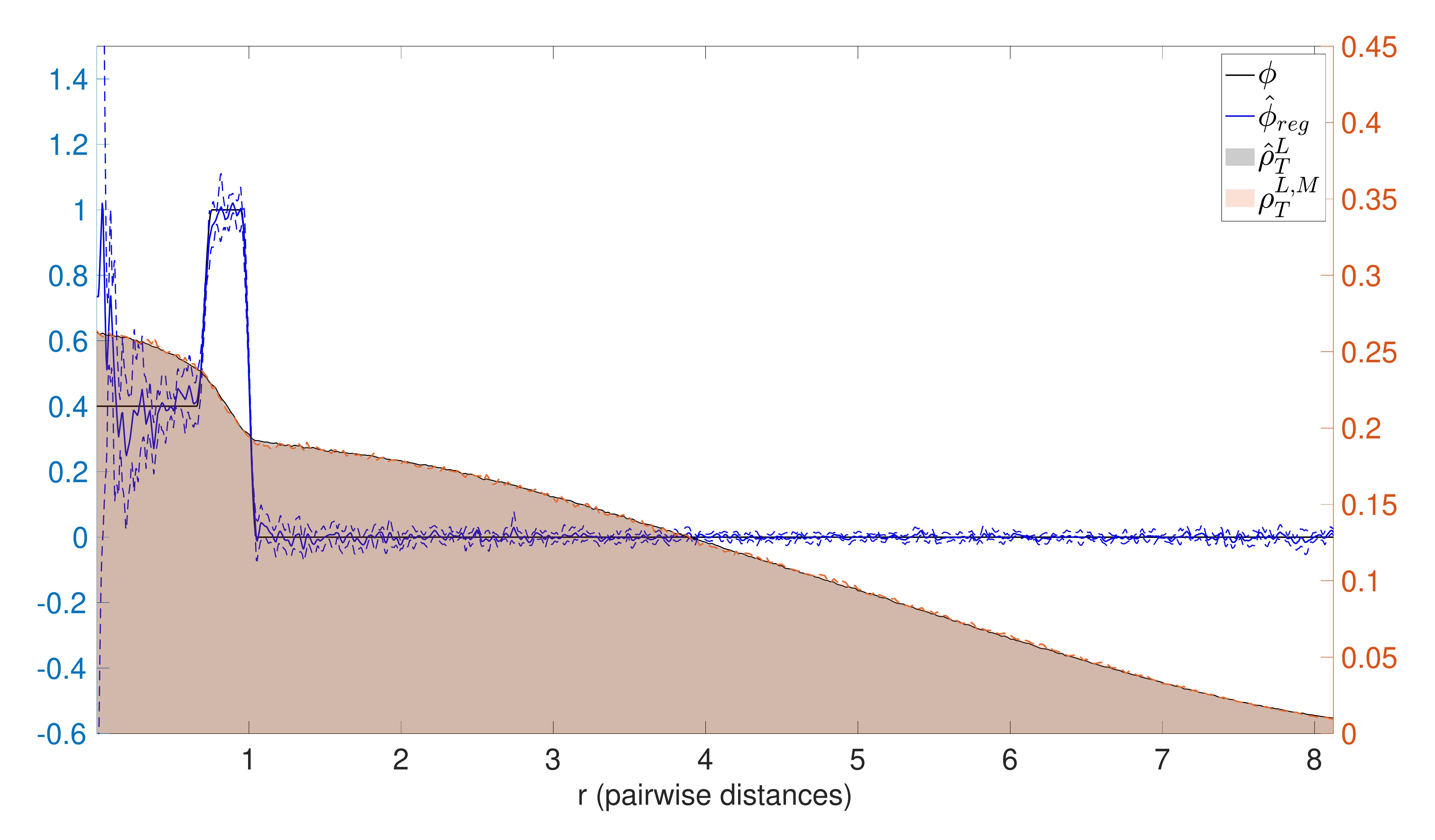}}
\caption{Stochastic opinion dynamics: comparison between true and estimated interaction kernels $\widehat\intkernel_{T,M,\hypspace}$ for different values of $M$ and $\sigma$, together with histograms (shaded regions) for $\rhoT$ and $\rhoT^M$. In black: the true interaction kernel. In blue: the mean of estimators in 10 independent trials, with dash-lines representing the standard deviation.   From top to bottom: learning from $M=2^5,2^{12}$ trajectories for kernels in systems with $\sigma=0.1$ (left) and $\sigma=0.5$ (right). The standard deviation bars on the estimated interaction kernels become smaller if $M$ increases and $\sigma$ decreases. The mean of the estimation error can be found in Figure \ref{fig:SODH1_Convergence}.}
\label{f:SODH1_kernel}
\end{figure}
Figure \ref{f:SODH1_kernel} shows that, as the number of trajectories increases, we obtain increasingly accurate approximations to the true interaction kernel, including at locations with sharp transitions of $\phi$. The lack of artifacts at these locations is an advantage provided by the use of local basis. The estimators oscillate near $0$, with amplitudes scaling with the level of noise. We believe that the reason for this phenomenon is that due to  the structure of the equations, we have terms of the form $\phi(0)\vec{0} = \vec{0}$ at, and near, $0$, with subsequent loss of information about the interaction kernel about $0$. 


We then use the learned interaction kernels $\widehat\phi$ in Figure \ref{f:SODH1_kernel}  to predict the dynamics, and summarize the results in Figure  \ref{t:SODH1_trajM32_err}  and Table \ref{t:SODH1_traj_err}.   Even with $M=32$, our estimator produces very accurate approximations of the true trajectories both in the training time interval $[0,5]$  and the future time interval $[5, 50]$, including number and location of clusters,  and the time of their formation.   As $M$ increases to 4096, we have more accurate predictions on the locations of clusters. We impute this improvement to the better reconstruction of estimators at locations near 0.

 Next we investigate the convergence rate of estimators.  It is well-known in  approximation theory (see  Theorem 6.1 in \cite{schumaker2007spline}) that 
$\inf_{\varphi\in \hypspace_n}\|\varphi-\intkernel\|_{\infty} \leq \text{Lip}[\intkernel]n^{-1}$. With the dimension $n$ being proportional to $\smash{(\frac{M}{\log M})^{\frac{1}{3}}}$, Figure \ref{f:ODH1_Convergence_Plot} shows that the learning rate in terms of $M$ is around $M^{-0.34}$, which matches the optimal min-max rate  $M^{-\frac{1}{3}}$ stated in  Theorem \ref{maintheorem}  with $s=1$. 

We also study the convergence of the estimator as the length of the trajectory $T$ increases, for the estimator $\widehat\intkernel_{T,M,\hypspace}$ from continuous-time trajectories (i.e. without gaps between observations). The auto-correlation time for this system is estimated to be about $\tau= 10$ time units. Therefore, we use relatively long trajectories up to $T=1500$ time units to test the convergence, contributing up to about 150 effective samples. We set the dimension of the hypothesis space to be $n=4(\frac{MT/dt}{\log(MT/dt)})^{\frac{1}{3}} $ for each pair $(M,T)$, where $dt$ is the time step size of the Euler-Maruyama scheme. The convergence rate of the estimators in terms of $MT$ is about $0.33$, showing the equivalence of learning from a single long trajectory with multiple short trajectories when the underlying process is ergodic.

\begin{figure}[tbp]
\centering     
\subfigure[$M=32,\sigma=0.1$]{\label{figOD:5}\includegraphics[width=0.48\textwidth]{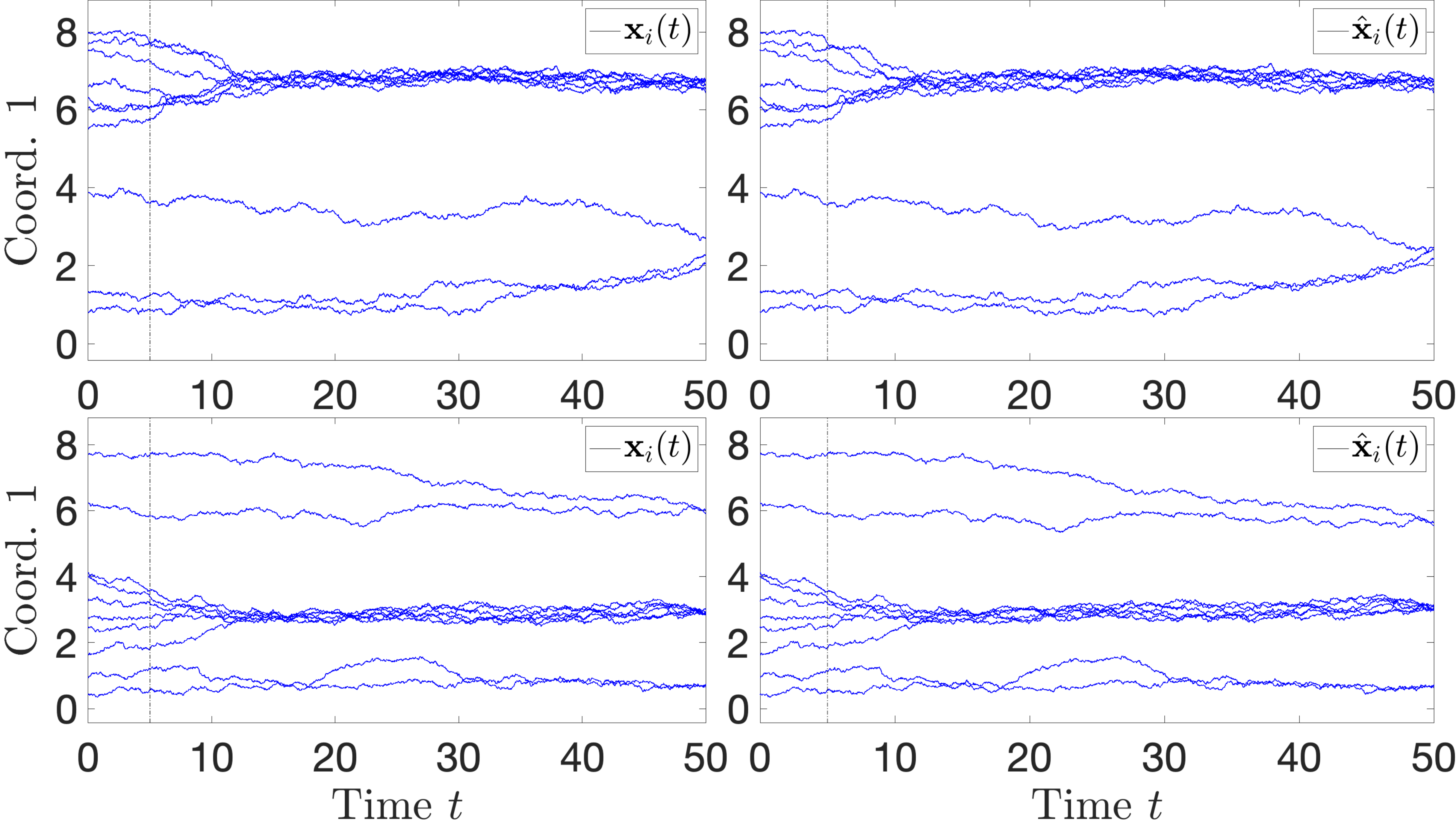}}
\subfigure[$M=32,\sigma=0.5$]{\label{figOD:6}\includegraphics[width=0.48\textwidth]{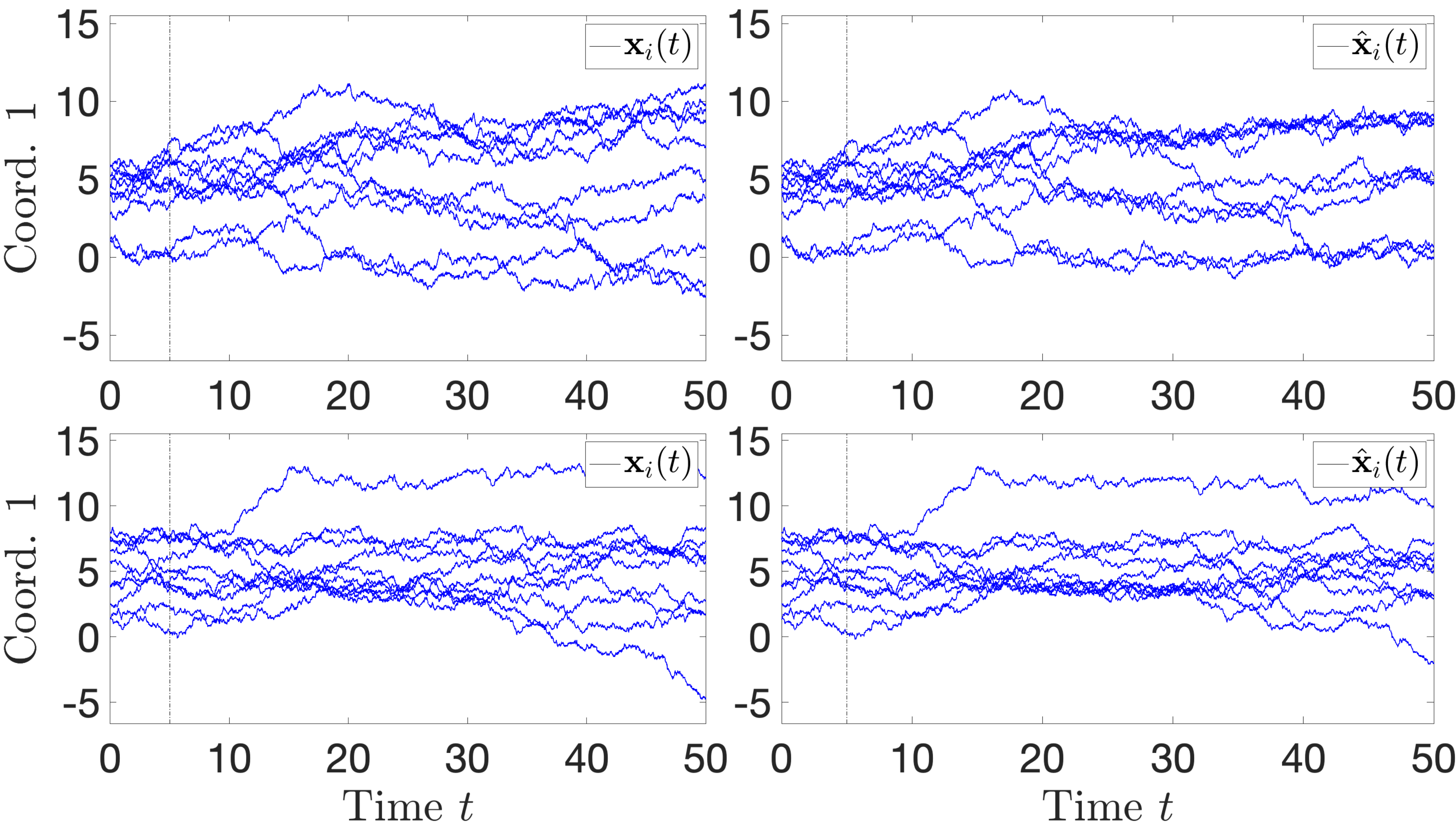}}

\subfigure[$M=4096,\sigma=0.1$]{\label{figOD:5}\includegraphics[width=0.48\textwidth]{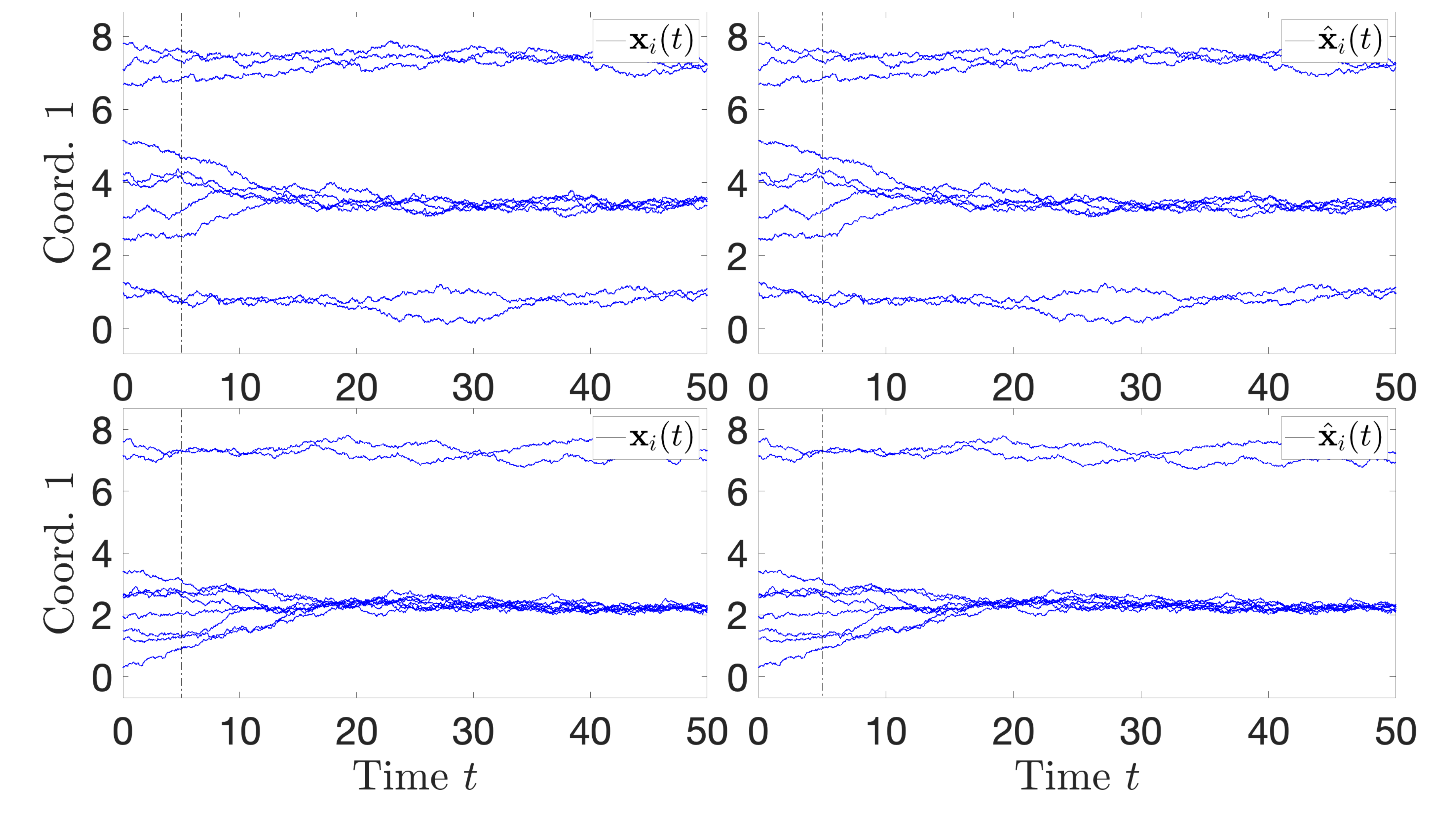}}
\subfigure[$M=4096,\sigma=0.5$]{\label{figOD:6}\includegraphics[width=0.48\textwidth]{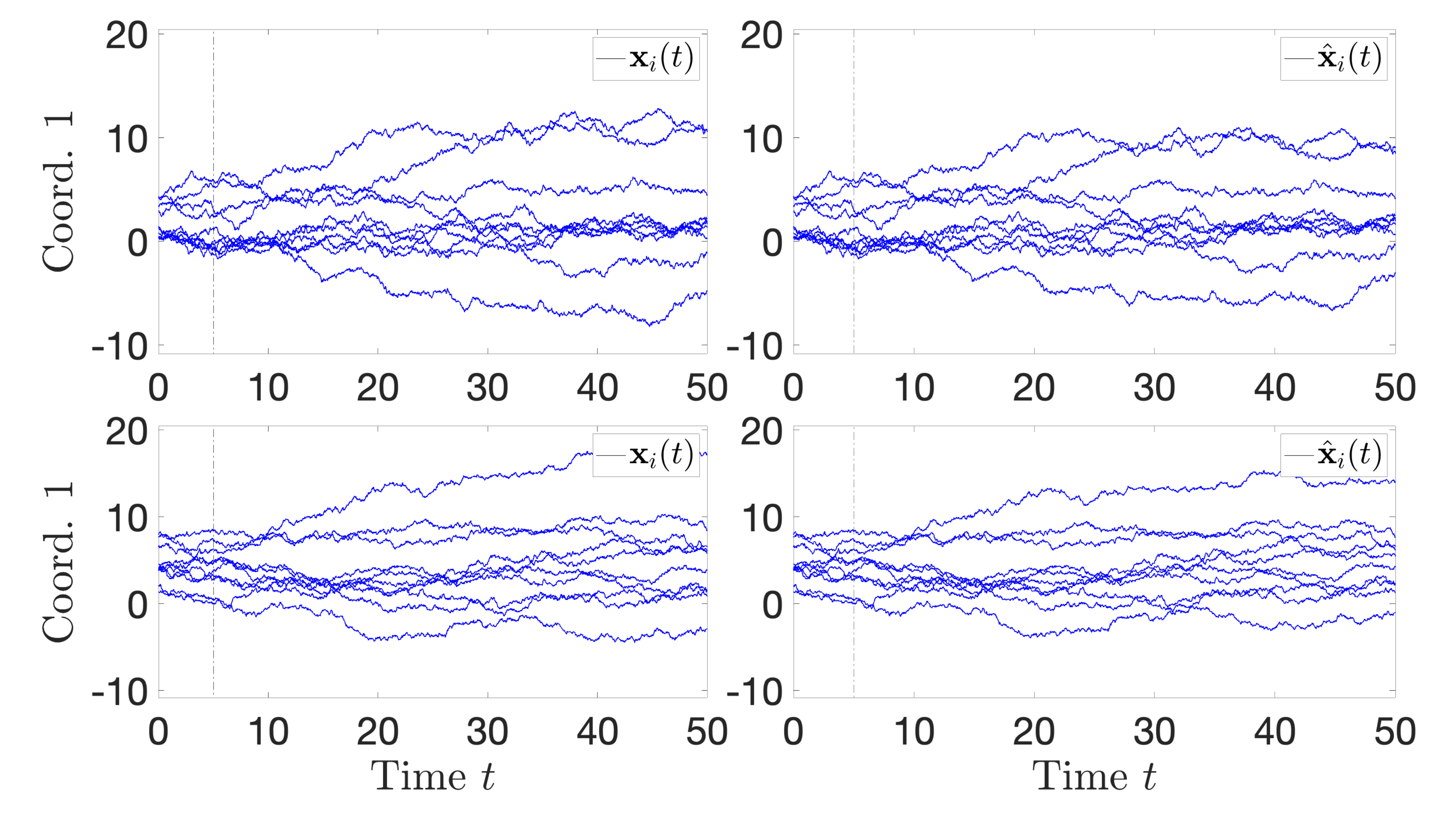}}

\caption{ \textmd{Stochastic opinion dynamics: trajectory prediction. In each panel:  $\bX_t$ (left column) and $\widehat \bX_t$ (right column) obtained with  the true kernel $\intkernel$ and the estimated interaction kernel $\widehat \intkernel_{T,M,\hypspace}$   from $M=32$ (top panel) and 4096 (bottom panel) trajectories, for an initial condition in the training data (top row in each panel) and a (new) initial condition randomly chosen from $\mu_0$ (bottom row in each panel). The black dashed vertical line at $t = T = 5$ divides the  ``training" interval $[0, T]$ from the ``prediction" interval [5,50].  As $M$ increases, our estimators achieve better approximation of the true kernel overall, and at regions near 0 (see Figure \ref{f:SODH1_kernel}). As a result, they produced more faithful prediction of the number and location of clusters for large time.  Statistics of trajectory prediction errors are reported in Table \ref{t:SODH1_traj_err}. }}\label{t:SODH1_trajM32_err}
\end{figure}

\begin{table}[tbp]
\centering
\begin{tabular}{| c || c | c |} 
\hline
                                                             & $[0, 5]$                                                    & $[5, 50]$\\
\hline
$M=32, \sigma=0.1,\text{mean}_{\text{traj}}$: Training ICs & $2.0 \cdot10^{-1} \pm 1.4 \cdot10^{-1}$  & $6.3 \cdot10^{-1} \pm 5.7 \cdot10^{-1}$\\
\hline            
$M=32,\sigma=0.1,\text{mean}_{\text{traj}}$: Random ICs & $1.7 \cdot10^{-1} \pm 1.2 \cdot10^{-1}$  & $5.7 \cdot10^{-1} \pm 3.9 \cdot10^{-1}$\\
\hline      
$M=32, \sigma=0.5,\text{mean}_{\text{traj}}$: Training ICs & $3.8 \cdot10^{-1} \pm 1.7 \cdot10^{-1}$  & $4.0 \cdot10^{0} \pm 2.3 \cdot10^{0}$\\
\hline            
$M=32, \sigma=0.5,\text{mean}_{\text{traj}}$: Random ICs & $3.6 \cdot10^{-1} \pm 1.1\cdot10^{-1}$  & $3.5 \cdot10^{0} \pm 1.4 \cdot10^{0}$\\
\hline     
\hline
$M=4096, \sigma=0.1,\text{mean}_{\text{traj}}$: Training ICs & $ 2.1\cdot10^{-2} \pm 2.0 \cdot10^{-2}$  & $ 9.3 \cdot10^{-2} \pm 1.6 \cdot10^{-1}$\\
\hline            
$M=4096,\sigma=0.1,\text{mean}_{\text{traj}}$: Random ICs & $2.1 \cdot10^{-2} \pm 2.3 \cdot10^{-2}$  &$ 9.8 \cdot10^{-2} \pm 1.7 \cdot10^{-1}$\\
\hline      
$M=4096, \sigma=0.5,\text{mean}_{\text{traj}}$: Training ICs & $ 5.2 \cdot10^{-2} \pm 3.5 \cdot10^{-2}$  & $ 3.8 \cdot10^{-1} \pm 3.0  \cdot10^{-1}$\\
\hline            
$M=4096, \sigma=0.5,\text{mean}_{\text{traj}}$: Random ICs & $ 5.2 \cdot10^{-2} \pm 3.5 \cdot10^{-2}$  & $ 3.8 \cdot10^{-1} \pm 3.0  \cdot10^{-1}$\\
\hline    
\end{tabular}
\caption{ \textmd{Stochastic opinion dynamics: means and standard deviations of trajectory prediction errors. The tests with ``Training ICs'' use initial conditions from the training data set.  The tests with  ``Random ICs use initial conditions that are randomly drawn from $\mu_0$. Means are taken over $M$ trajectories. There is little difference between errors on training and test ICs, indicating the prediction of trajectories generalizes perfectly to new ICs.}  }
\label{t:SODH1_traj_err}
\end{table}

We also investigate the effects of the scale of the random noise, which is represented by the standard deviation $\sigma$. Figure \ref{f:SODH1_kernel} shows that the estimators for the system with $\sigma=0.5$ have much large oscillations than those with $\sigma=0.1$. The left plot in Figure \ref{f:ODH1_Convergence_Plot}  shows that the scale of the random noise does not affect the learning rate, matching our theory.  We also see that the absolute $L^2(\rhoT)$  error of estimators increase as the system noise increases, this may indicate that the coercivity constant decreases as the level of noise in the system increases. The left  plot in Figure \ref{t:ODH1_Convergence_Plot} shows that the scale of the errors increase linearly in $\sigma$ (in particular, when the observation gap is 1).

\begin{figure}[tbp]
\centering
\subfigure{\label{fig:SODH1_Convergence}\includegraphics[width=0.49\textwidth]{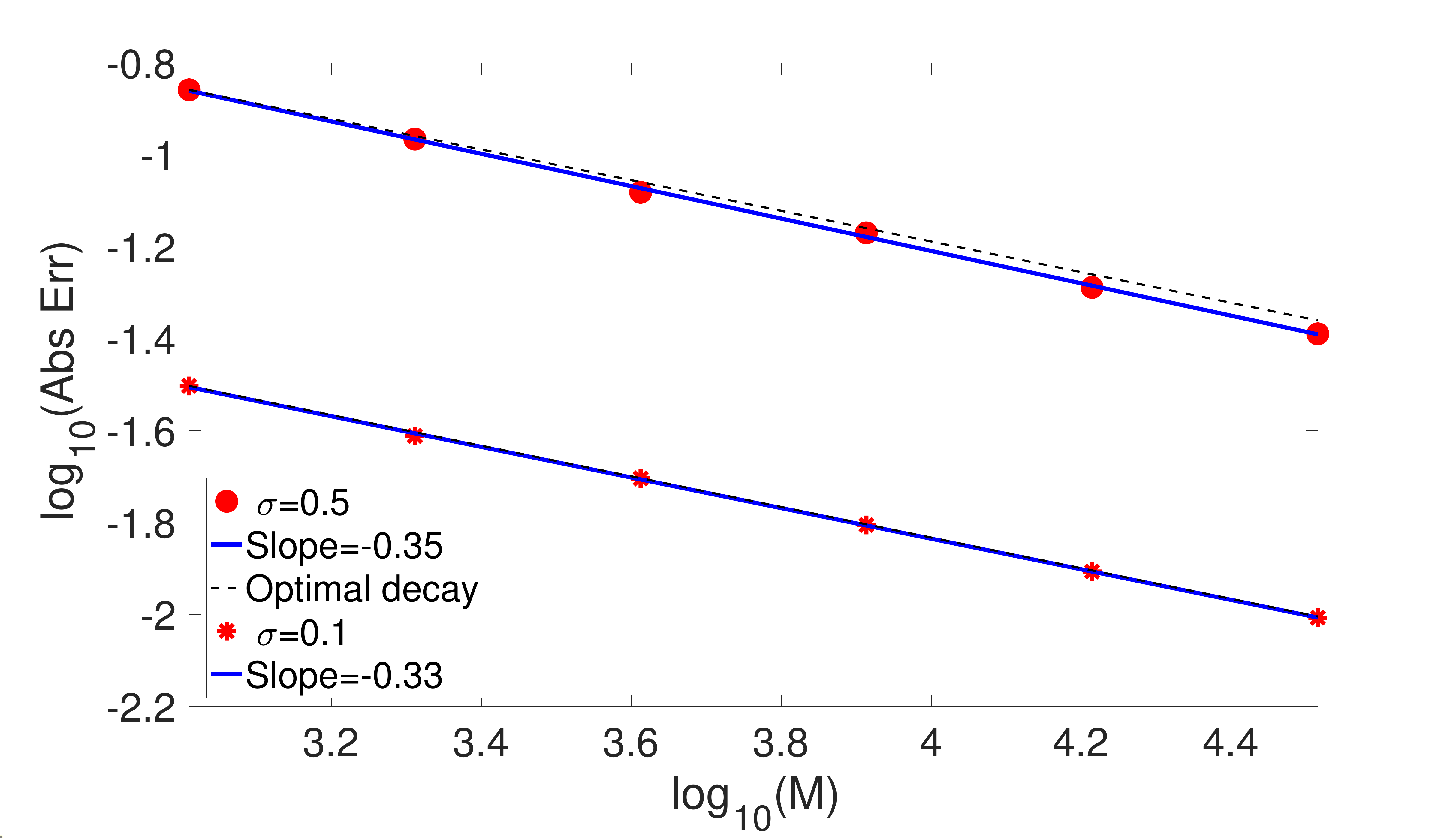}}
\subfigure{\label{t:ODH1_Convergence}\includegraphics[width=0.49\textwidth]{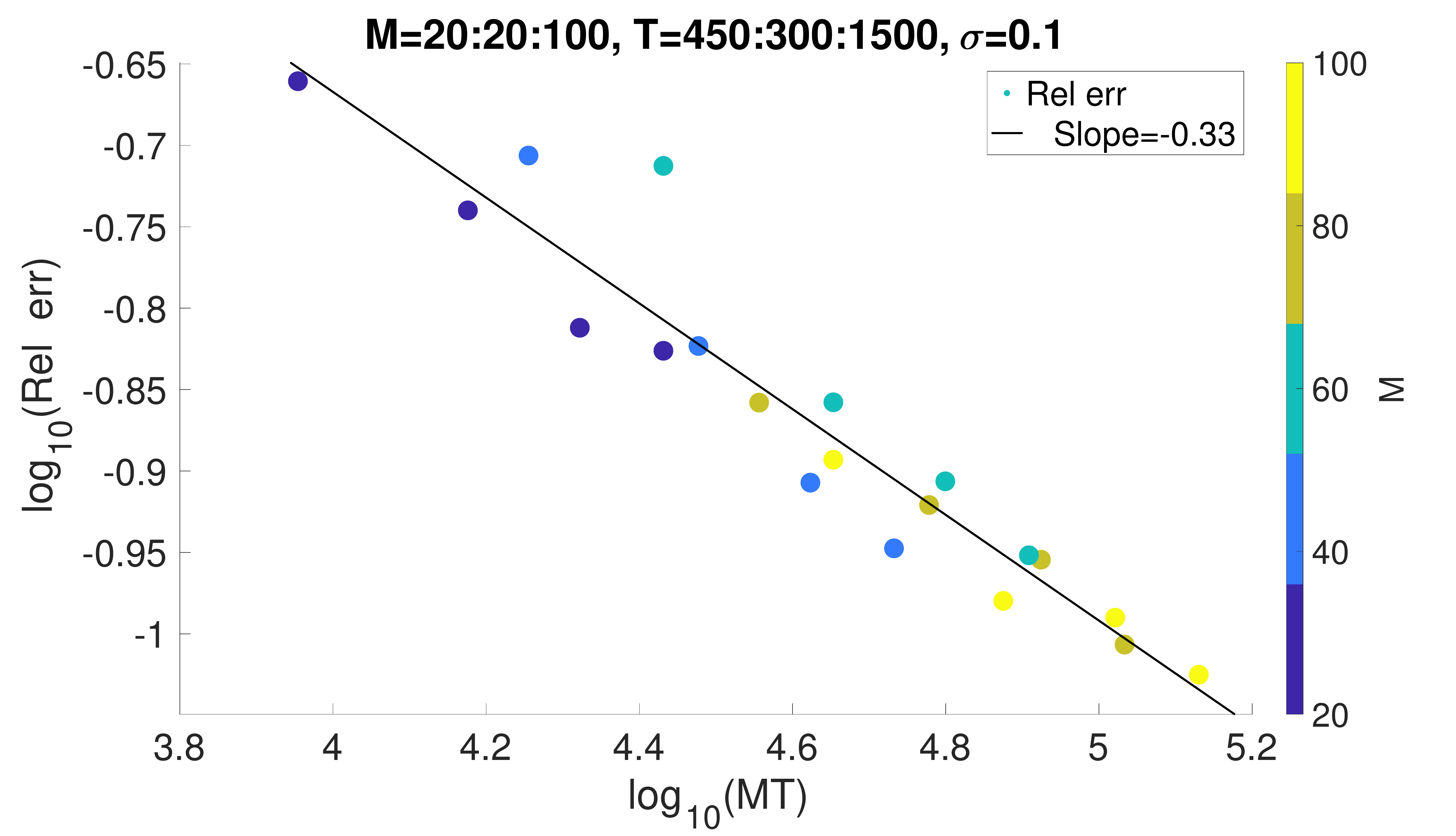}}
\caption{Stochastic opinion dynamics: learning rates for continuous-time observations. Left: the convergence rate of the estimators in terms of $M$ is $0.35$ for $\sigma = 0.5$ and is $0.33$ for $\sigma = 0.1$, close to the theoretical optimal min-max rate ${1}/{3}$ (shown in the black dot line).  Right: the convergence rate of the estimators in terms of $MT$,when both $M$ and $T$ increases,  is about $0.33$.  The colors of points are assigned  according to $M$. The learning rate is still close to the theoretical optimal min-max rate ${1}/{3}$, showing the equivalence of learning from a single long trajectory with multiple short trajectories when the underlying process is ergodic. 
}
\label{f:ODH1_Convergence_Plot} 
\end{figure}

\begin{figure}[tbp]
\centering
\subfigure{\label{t:SODH1_Obsgap}\includegraphics[width=0.49\textwidth]{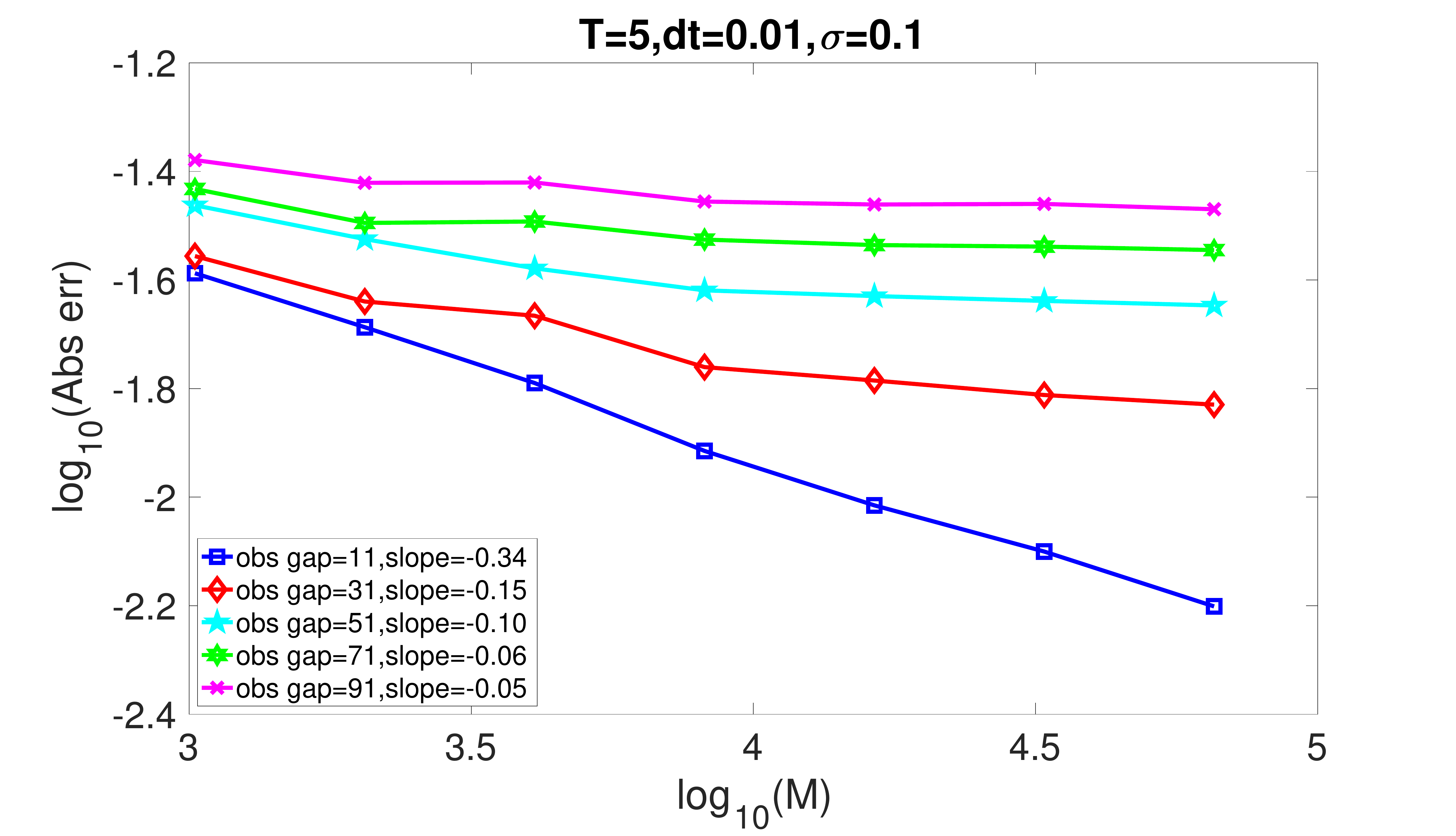}}
\subfigure{\label{fig:SODH1_NoisevsErr}\includegraphics[width=0.49\textwidth]{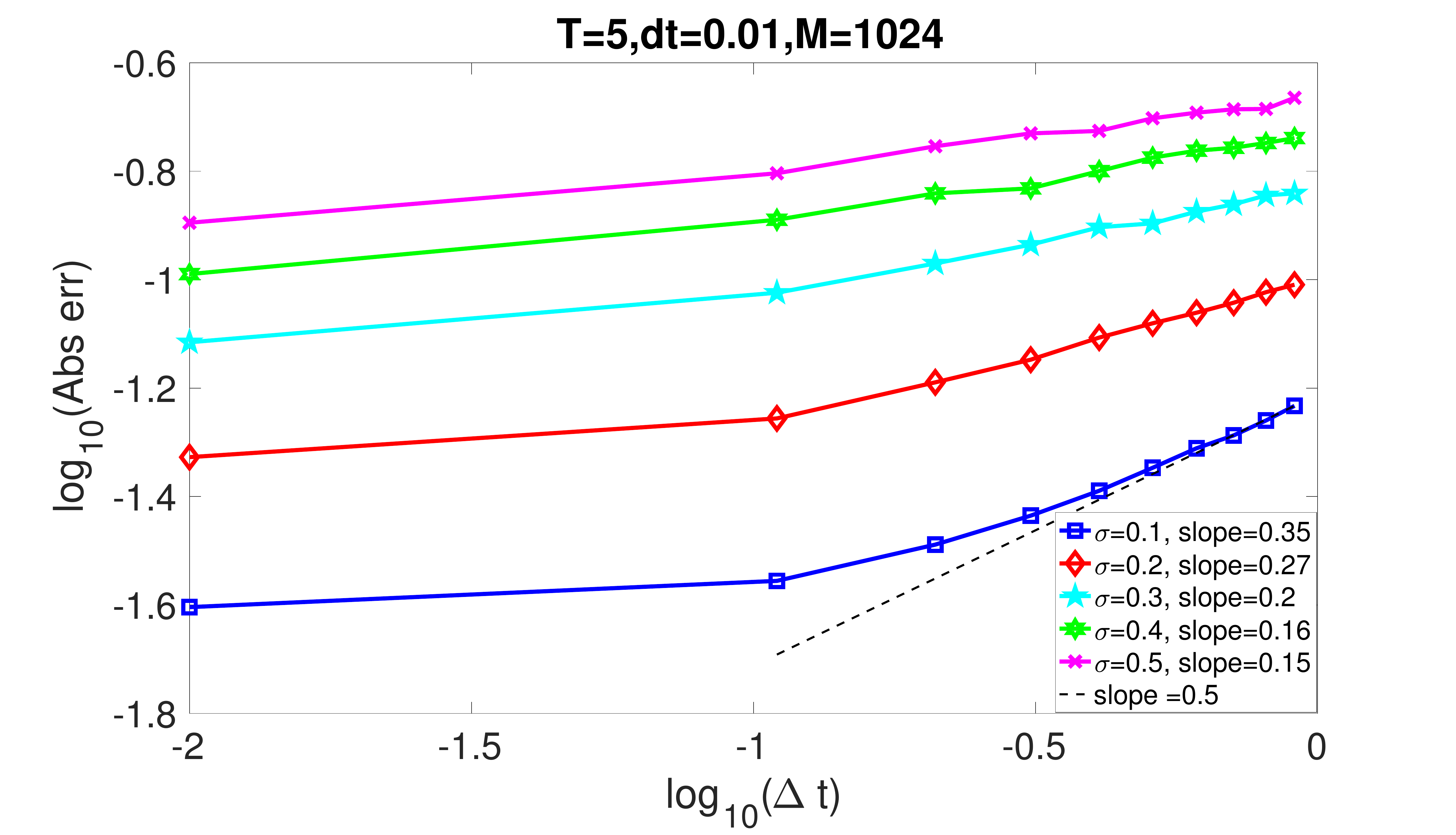}}
\caption{Stochastic opinion dynamics: discretization error due to discrete-time observation. Left: the learning rates of estimators $\widehat\intkernel_{L,T,M,\hypspace}$ obtained from data with different observation gaps $\Delta t= k dt$ for $k$ ranging from $11$ to $100$. Recall that $L=T/\Delta t$. As $k$ increases, the learning rate curves become flat, due to the bias induced by discretization of the likelihood function \eqref{lkhd_cts} on coarse time grids. Right: the log-log plot of the absolute error of the estimator in terms of observation gap $\Delta t= k dt$ for $k$ ranging from $1$ to $100$, for systems with different levels of random noise in terms of $\sigma$, computed with $M=1024$, $T=5$ and $dt = 0.01$ fixed. The orders of the absolute error in both $\sigma$ and $\Delta t$ are bounded by the theoretical order $\sigma O((\Delta t)^{1/2})$, dominating the statistical error due to sampling, finite-dimensional approximation, and noise. The slopes of the lines are calculated using points whose $x$ coordinate fall in the range $[-1,0]$.}  \label{t:ODH1_Convergence_Plot} 
\end{figure}

Finally, we study the discretization error due to approximation of the integral in the likelihood using discrete-time observations. In the left plot of Figure \ref{t:ODH1_Convergence_Plot}, as the observation gap $k$ increases, the learning rate curves become flat, due to the error induced by discretization of the likelihood function \eqref{lkhd_cts}. The right plot shows that the absolute error of the estimator is dominated by $\sigma O((\Delta t)^{1/2})$.

\subsection{Example 2: Stochastic Lennard Jones dynamics}


In this example, we consider the Lennard-Jones type kernel $\intkernel(r)=\frac{\Phi'(r)}{r}$,  with 
\[
\Phi(r) = \frac{p\epsilon}{(p-q)}\left[\frac{q}{p}\left(\frac{r_m}{r}\right)^{p}-\left(\frac{r_m}{r}\right)^{q} \right]
\] for some $p>q \in \mathbb{N}$. The  system of particles is assumed to be associated with a potential energy function only depending on the pairwise distance and $\Phi$, and the evolution is driven by minimization of the energy function.  In particular,  $\epsilon$ represents the depth of the potential well, $r$ is the distance between the particles, and $r_m$ is the distance at which the potential reaches its minimum. At $r_m$, the potential function has the value $-\epsilon$.  The  $r^{-p}$ term, which is the repulsive term, describes Pauli repulsion at short ranges due to overlapping electron orbitals, and the $r^{-q}$ term, which is the attractive long-range term.   The corresponding system  has wide applications in molecular dynamics and materials sciences  where  $\phi$ models atom-atom interactions.  Note that $\intkernel$ is singular at $r=0$: we truncate it at $r_{\text{trunc}}$ by connecting it with an exponential function of the form $a\exp(-br^{12})$ so that it has a continuous derivative on $\mathbb{R}^+$. 


In this system, the particle-particle interactions are all short-range repulsions and long-range attractions. The short-range repulsion force prevents the particles to collide and long-range attractions keep the particles in the flock. In the deterministic setting, 
the system evolves to equilibrium configurations very quickly, which are crystal-like structure, whose pairwise distance corresponds to the local minimizers of the associated energy function.  Table \ref{t:LJ_kernel_params} and \ref{t:LJ_system_params} summarize the system and learning parameters.

 Note that the true kernel $\phi$ is not compactly supported. But in our simulations, we observe the dynamics up to a time $T$ which is a fraction of the equilibrium time. Since the particles only explore a bounded region due to the large-range attraction,  $\rho_T$ is essentially compactly supported on a bounded region (see the histogram background of Figure \ref{t:LJH1_kernel}), on which $\phi$ is in our admission space.


We  use piecewise linear functions on $n$ uniform partitions of the learning interval to approximate the true kernel $\intkernel$.  With $M=32,$ Figure \ref{t:LJH1_kernel} shows that  we  have already obtained  faithful approximations to the true interaction kernel, except for on regions are close 0. Increasing number of observations improves the accuracy of estimators at locations near 0, which seems to be very helpful  for the system with larger noise level. 


In terms of the trajectory prediction, we use the learned interaction kernels $\widehat\phi$ in Figure \ref{f:SODH1_kernel}. We  summarize the results in Figure  \ref{t:SLJH1_trajM32_err}  and Table \ref{t:SLJH1_traj_err}.   In the experiments, we study two cases, one with small random noise 
 where the particles still form an  equilibrium configuration, and then this configuration have small fluctuation in the space; the other one with medium level of random noise,   where the random noise begins to break the formation of a fixed equilibrium configuration and we see the transition between different  configurations.  We see that in both cases, our estimators produce good prediction of the true dynamics in both training and future time interval. 

We plot the convergence rate of estimators in terms of $M$ in the right plot of Figure \ref{f:LJH1_Convergence_Plot}.  In this case, we have  
$\inf_{\varphi\in \hypspace_n}\|\varphi-\intkernel\|_{\infty} \leq \text{Lip}[\intkernel']n^{-2}$. We choose a choice of dimension $n$  proportional to $(\frac{M}{\log M})^{\frac{1}{5}}$, our numerical results show that the learning rate  is around $M^{-0.39}$, which matches the optimal min-max rate  $M^{-\frac{2}{5}}$ stated in  Theorem \ref{maintheorem}.  

We also study the  convergence of the estimators as the length of the trajectory $T$ increases. In this example with $\sigma= 0.35$,  the estimated auto-correlation time is about $\tau= 10$ time units. Therefore, we use relatively long trajectories up to $T=1200$ time units, contributing up to about 120 effective samples. We set the dimension of the hypothesis space to be $n=4(\frac{MT/dt}{\log(MT/dt)})^{\frac{1}{5}} $ for each pair $(M,T)$, where $dt$ is the time step size of the Euler-Maruyama scheme.
 The right plot of Figure \ref{f:LJH1_Convergence_Plot} shows that  the rate is 0.39, indicating the equivalence between a single long trajectory and multiple short trajectories for inference.

\begin{table}[H]
\centering
\begin{tabular}{ | c | c | c | c | c |}
\hline 
 $p$ &$q$ & $\epsilon$ & $r_m$ & $r_{\text{trunc}}$  \\ 
\hline 
$8$ &2& $1$ & $1$ &0.95\\
\hline
\end{tabular}
\caption{\textmd{\footnotesize{(Stochastic LJ) Parameters for the Lennard Jones kernel} }}
\label{t:LJ_kernel_params}
\end{table}

\begin{table}[H]
\centering
\begin{tabular}{ |c| c|c|c|c|c|c|c|c|}
\hline 
 $d$ &$N$&$M_{\rhoT}$&  $dt$  & $[0;T;T_{f}]$ & $\ProbIC$  & deg($\psi$) &$n$    \\ 
\hline 
2&10&$5\cdot 10^4$&0.001&[0;0.5;20]&$\mathcal{N}(0,I)$&1&$30(\frac{M}{\log M})^{\frac{1}{5}}$\\
\hline
\end{tabular}
\caption{\textmd{\footnotesize{(Stochastic LJ) Parameters for the system} }}
\label{t:LJ_system_params}
\end{table}

\begin{figure}[tbp]
\centering     
\subfigure[$\sigma=0.05,M=128$]{\label{figLJ:1}\includegraphics[width=0.48\textwidth]{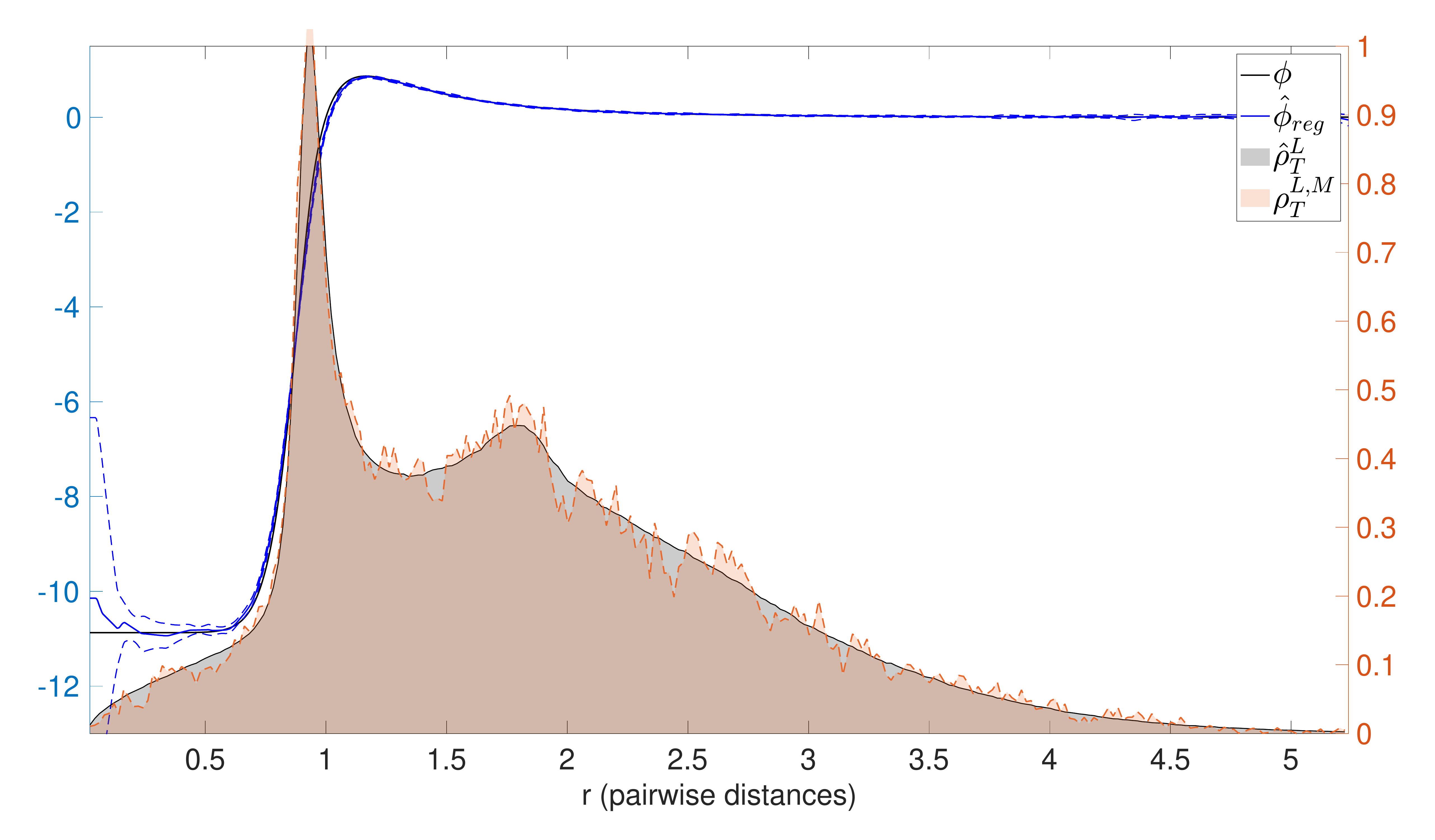}}
\subfigure[$\sigma=0.25,M=128$]{\label{figLJ:3}\includegraphics[width=0.48\textwidth]{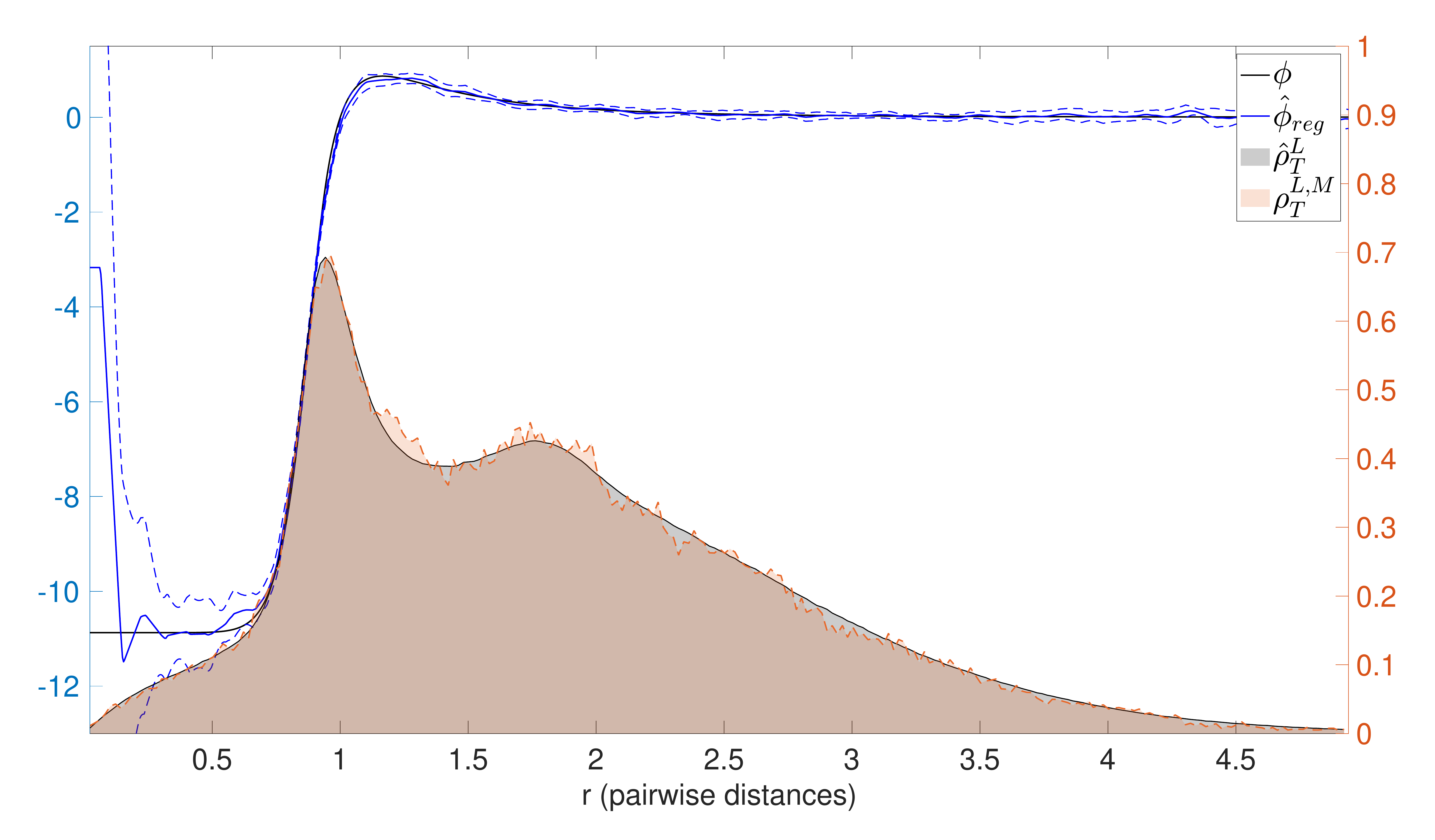}}
\subfigure[$\sigma=0.05,M=1024$]{\label{figLJ:2}\includegraphics[width=0.48\textwidth]{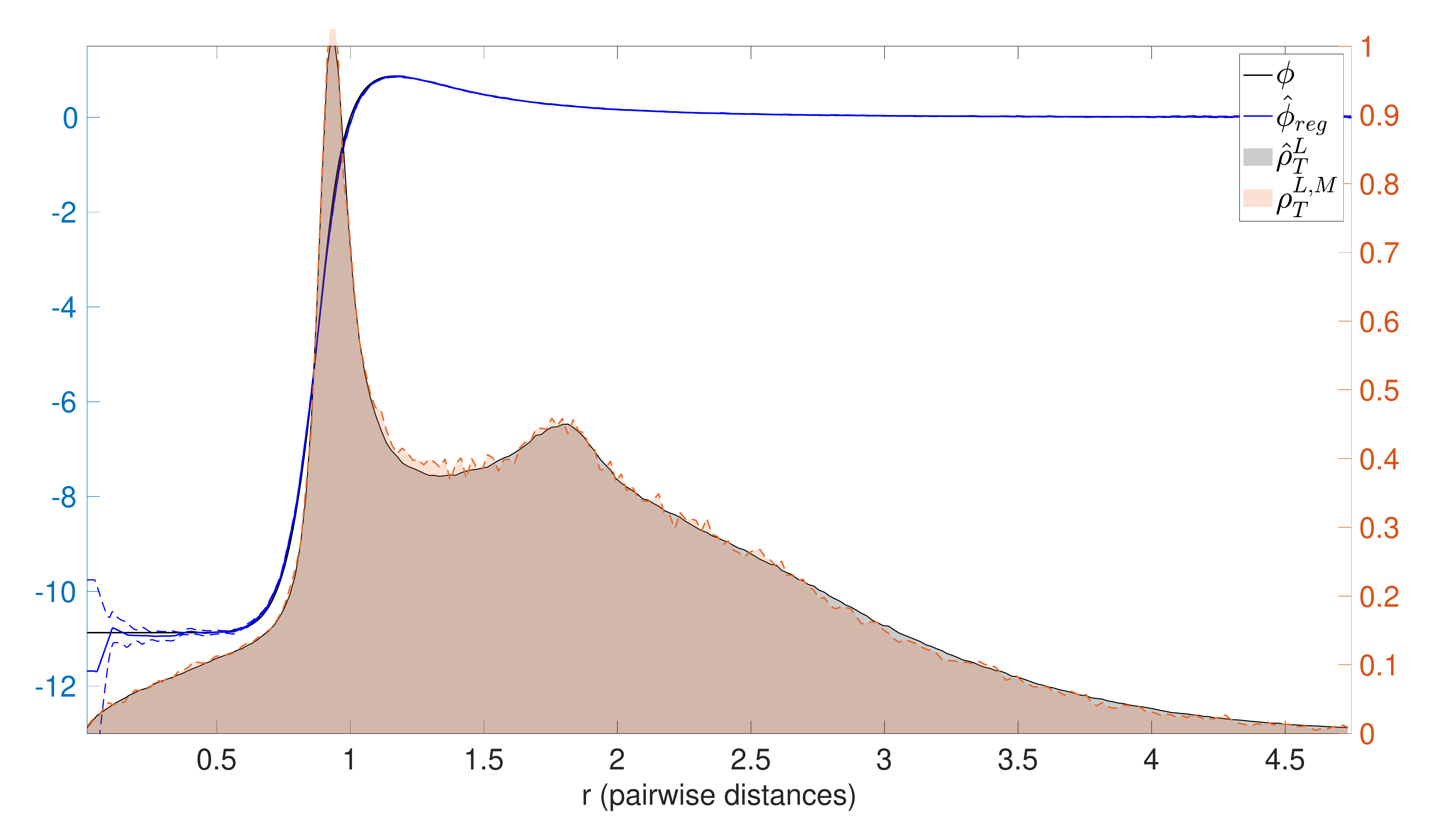}}
\subfigure[$\sigma=0.25,M=1024$]{\label{figLJ:4}\includegraphics[width=0.48\textwidth]{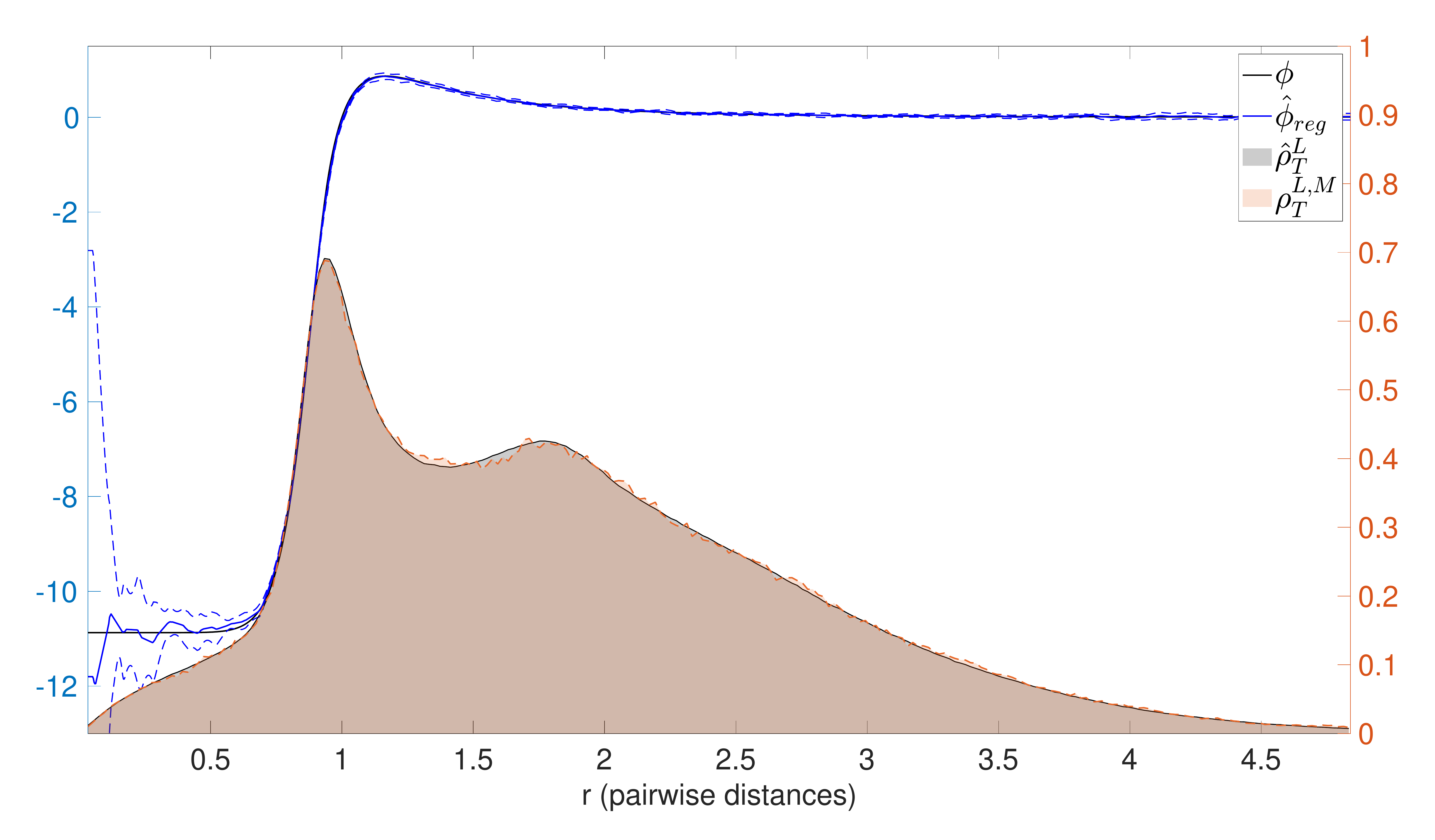}}
\caption{Stochastic Lennard-Jones dynamics: comparison between true and estimated interaction kernels with different values of $M$, together with histograms (shaded regions) for $\rhoT$ and $\rho_T^{M}$. In black: the true interaction kernel. In blue:  the mean of estimators in 10 independent trials, with dash-lines representing the standard deviation.  From top to bottom: learning from $M=2^7,2^{10}$ trajectories for kernels in systems with $\sigma=0.05$ (left) and $\sigma=0.25$ (right). The standard deviation bars on the estimated interaction kernels become smaller if $M$ increases and $\sigma$ decreases.  More details of the estimation errors can be found in Figure \ref{fig:SLJH1_Convergence}.}
\label{t:LJH1_kernel}
\end{figure}

\begin{figure}[tbp]
\centering     
\subfigure[$\sigma=0.05,M=128$]{\label{figLJ:5}\includegraphics[width=0.48\textwidth]{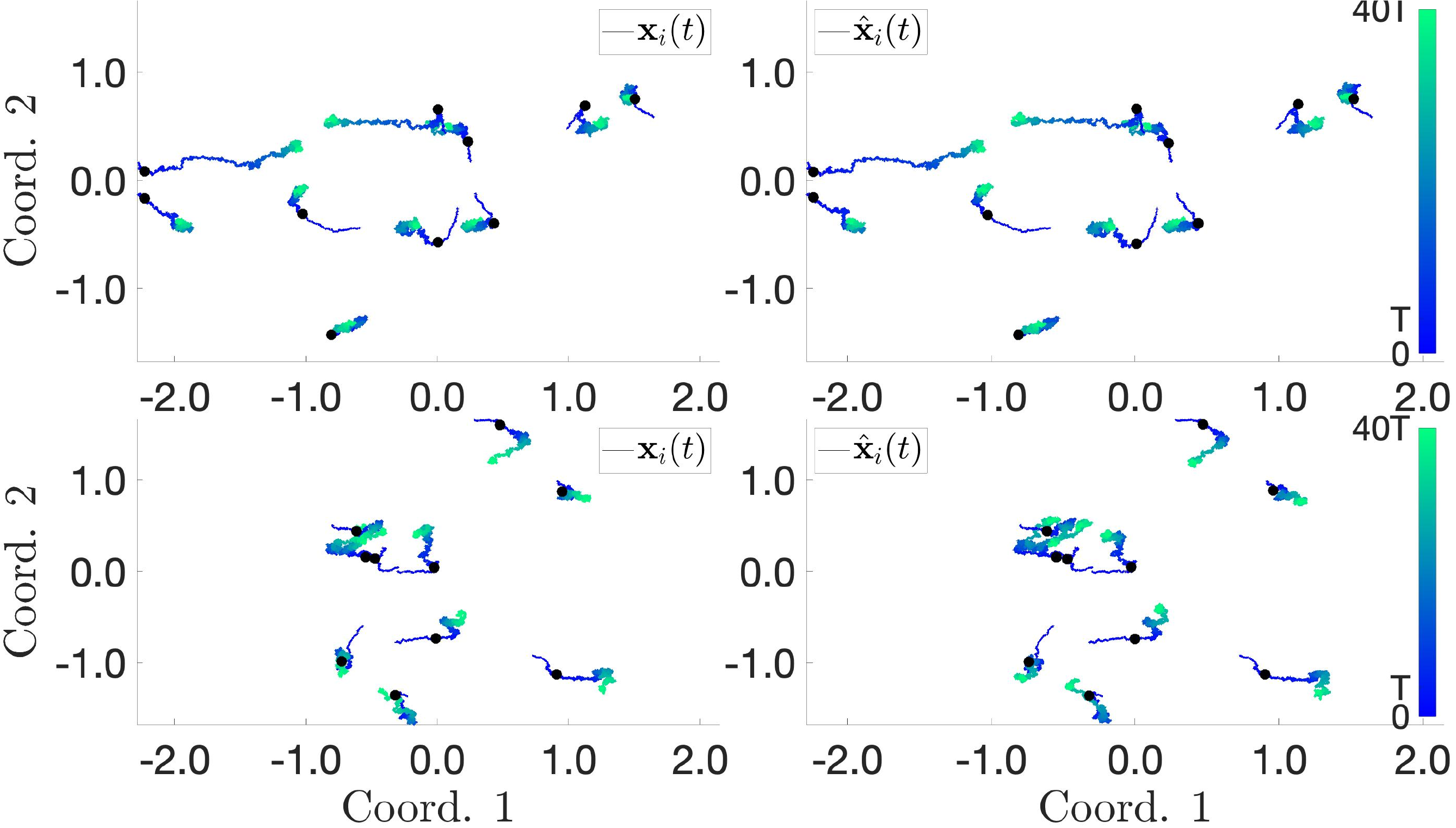}}
\subfigure[$\sigma=0.25, M=128$]{\label{figLJ:6}\includegraphics[width=0.48\textwidth]{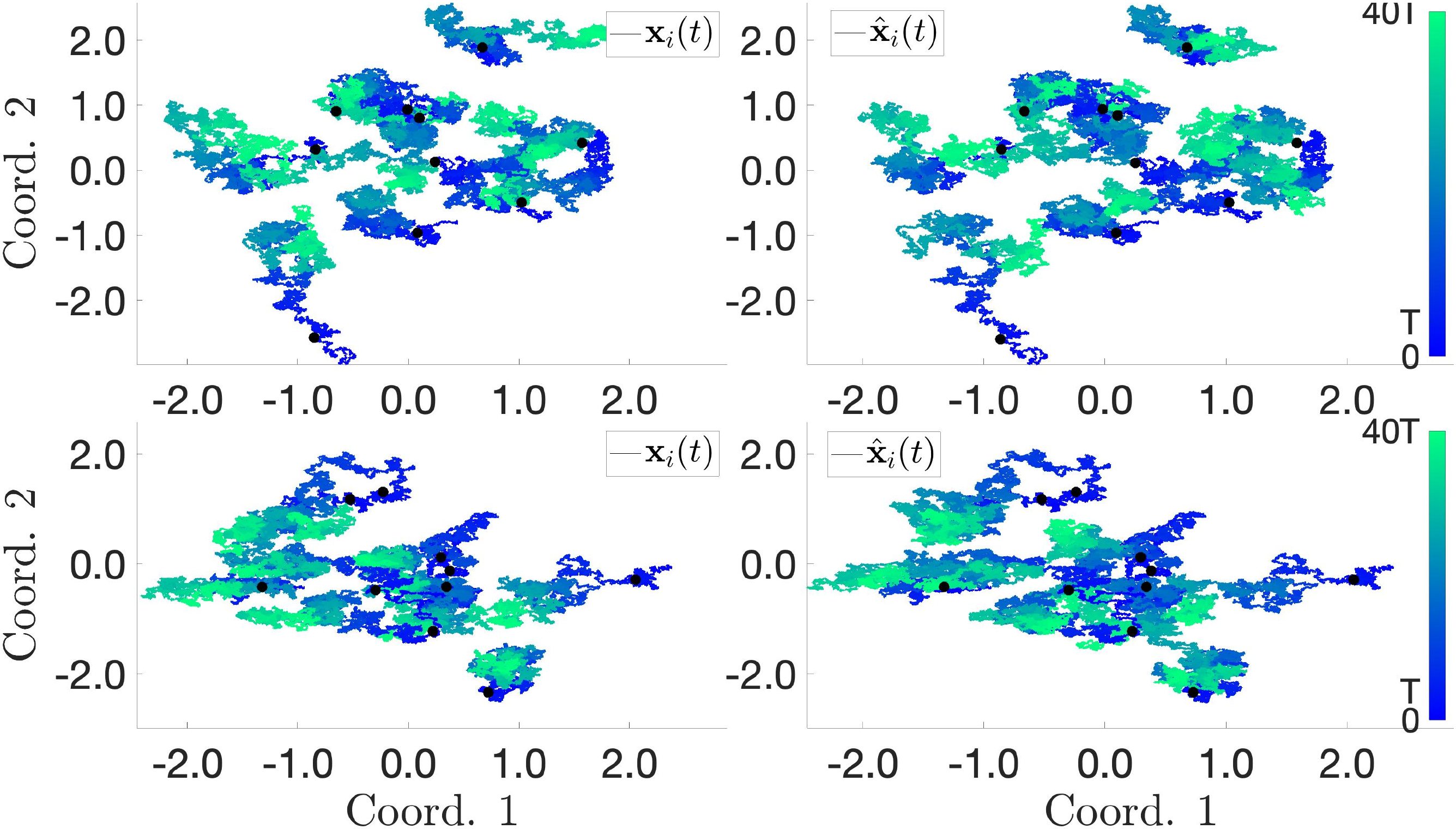}}
\subfigure[$\sigma=0.05,M=1024$]{\label{figLJ:5}\includegraphics[width=0.48\textwidth]{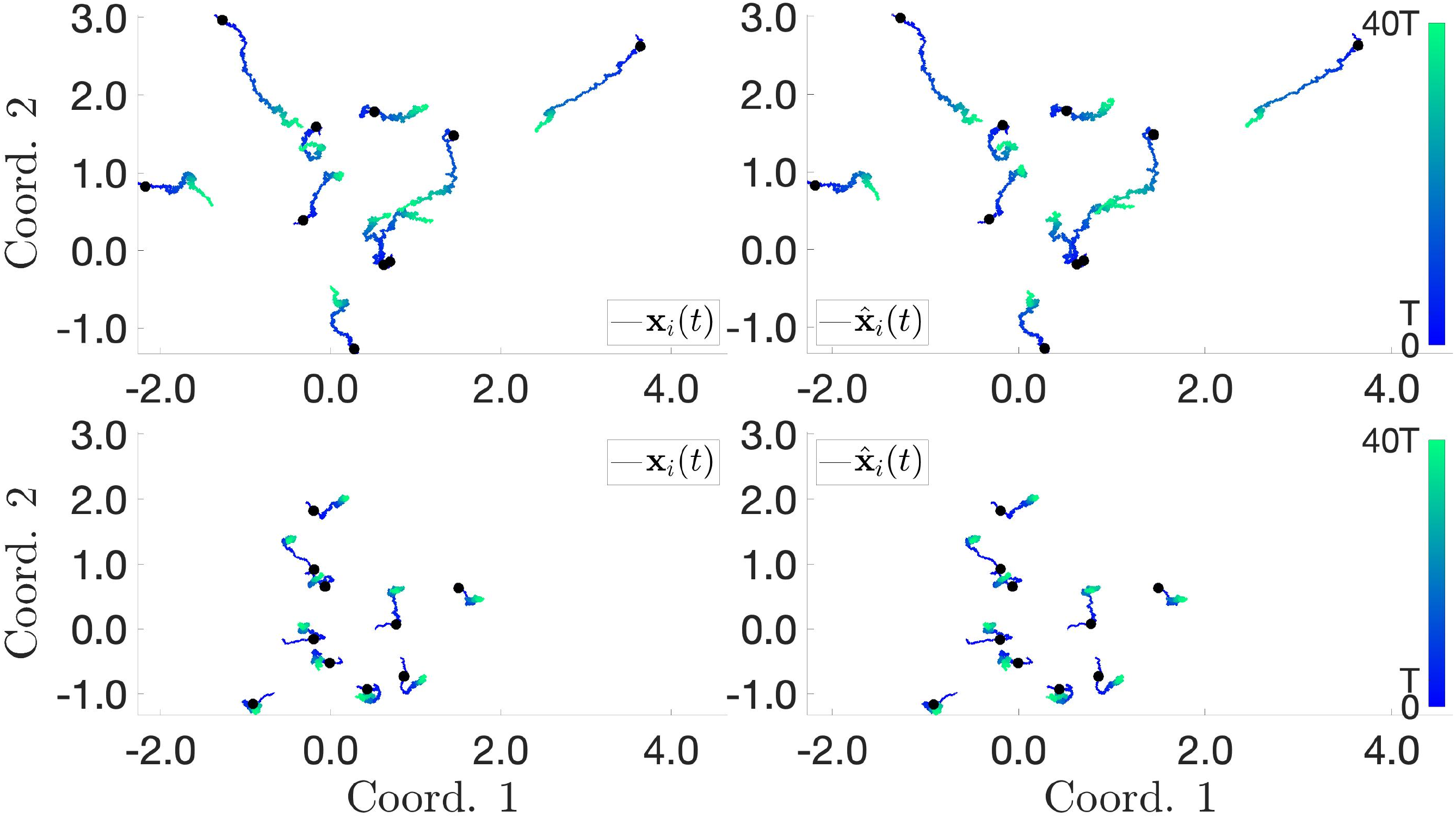}}
\subfigure[$\sigma=0.25, M=1024$]{\label{figLJ:6}\includegraphics[width=0.48\textwidth]{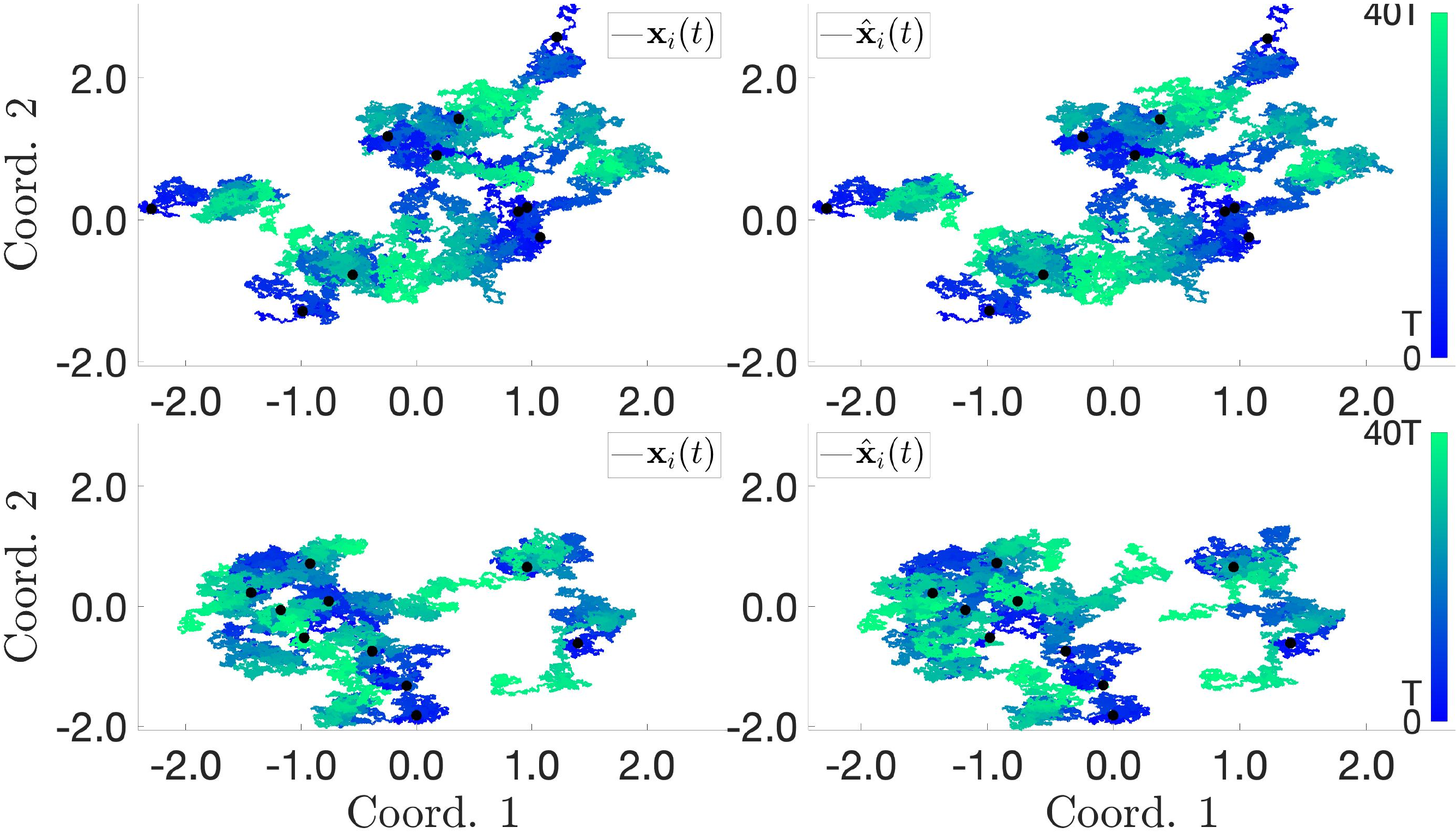}}
\caption{ \textmd{ (Stochastic Lennard Jones Dynamics) In each panel: true trajectory $\bX_t$ (Left column) and learned trajectory $\widehat \bX_t$ (Right column) obtained with  the true kernel $\intkernel$ and the estimated  kernel $\widehat \intkernel$   from $M=128$ and $1024$ trajectories, for an initial condition in the training data (Top row) and an initial condition randomly chosen (bottom row). The black dot at $t = 0.5$ divides the  ``training" interval $[0, 0.5]$ from the ``prediction" interval [0.5,20]. The trajectory prediction errors  are small in all cases. The statistics of the errors are presented in Table \ref{t:SLJH1_traj_err}. }
}\label{t:SLJH1_trajM32_err}
\end{figure}

\begin{table}[tbp]
\centering
\begin{tabular}{| c || c | c |} 
\hline
                                                             & $[0, 0.5]$                                                    & $[0.5, 20]$\\
\hline
$M=128, \sigma=0.05,\text{mean}_{\text{traj}}$: Training ICs & $3.1 \cdot10^{-2} \pm 8.3 \cdot10^{-3}$  & $3.0 \cdot10^{-1} \pm 3.9 \cdot10^{-1}$\\
\hline            
$M=128, \sigma=0.05,\text{mean}_{\text{traj}}$: Random ICs & $3.1 \cdot10^{-2} \pm 9.3\cdot10^{-3}$  & $3.1 \cdot10^{-1} \pm 4.2 \cdot10^{-1}$\\
\hline      
$M=128, \sigma=0.25,\text{mean}_{\text{traj}}$: Training ICs & $5.5 \cdot10^{-1} \pm 2.4 \cdot10^{-2}$  & $1.3 \cdot10^{0} \pm 7.5 \cdot10^{-1}$\\
\hline            
$M=128, \sigma=0.25,\text{mean}_{\text{traj}}$: Random ICs & $5.8 \cdot10^{-2} \pm 2.3\cdot10^{-2}$  & $1.3 \cdot10^{0} \pm 7.3 \cdot10^{-1}$\\
\hline     
\hline
$M=1024, \sigma=0.05,\text{mean}_{\text{traj}}$: Training ICs & $ 1.2\cdot10^{-2} \pm 3.4 \cdot10^{-3}$  & $ 1.7 \cdot10^{-1} \pm 2.7 \cdot10^{-1}$\\
\hline            
$M=1024,\sigma=0.05,\text{mean}_{\text{traj}}$: Random ICs & $1.2 \cdot10^{-2} \pm 3.6 \cdot10^{-3}$  &$ 1.5  \cdot10^{-1} \pm 2.5 \cdot10^{-1}$\\
\hline      
$M=1024, \sigma=0.25,\text{mean}_{\text{traj}}$: Training ICs & $ 2.2 \cdot10^{-2} \pm 6.2 \cdot10^{-3}$  & $ 3.2 \cdot10^{-1} \pm 3.7  \cdot10^{-1}$\\
\hline            
$M=1024, \sigma=0.25,\text{mean}_{\text{traj}}$: Random ICs & $ 2.2 \cdot10^{-2} \pm 6.4 \cdot10^{-3}$  & $ 3.2 \cdot10^{-1} \pm 3.5  \cdot10^{-1}$\\
\hline    
\end{tabular}
\caption{ \textmd{(Stochastic Lennard Jones Dynamics) Trajectory Errors: ICs used in the training set (first two rows), new ICs randomly drawn from $\mu_0$ (second set of two rows). Means are taken over the same number of trajectories as in the training data set.}}
\label{t:SLJH1_traj_err}
\end{table}

\begin{figure}[tbp]
\centering
\subfigure{\label{fig:SLJH1_Convergence}\includegraphics[width=0.49\textwidth]{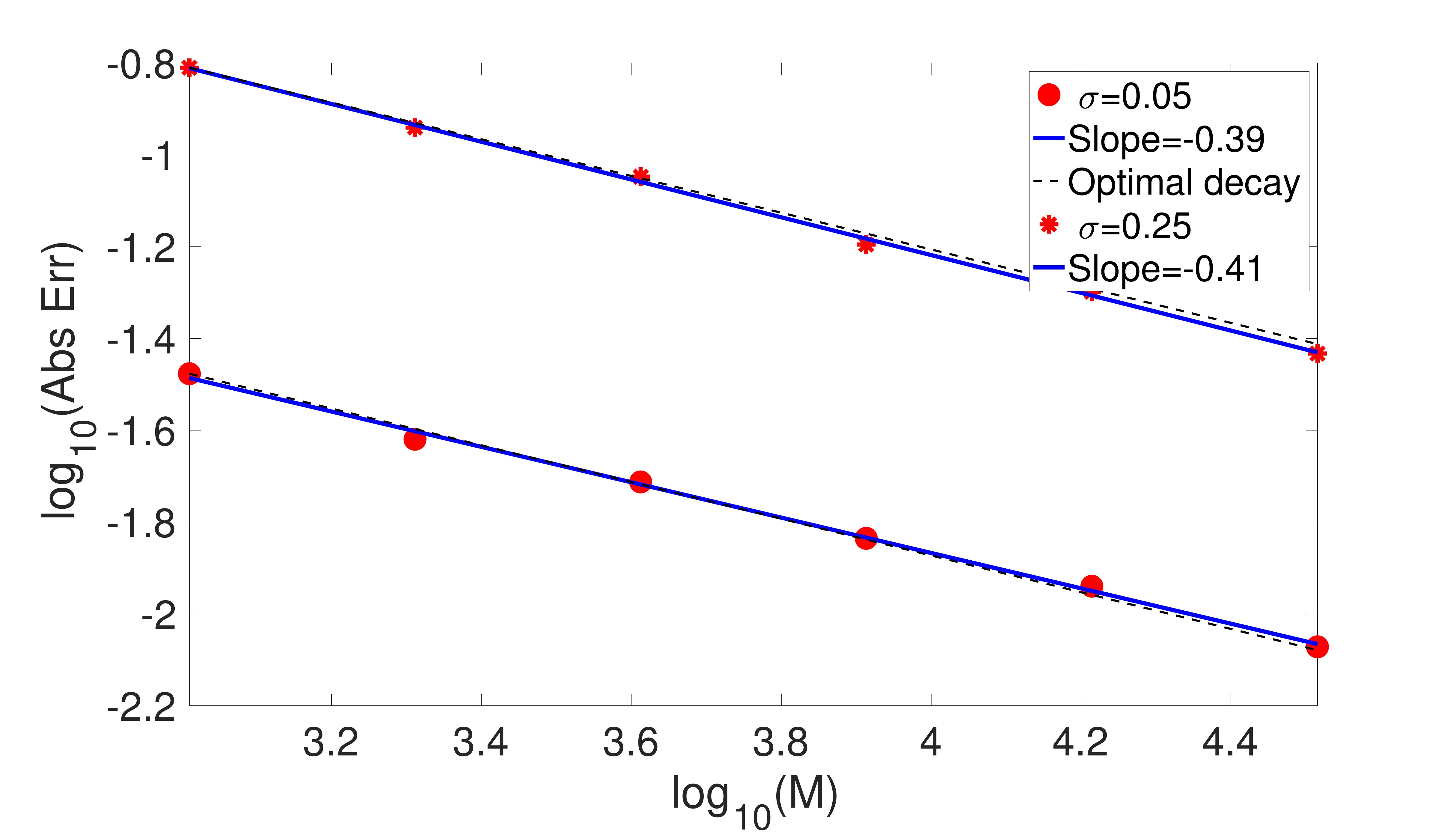}}
\subfigure{\label{fig:SLJH1_NoisevsErr}\includegraphics[width=0.49\textwidth]{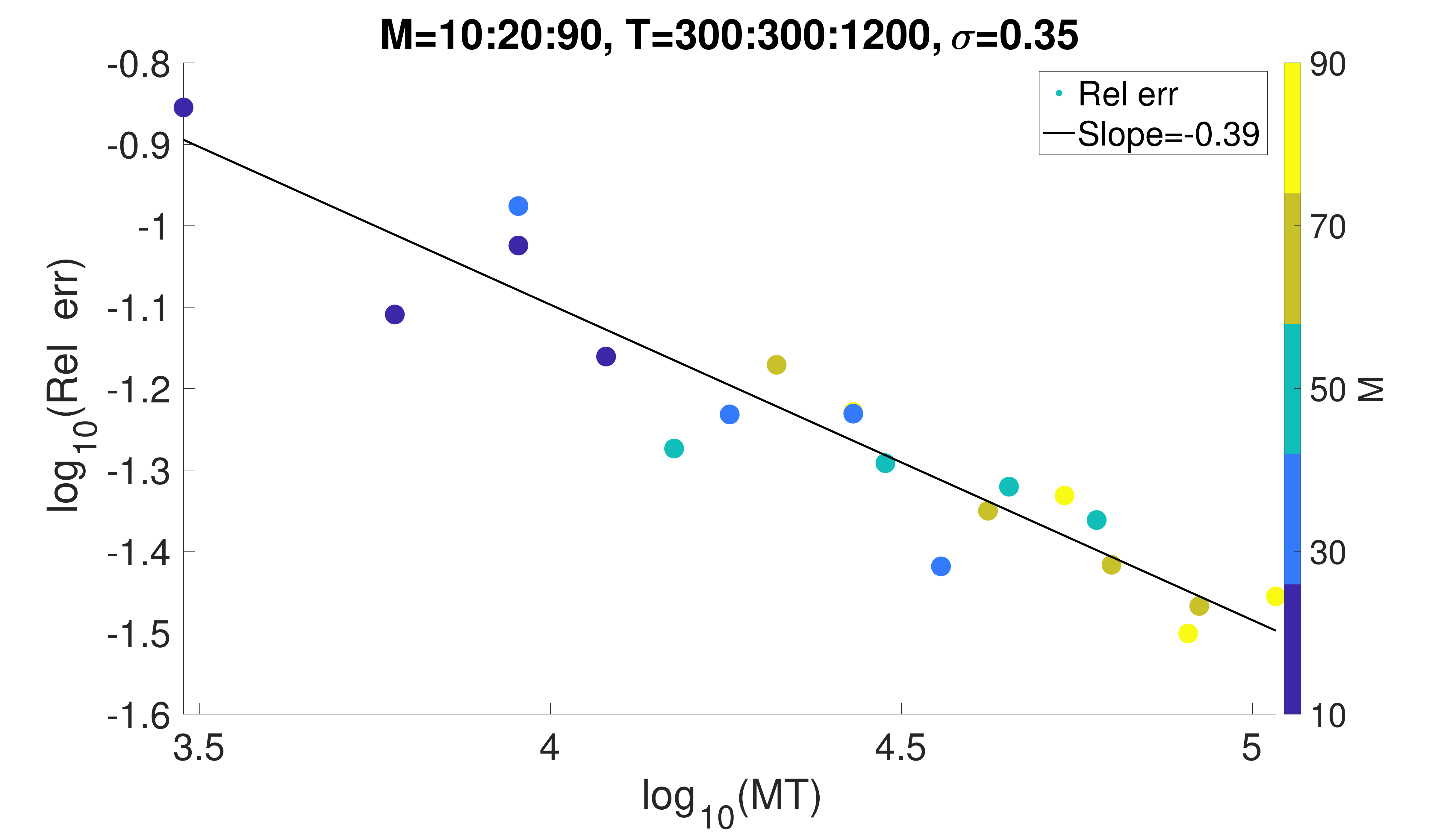}}
\caption{Stochastic Lennard-Jones: learning rates for continuous-time observations. Left: the learning rate of the estimators in terms of $M$ is $0.39$ when $\sigma = 0.05$ and is $0.41$ when $\sigma = 0.25$, close to the theoretical optimal min-max rate ${2}/{5}$ (shown in the black dot line).  Right: the convergence rate of the estimators in terms of $MT$, when both $M$ and $T$ increases. The colors of points are assigned  according to $M$. The rate is still close to the  optimal min-max rate ${2}/{5}$, showing the equivalence of learning from a single long trajectory with multiple short trajectories when the underlying process is ergodic.}
\label{f:LJH1_Convergence_Plot} 
\end{figure}

\begin{figure}[tbp]
\centering
\subfigure{\label{t:SLJH1_Obsgap}\includegraphics[width=0.49\textwidth]{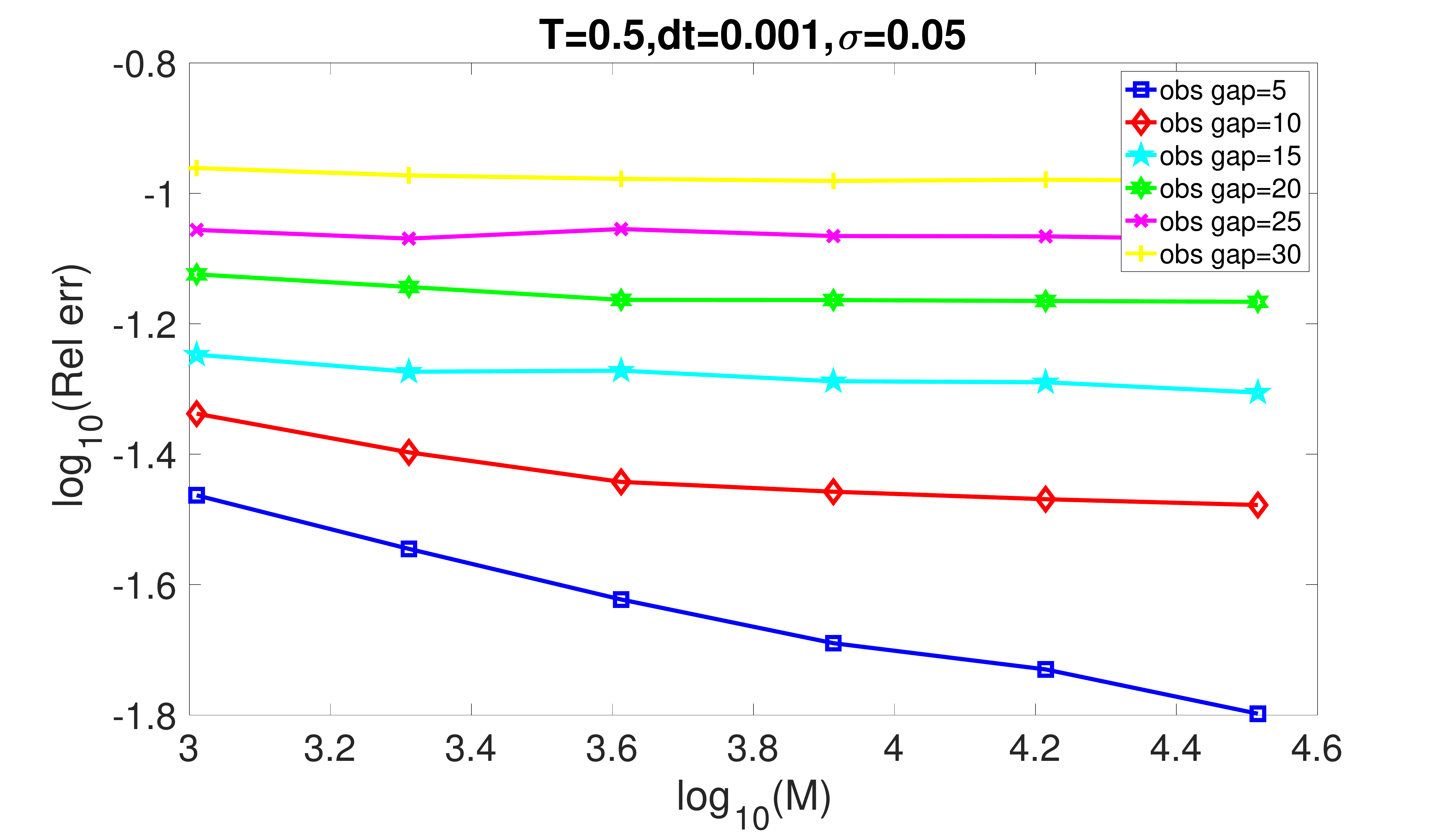}}
\subfigure{\label{fig:SLJH1_NoisevsErr}\includegraphics[width=0.49\textwidth]{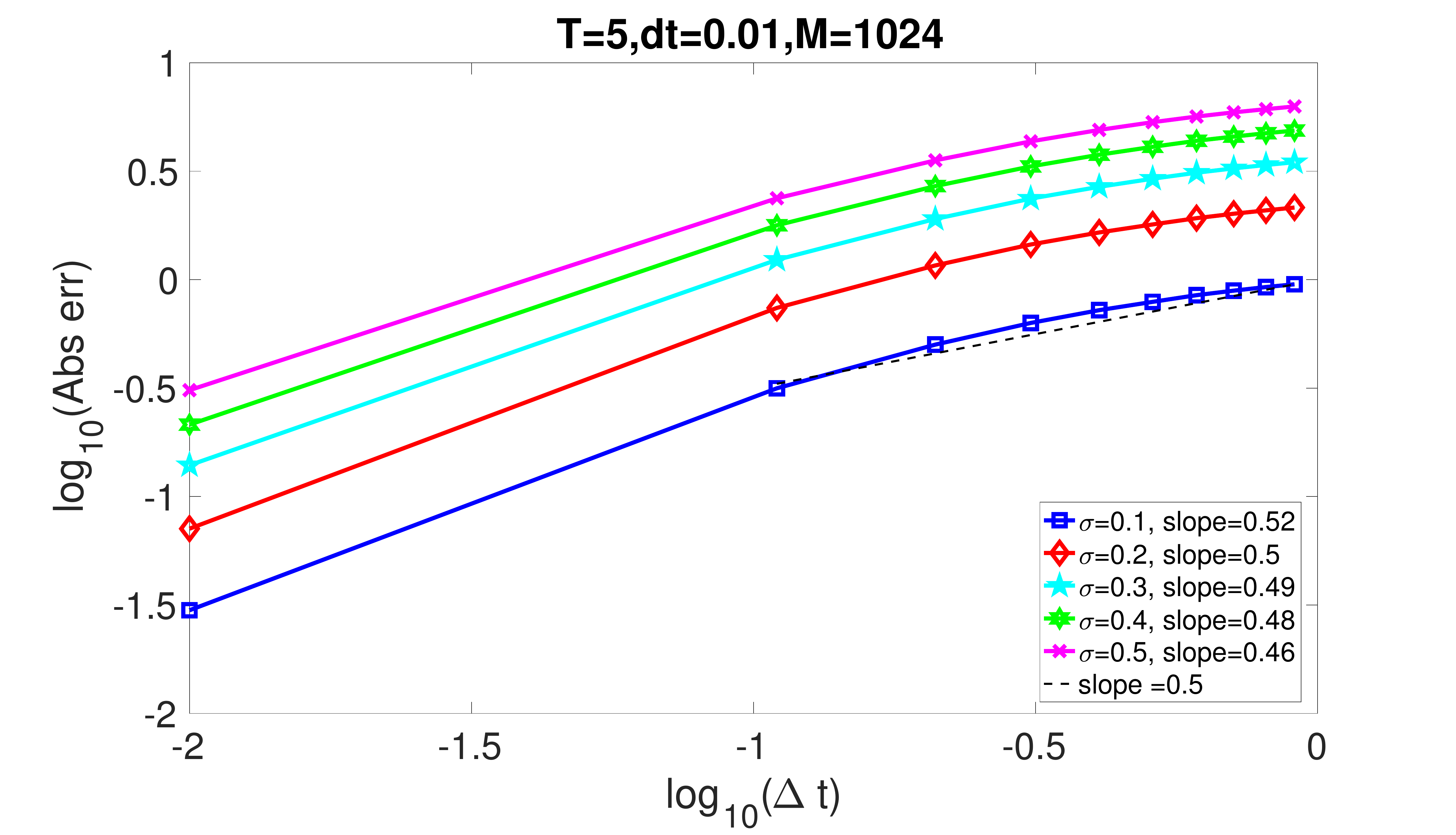}}
\caption{Stochastic Lennard-Jones: discretization error due to discrete-time observation. 
Left: The learning rate of estimators in terms of different observation gap  $\Delta t= k dt$ for $k=5:5:30$. The learning rate becomes flat, due to the bias induced by discretization of the likelihood function on coarse time grids. Right: the log-log plot of the absolute error of the estimator in terms of observation gap $\Delta t= k dt$ for $k=5:5:45$, for systems with different levels of random noises in terms of $\sigma$, computed with $M=1024$, $T=0.5$ and $dt = 0.001$ fixed. The orders of the absolute error in both $\sigma$ and $\Delta t$ are close to the theoretical order $\sigma O((\Delta t)^{1/2})$. The slopes of the lines are calculated using points whose $x$ coordinate fall in the range $[-1,0]$.
 }  \label{t:LJH1_Convergence_Plot} 
\end{figure}

Next, we  investigate the effects of the scale of the random noise on learning.  We observe phenomenon similar to those in Example 1.  Figure \ref{t:LJH1_kernel} shows that the estimators for the system with $\sigma=0.25$ oscillates more than the one with $\sigma=0.05$ at locations near 0. The random noise also did not affect the learning rates, suggested by the left plot of Figure \ref{f:LJH1_Convergence_Plot}.   as the random noise increases, absolute $L^2(\rhoT)$  error of estimators also increases, suggesting that coercivity constant is getting smaller.

At last, we study the effects of discretization error induced by discrete observations.  As the observation gap increases, the discretization errors flatten the learning rate curve of $M$, see left plot of Figure \ref{f:LJH1_Convergence_Plot}.  Similar to Example 1, the right plot of  Figure \ref{f:LJH1_Convergence_Plot} shows that the absolute error of the estimator is of order close to the theoretical order $\sigma O((\Delta t)^{1/2})$.

\subsection{Conclusions from the numerical experiments} 
\label{s:expconclusions}
Numerical results show that in case of continuous-time observations, the algorithm effectively estimates the interaction kernel, achieves the near-optimal learning rate in $M$, is robust to different magnitudes of the random noise, and the system with the estimated kernels accurately predicts trajectories. In case of discrete-time observations, the estimator has an estimation error  of order $\Delta t^{1/2}$, due to the discretization error in the approximation of the likelihood ratio. These numerical results are in full agreement with the learning theory in Section \ref{sec:theory_cts_traj}--\ref{sec:theory_dis_traj}:
\begin{itemize}
\item In case of continuous-time observations, the estimators in 10 trials are faithful approximations of the true interaction kernels, with a mean close to the truth. The standard deviation of the estimators decreases as the sample size increases, and gets larger as the diffusion constant increases.  

\item The estimator from data achieves the min-max learning rate   $(\log{M}/ M)^{s/(2s+1)}$ in Theorem \ref{maintheorem} by the appropriate choice of the hypothesis spaces and their dimension as a function of $M$. For $\phi$ in $C^{k+\alpha}$ with $k+\alpha\ge2$, the learning rate is around $M^{-\frac{1}{3}}$ when using piece-wise constant estimators ($s=1$); and the learning rate is around  $M^{-\frac{2}{5}}$ using the piecewise linear estimators ($s=2$), which is the minmax optimal rate for the case $k+\alpha=2$. 

\item The estimators predict transient dynamics well in the training time interval, and the results validate Proposition \ref{Trajdiff}: the trajectory discrepancy is controlled by $L^2({\rhoT})$ error of estimators, demonstrating the effectiveness of distances in $L^2(\rhoT)$ in quantifying the performance of estimators.  In addition, the estimators even predict in a remarkably accurate fashion the collective behaviour of particles in larger future time intervals, indicating that the bound in Proposition \ref{Trajdiff} may be overly pessimistic in some cases. Our intuition is that this benign phenomenon benefits from the large support of $\rhoT$, encouraged by the randomness of the initial conditions and presence of stochastic noise.

\item In case of discrete-time observations with observation gap $\Delta t$, the estimation error of the estimator is of order $\Delta t^{1/2}$ and depends linearly on {$\sigma$, the square root of the diffusion constant.}
 Therefore, as $\Delta t$ increases, the discretization error dominates the estimation error, consistently with the learning theory in Section \ref{sec:theory_dis_traj}, which leads to bounding the esitmation error of the estimator by $M^{-\frac{s}{2s+1}}+\sigma\mathcal{O}(\Delta t^{1/2})$. 

\item When the length $T$ of the trajectories increases, the optimal learning rate (in $M$) is still achieved. The estimation errors of the estimator exhibits a convergence rate around  $(\frac{\log(MT)}{MT})^{s/(2s+1)}$ with $s=1,2$ respectively, demonstrating an equivalence of ``information'' between few long trajectories and many short trajectories initiated at suitably random initial conditions, as discussed above in Section \ref{s:hypspacesnonparamestimators}.

\end{itemize}

\section{Final remarks and future work}
There are many venues in which the present work could be extended. 

The first notable extension is to heterogeneous particle systems with multiple types of particles, which arise in many applications.  In this case one assumes that there are different interaction kernels, modeling the non-symmetric interactions between different types of particles. Examples of these systems are considered in \cite{LZTM19} in the deterministic case, with the theoretical analysis achieved in \cite{LMT19}, where the coercivity condition is generalized to the multiple-particle-types setting, and (near-)optimal convergence rates of the estimators where established. We believe a similar extension is possible in the stochastic case, combining the ideas of this work and \cite{LMT19}.

Another notable extension is to second order differential systems of interacting particles, where interaction kernels of more general forms than those considered here arise. In the deterministic case \cite{LZTM19} considers examples of such systems, with a forthcoming theoretical analysis. In the stochastic case the extension would require significant effort, especially if important cases of systems with degenerate diffusion (e.g. stochastic Langevin) were considered.
We also remark again that in this work we do not observe velocities, as done in the works just cited in the case of deterministic systems: here we fully take into account the discretization (in time) error, and if we let $\sigma\rightarrow0$, the results here would imply similar results in the deterministic case. Extending these considerations to second-order systems would be valuable.

Further work is also needed to formalize the considerations we put forward in Section \ref{s:hypspacesnonparamestimators} regarding ergodic systems, and design robust and optimal algorithms in the regimes of observation a long trajectory or many independent trajectories. 

We assume in this work that all particles are observed. A desirable extension is to the case of partial observations of a subset of particles or macroscopic observations of the population density, which is a practical concern when the system is large with millions of particles in high dimension. Since it is an ill-posed inverse problem to recover the missing trajectories of unobserved particles \cite{zhang2020cluster}, a new formulation based on the corresponding mean field equations  \cite{MT2014,jabin2016_MeanField,jabin2018_QuantitativeEstimates} is under investigation. 

In this work we assume that the noise coefficient is a known constant: there has been of course significant work in estimating the noise coefficient, for example in the case of interacting particle systems see the recent work \cite{huang2018_LearningInteracting} and references therein, and for the case of model reduction for Langevin equations with state-dependent diffusion coefficient \cite{CM:ATLAS}.

\appendix

\section{Appendix}\label{sec:Appendix}

\subsection{Preliminaries for SDEs}
Let $\bX_t$ satisfy 
\begin{align}\label{eq:sde}
d\mbf{X}_{t}=V(\mbf{X}_{t}, t) d\tau +\sigma(\mbf{X}_{t},t) d\mbf{B}_{t}.
\end{align}
We first review the existence and uniqueness of strong solutions for SDEs (see Theorem 5.4 in \cite{klebaner2005introduction})

\begin{theorem}[Existence and Uniqueness] If the following conditions are satisfied
\begin{itemize}
\item Coefficients are locally Lipschitz in $\bX$ uniformly in $t$, that is for every $T$ and $K$, there is a constant $C$ depend only on $T$ and $K$ such that for all $\|\bX\|, \|\bY\| \leq K$ and all $0\leq t \leq T$
\begin{align}\label{sde:lip}
\|V(\bX,t)-V(\bY,t)\|+\|\sigma(\bX,t)-\sigma(\bY,t)\|<C\|\bX-\bY\|.
\end{align}
\item Coefficients satisfy the linear growth condition 
\begin{align}\label{sde:lineargrowth}
\|V(\bX,t)\|+\|\sigma(\bX,t)\|<C(1+\|\bX\|).
\end{align}
\item $\bX(0)$ is independent of $(\bB(t),0\leq t\leq T),$ and $\mathbb{E}\|\bX(0)\|^2<\infty$.
Then there exists a unique strong solution $\bX_t$ of the SDEs \eqref{eq:sde}. $\bX_t$ has continuous paths, moreover,
$$\mathbb{E}[\sup_{0\leq t\leq T}\|\bX_t\|^2]<C_1(1+\mathbb{E}[\|\bX(0)\|^2]),$$ 
\end{itemize}
where constants $C_1$ depend only on $C$ and $T$.
\end{theorem}

 It is straightforward to show that $\rhsfo_{\intkernel}$ satisfy \eqref{sde:lip} and  \eqref{sde:lineargrowth}. Therefore, suppose $\mu_0$ is independent of the underlying Brownian motion and has finite second moment, then there exists a unique strong solution up to time $T$ for the system \eqref{eq:sod} for any $\bX_0$ drawn from $\mu_0$. 

\begin{theorem}[Girsonov Theorem]\label{Girsonov}
Let $P_{\sigma}$ be the probability measure induced by the solution of the  SDEs \eqref{eq:sde} for $t \in [T_0, T]$ and a fixed starting value at time $T_0$, and let $W_{\sigma}$ be the law of the respective driftless process. Suppose that $\Sigma=\sigma \sigma'$ is invertible and $V$ fulfills the Novikov condition
$$\mathbb{E}_{P_{\sigma}}\bigg[\mathrm{exp} \bigg (\frac{1}{2} \int_{T_0}^T \|V(\mbf{X}_{t}, t)\|^2 dt \bigg)\bigg] < \infty. $$

Then $P_{\sigma}$ and $W_{\sigma}$ are equivalent measures with Radon-Nikodym derivative given by Girsonov's formula
$$\frac{dP_{\sigma}}{dW_{\sigma}} \big(\mbf{X}_{[T_0, s]}\big)=\mathrm{exp}\bigg(\int_{T_0}^{s} V^T \Sigma^{-1} d\mbf{X}_{t}-\frac{1}{2} \int_{T_0}^{s} V^T\Sigma^{-1} V d t \bigg)$$

for all $s \in [T_0, t]$ and $\mbf{X}_{[T_0,s]}=(\mbf{X}_{t})_{t \in [T_0, s]}.$
\end{theorem}

The proof of Theorem \ref{Girsonov} can be found in  \cite[Chapter 3.5]{karatzas1998brownian},\cite[Chapter 8.6]{oksendal2013sde}.

\begin{theorem}[The It\^o formula, see Theorem 4.1.2 in \cite{oksendal2013sde}]\label{Itoformula} Let $g(\bX)=(g_1(\bX),\cdots, g_p(\bX))$ be a $C^{2}$ map from $ \R^{dN}$ into $\R^p$. Then the process 
$$Y(t)=g(\bX_t)$$
is  an It\^o process with components given by 
$$dY_k=
\sum_{i=1}^{Nd} \frac{\partial g_k}{\partial x_i}(\bX_t)d\bX_i+\frac{1}{2}\sum_{i,j} \frac{\partial^2 g_k}{\partial x_i\partial x_j}(\bX_t)d\bX_id\bX_j$$ where $d\bB_id\bB_j=\delta_{ij}dt, d\bB_idt=dtd\bB_i=0$.
\end{theorem}

\subsection{Useful  inequalities}

\begin{theorem}[Bernstein inequality for unbounded random variables] \label{bernstein:unbound} Let $X_1, X_2,\cdots,X_M$ be independent random variables with $\mathbb{E}(X_i)=0$.  If for some constants $K_1, v_1>0$, the bound $\mathbb{E}|X_i|^p \leq \frac{1}{2} p! K_1^{p-2}v_1$ holds for every $2 \leq p \in \mathbb{N}$, then 
\begin{align}
\mathrm{Prob} \bigg\{ \sum_{i=1}^{M} X_i \geq \epsilon\bigg\}  \leq e^ {-\frac{\epsilon^2}{2}(Mv_1+K_1\epsilon)^{-1}}.
\end{align}
\end{theorem}

 For the proof of  Theorem \ref{bernstein:unbound} , we refer to \cite{bennett1962probability} and David Pollard's book notes \cite{Pollard2000}(page 14).

\begin{corollary}\label{berstein1} Denote $\mathcal{E}_{M}(g)=\frac{1}{M} \sum_{m=1}^{M}g(X_m)$ for a measurable function $g$. If for some $K_2, v_2 >0$, the bound $$\mathbb{E}| g-\mathbb{E}g|^p \leq \frac{1}{2} p!K_2^{p-2}v_2$$ holds for $2\leq p \in \mathbb{N}$, then there holds
\begin{align} \label{berstein:unbound2}
\mathrm{Prob}\bigg \{\mathbb{E}g-\mathcal{E}_{M}(g) \geq \epsilon \bigg\}  \leq e^ {-\frac{M \epsilon^2}{2(v_2+K_2\epsilon)}}, \forall \epsilon>0.
\end{align}
\end{corollary}
\begin{proof}
Applying Theorem  \ref{bernstein:unbound} on the random variable $\mathbb{E}g-g$, we immediately obtain the desired bound. 
\end{proof}

\begin{corollary}\label{berstein2}  If for some $K_3, v_3>0$, the bound 
$$\mathbb{E}|g-\mathbb{E}g|^p \leq \frac{1}{2} p! K_3^{p-2} v_3|\mathbb{E}g|$$ holds for $2\leq p \in \mathbb{N}$, then 
$$\mathrm{Prob}( \mathbb{E}g-\mathbb{E}_{M}(g) \geq \sqrt{\epsilon} \sqrt{\epsilon+|\mathbb{E}g|}) \leq e^ { -\frac{M\epsilon}{2(v_3+K_3)}}, \forall \epsilon>0$$
\end{corollary}

\begin{proof}

If we replace $\epsilon$ with $\sqrt{\epsilon(\epsilon+|\mathbb{E}g|)}$ in \eqref{berstein:unbound2}, and let $K_2=K_3$, $v_2=v_3|\mathbb{E}g|$, the desired bound follows from the inequality 
\begin{align*}
  e^ {-\frac{M \epsilon(\epsilon+|\mathbb{E}g|)} {2(v_2+K_2\sqrt{\epsilon(\epsilon+|\mathbb{E}g|)}})}& \leq e^ { -\frac{M\epsilon}{2(v_3+K_3)}}\\ \Leftrightarrow v_3\epsilon+K_3(\epsilon+|\E g|) &\geq K_3\sqrt{\epsilon(\epsilon+|\mathbb{E}g|)},
\end{align*} where the last inequality is true since $\sqrt{\epsilon(\epsilon+|\mathbb{E}g|)}\leq \epsilon+|\mathbb{E}g|$ for all $\epsilon\geq 0$. 
\end{proof}

We also refer to the reader  \cite{wang2011optimal} (see its Lemma 3 and  Lemma 5) for the analog of Corollary \ref{berstein1} and \ref{berstein2}.

\begin{theorem}[Moment inequality for stochastic integrals, see Theorem 7.1 in \cite{mao2007stochastic}]\label{momentinequality} Let $\mathcal{M}^2([0, T]; \mathbb{R}^{n \times m})$ denote the family of all  $n\times m$-matrix-valued measurable $\{\mathcal{F}_t\}_{t\geq t_0}$ -adapted process $f=\{(f_{ij}(t))_{n \times m}\}_{0\leq t\leq T}$ such that
$\mathbb{E}\int_{0}^{T}\|f(t)\|^2dt <\infty.$ If $p \geq 2$, $ f \in \mathcal{M}^2([0, T]; \mathbb{R}^{n \times m} )$ such that  
$$\mathbb{E} \int_{0}^{T}\|f(t)\|^p dt < \infty, $$
then
$$\mathbb{E} \bigg\| \int_{0}^{T} f(s)d\bB(s) \bigg\|^p \leq (\frac{p(p-1)}{2})^{\frac{p}{2}} T^{\frac{p-2}{2}}\mathbb{E}\int_{0}^{T}\|f(s)\|^p ds$$ In particular, for $p=2$, there is equality. 
\end{theorem}

\subsection{Proof of Proposition \ref{Trajdiff}}
\label{s:appendixproofs}

\begin{proof}[of Proposition \ref{Trajdiff}] For ease of notation, in this proof we use $\mathbb{E}$ to represent 
$\E_{\mu_0,\bB}$.  For every $t \in [0,T]$, we have
\begin{align*}
&\mathbb{E}\left[\|\bX_t-\widehat \bX_t\|^2 \right]
=\mathbb{E}\left[\left\| \int_{0}^{t} \rhsfo_{\intkernel}({\bX}(s))- \rhsfo_{\widehat\intkernel}(\widehat{\bX}(s)) ds\right\|^2 \right]\\
 &\leq  \, t  \mathbb{E} \left[\int_{0}^{t} \left\| \rhsfo_{\intkernel}({\bX}(s))- \rhsfo_{\widehat\intkernel}(\widehat{\bX}(s))\right\|^2 ds\right] \\
 & \leq  \, 2T\mathbb{E}\left[\int_0^t  \left\|\rhsfo_{\intkernel}(\bX(s))-\rhsfo_{\widehat\intkernel}(\bX(s))\right\|^2 ds\right]  + 2T\mathbb{E}\left[\int_0^t  \left\| \rhsfo_{\widehat\intkernel}(\bX(s)) - \rhsfo_{\widehat\intkernel}(\widehat\bX(s)) \right\|^2 ds\right]\,. 
\end{align*}
Letting $ \bx_{ji}(s) := \bx_j(s)- \bx_i(s)$, $ \widehat \bx_{ji}(s) := \widehat \bx_j(s)-\widehat \bx_i(s)$, and $F_{\intkernelvar}(\bx)=\intkernelvar(\|\bx\|)\bx$, for $\intkernelvar \in \mathcal{K}_{R,S}$ and $\bx \in \mathbb{R}^d$,
\begin{align*}
 \left\| \rhsfo_{\widehat\intkernel}(\bX(s)) - \rhsfo_{\widehat \intkernel}(\widehat\bX(s)) \right\|^2 &= \sum_{i=1}^{N}\Bigg\| \frac{1}{N}\sum_{j=1}^{N}\big(F_{[\widehat\intkernel]}(\bx_{ji}(s))-F_{[\widehat \intkernel]}(\widehat \bx_{ji}(s))\big)\bigg\|^2 \\
  &\leq  4 \text{Lip}^2(F_{[\widehat\intkernel]})\left\| \bX(s)-\widehat \bX(s)\right\|^2\,,\quad \text{almost surely}.
\end{align*}
Then an application of Gronwall's inequality yields the estimate
\begin{align*}
\mathbb{E}\left[\left\|\bX_t-\widehat \bX_t \right\|^2\right] &\leq  2T e^{8T^2 \text{Lip}^2(F_{[\widehat{\intkernel}]})}\mathbb{E}\left[\int_{0}^{T}   \left\|\rhsfo_{\intkernel}(\bX(s))-\rhsfo_{\widehat\intkernel}(\bX(s)) \right\|^2ds\right].
\end{align*} 
Note that by Jensen's inequality,  
 \begin{align*}
\frac{1}{T}\int_{0}^{T}  \E \left[ \left\| \rhsfo_{\intkernel}(\bX(s))-\rhsfo_{\widehat\intkernel}(\bX(s)) \right\|^2\right]ds < N \vertiii{\widehat{\intkernel}-\intkernel}^2.
\end{align*}Then the conclusion follows by combining with the estimate
$\text{Lip}(F_{[\widehat{\intkernel}]})\leq (R+1)S.$
\end{proof}

\bibliographystyle{abbrv}
\bibliography{learning_dynamics,ref_FeiLU,LearningTheory,ref_stochasticPS,OpinionDS_ref,SDE_Inference}

\end{document}